\documentclass[10pt]{article}
\usepackage[margin=1in]{geometry}

\usepackage{amsmath,amsthm,amssymb}
\usepackage{commath,mathtools}
\usepackage{algorithm,algorithmic}
\usepackage{graphicx}
\usepackage[labelformat=simple]{subcaption}

\usepackage{hyperref}
\usepackage{color}
\usepackage{diagbox}

\newtheorem{theorem}{Theorem}

\DeclareMathOperator*{\argmin}{\mathrm{arg\,min}}

\newcommand{\Au}{\mathcal{A}[u]}
\newcommand{\ut}{u_{\theta}}
\newcommand{\Aut}{\mathcal{A}[u_{\theta}]}
\newcommand{\phieta}{\varphi_{\eta}}
\newcommand{\veta}{v_{\eta}}

\newcommand{\Lint}{L_{\text{int}}}
\newcommand{\Lbdry}{L_{\text{bdry}}}
\newcommand{\Linit}{L_{\text{init}}}

\newcommand{\Bu}{\mathcal{B}[u]}
\newcommand{\Rbb}{\mathbb{R}}

\usepackage[pagewise,mathlines]{lineno}
\title{Weak Adversarial Networks for High-dimensional Partial Differential Equations}
\author{Yaohua Zang
\footnote{School of Mathematical Sciences, Zhejiang University, Hangzhou, Zhejiang 310007, China. Email: \texttt{11535015@zju.edu.cn}.}
\and
Gang Bao
\footnote{School of Mathematical Sciences, Zhejiang University, Hangzhou, Zhejiang 310007, China. Email: \texttt{baog@zju.edu.cn}.}
\and
Xiaojing Ye
\footnote{Department of Mathematics and Statistics, Georgia State University, Atlanta, GA 30303, USA.
Email: \texttt{xye@gsu.edu}.}
\and
Haomin Zhou
\footnote{School of Mathematics, Georgia Institute of Technology, Atlanta, GA 30332, USA.
Email: \texttt{hmzhou@math.gatech.edu}.}
}
\date{}

\begin{document}
\maketitle

\begin{abstract}	
Solving general high-dimensional partial differential equations (PDE) is a long-standing challenge in numerical mathematics. In this paper, we propose a novel approach to solve high-dimensional linear and nonlinear PDEs defined on arbitrary domains by leveraging their weak formulations. We convert the problem of finding the weak solution of PDEs into an operator norm minimization problem induced from the weak formulation. The weak solution and the test function in the weak formulation are then parameterized as the primal and adversarial networks respectively, which are alternately updated to approximate the optimal network parameter setting. Our approach, termed as the weak adversarial network (WAN), is fast, stable, and completely mesh-free, which is particularly suitable for high-dimensional PDEs defined on irregular domains where the classical numerical methods based on finite differences and finite elements suffer the issues of slow computation, instability and the curse of dimensionality. We apply our method to a variety of test problems with high-dimensional PDEs to demonstrate its promising performance.
\bigskip

\smallskip

%\noindent
%\textbf{AMS subject classifications.} 90C25, 68U10, 65J22
\end{abstract}

\section{Introduction}
\label{sec:intro}
Solving general high-dimensional partial differential equations (PDEs) has been a long-standing challenge in numerical analysis and computation \cite{rudd2015constrained,lagaris1998artificial,raissi2019physics,raissi2017physics,
beck2017machine,fujii2017asymptotic,han2017overcoming,weinan2017deep,khoo2019solving,nabian2018deep,weinan2018deep,sirignano2018dgm}. In this paper, we
present a novel method that leverages the form of weak solutions and adversarial networks to compute solutions of PDEs, especially to tackle problems posed in
high dimensions.
To instantiate the derivation of the proposed method, we first consider the following second-order elliptic PDE with either Dirichlet's or Neumann's boundary conditions on \textit{arbitrary domain} $\Omega \subset \mathbb{R}^d$,
\begin{equation}
\label{eq:soe}
\begin{cases}
-\sum_{i=1}^d \partial_{i} (\sum_{j=1}^d a_{ij}\partial_{j}u) + \sum_{i=1}^d b_i \partial_{i} u + cu - f = 0, & \mbox{in} \ \Omega \\
u(x) - g(x) = 0\quad \text{(Dirichlet)}\quad \text{or}\quad (\partial u / \partial \vec{n})(x) - g(x) = 0\quad \text{(Neumann)}, & \mbox{on}\  \partial\Omega
\end{cases}
\end{equation}
where $a_{ij}, b_i, c: \Omega \to \mathbb{R}$ for $i,j \in [d]\triangleq \{1,\dots,d\}$, $f:\Omega \to \mathbb{R}$ and $g:\partial \Omega \to \mathbb{R}$ are all given, and $(\partial u / \partial \vec{n})(x) $ denotes the directional derivative of $u$ along the outer normal direction $\vec{n}$ at the boundary point $x\in \partial \Omega$.
In addition, we assume that the elliptic operator has a strong ellipticity, meaning there exists a constant $\theta>0$ such that $\xi^{\top}A(x)\xi \ge \theta|\xi|^2$ for any $\xi=(\xi_1,\dots,\xi_d)\in \Rbb^{d}$ with $|\xi|^2=\sum_{i=1}^d|\xi_i|^2$ and $x\in \Omega$ a.e., where $a_{ij}=a_{ji}$ for all $i,j\in[d]$ and $A(x) \triangleq [a_{ij}(x)]\in\mathbb{R}^{d\times d}$, i.e., $A(x)$ is symmetric positive definite with all eigenvalues no smaller than $\theta$ almost everywhere in $\Omega$.
%
%If we denote $\Omega\times[0,T]$ by $\Omega_{T}$, then
We also consider solving PDEs involving time, such as the linear second-order parabolic PDE (of finite time horizon):
\begin{equation}
\label{eq:parab}
\begin{cases}
u_t-\sum_{i=1}^d \partial_{i} (\sum_{j=1}^d a_{ij}\partial_{j}u) + \sum_{i=1}^d b_i \partial_{i} u + cu - f = 0, & \mbox{in} \ \Omega \times [0,T] \\
u(x,t) - g(x,t) = 0\quad \text{(Dirichlet)}\quad \text{or}\quad (\partial u / \partial \vec{n})(x,t) - g(x,t) = 0\quad \text{(Neumann)}, & \mbox{on}\  \partial\Omega\times[0,T] \\
u(x,0)-h(x)=0,& \mbox{on}\ \Omega
\end{cases}
\end{equation}
where $a_{ij}, b_i, c:\Omega\times[0,T]\to \mathbb{R}$ for $i,j\in [d]$ as before, $f:\Omega\times[0,T]\rightarrow \Rbb$ and $g: \partial \Omega\times[0,T]\rightarrow \Rbb$ and $h: \Omega\rightarrow \Rbb$ are given.
%Similarly, there exist a constant $\theta>0$ such that $\xi^{\top}A(x,t)\xi \ge \theta|\xi|^2$ for all $(x,t)\in\Omega_T$, $\xi\in\Rbb^d$, where $A(x,t)=[a_{ij}(x,t)]\in\Rbb^{d\times d}$ and $a_{ij}=a_{ji}, i,j\in[d]$.
%
%Then, we generalize the proposed method to PDEs involving time through a second-order parabolic PDE with initial boundary value conditions.
%
In either case, we will see that the method developed in this paper can be directly applied to \textit{general high-dimensional PDEs, including both linear and nonlinear ones}.

PDEs are prevalent and have extensive applications in science, engineering, economics, and finance \cite{Quarteroni:2008a,Thomas:2013a}.
The most popular standard approaches to calculate numerical solutions of PDEs include finite difference and finite element methods (FEM) \cite{Hughes:2012a}.
These methods discretize the time interval $[0,T]$ and the domain $\Omega$ using mesh grids or triangulations, create simple basis functions on the mesh, convert a continuous PDE %in \eqref{eq:general_PDE}
into its discrete counterpart, and finally solve the resulting system of basis coefficients to obtain numerical approximations of the true solution.
Although these methods have been significantly advanced in the past decades and are able to handle rather complicated and highly oscillating problems, they suffer the so-called ``curse of dimensionality'' since the number of mesh points increases exponentially fast with respect to the problem dimension $d$. Hence they quickly become computationally intractable for high dimensional problem in practice.
As a consequence, these numerical methods are rarely useful for general high-dimensional PDEs, e.g. $d\ge4$, especially when a sufficiently high-resolution solution is needed and/or the domain $\Omega$ is irregular. %for which these methods can be slow and unstable.

Facing the challenge, our goal is to provide a computational feasible alternative approach to solve general high-dimensional PDEs defined on arbitrarily shaped domains.
More specifically, using the weak formulation of PDEs, we parameterize the weak solution and the test function as the primal and adversarial neural networks respectively, and train them in an unsupervised form where only the evaluations of these networks (and their gradients) on some sampled collocation points in the interior and boundary of the domain are needed.
Our approach retains the continuum nature of PDEs for which partial derivatives can be carried out directly without any spatial discretization, and is fast and stable in solving general high-dimensional PDEs.
Moreover, our method is completely mesh-free and can be applied to PDEs defined on arbitrarily shaped domains, without suffering the issue of the curse of dimensionality.

In the remainder of this paper, we first provide an overview of related work on solving PDEs using machine learning approaches in Section \ref{sec:related}.
In Section \ref{sec:proposed},  we introduce the weak formulation of stationary PDEs, reformulate the PDE as a saddle-point problem based on the operator norm induced from the weak formulation, and present our proposed algorithm with necessary training details.
Then we extend the method to solve PDEs involving time.
In Section \ref{sec:experiment}, we provide a number of numerical results to show that our method can solve high-dimensional PDEs efficiently and accurately. Our examples also demonstrate some numerical understandings about selections of neural networks structures including the numbers of layers and nodes in the computations.
Section \ref{sec:summary} concludes this paper.
%%%%%%%%%%%%%%%%%%%%%%%%%%%%%%%%%%%%%%%%%%%%%%%%%%%%%%%%%%%%%%%%%%%%%%%%%%%%%%%%%%
\section{Related Work}
\label{sec:related}
Deep learning techniques have been used to solve PDEs in the past few years. %Many methods use innovative ideas achieving remarkable success that can't be done by the classical methods.
As an emergent research direction, great potentials %of using deep learning techniques in solving PDEs
have been demonstrated by many promising results, even though many fundamental questions remain to be answered.  Based on the strategies,
these works can be roughly classified into two categories.

In the first category, deep neural networks (DNN) are employed to assist the classical numerical methods.
In \cite{lee1990neural}, parallel neural networks are used to improve the efficiency of the finite difference method.
In \cite{wang1990structured}, neural network is used to accelerate the numerical methods for matrix algebra problems.
Neural network is also applied to improve the accuracy of finite difference method in \cite{gobovic1994analog}, which can be extended to solve two-dimensional PDEs \cite{yentis1996vlsi}.
In \cite{meade1994numerical,meade1994solution}, the solution of ordinary differential equations (ODE) is approximated by the combination of splines, where the combination parameters are determined by training a neural network with piecewise linear activation functions.
A constrained integration method called GINT is proposed to solving initial boundary value PDEs in \cite{rudd2015constrained}, where neural networks are combined with the classical Galerkin method.
In \cite{tompson2017accelerating}, convolutional neural networks (CNN) is used to solve the large linear system derived from the discretization of incompressible Euler equations. %
In \cite{suzuki2017neural}, a neural network-based discretization scheme is developed for the nonlinear differential equation using regression analysis technique.
Despite of the improvement over classical numerical methods, these methods still suffer the exponentially increasing problem size and are not tractable for high-dimensional PDEs.

In the second category, the deep neural networks are employed to directly approximate the solution of PDE, which may be more advantageous in dealing with high dimensional problems.
In \cite{lagaris1998artificial}, the solution of the PDE is decomposed into two parts, where the first part is explicitly defined to satisfy the initial boundary conditions and the other part is a product of a mapping parameterized as a neural network and an explicitly defined function that vanishes on the boundary.
Then the neural network is trained by minimizing the squared residuals over specified collocation points.
 An improvement of this method by parameterizing both parts using neural networks in \cite{berg2018unified}.
The singular canonical correlation analysis (SVCCA) is introduced to further improve this method in \cite{magill2018neural}.
In contrast to decomposing the solution into two parts, the idea of approximating the solution of PDEs by a single neural network is considered in \cite{dissanayake1994neural}, which is not capable of dealing with high-dimensional problems.
In \cite{raissi2019physics,raissi2017physics}, physics-informed neural networks (PINN) are proposed to approximate the solution of PDEs by incorporating observed data points and initial boundary conditions into the loss function for training.
A similar model is present in \cite{sirignano2018dgm} for high dimensional free boundary parabolic PDEs.
A different approach that represents a class of nonlinear PDEs by forward-backward stochastic differential equations is proposed and studied In \cite{beck2017machine,fujii2017asymptotic,han2017overcoming,weinan2017deep}.
Another appealing approach that exploits the variational form of PDEs is considered in \cite{khoo2019solving,nabian2018deep,weinan2018deep}.
In \cite{khoo2019solving}, a committer function is parameterized by a neural network whose weights are obtained by optimizing the variational formulation of the corresponding PDE.
In \cite{nabian2018deep}, deep learning technique is employed to solve low-dimensional random PDEs based on both strong form and variational form.
In \cite{weinan2018deep}, a deep Ritz method (DRM) is proposed to solve the class of PDEs that can be reformulated as equivalent energy minimization problems. The constraint due to boundary condition is added to the energy as a penalty term in \cite{weinan2018deep}.
More recently, an adaptive collocation strategy is presented for a method in \cite{anitescu2019artificial}.
In \cite{yang2019adversarial}, an adversarial inference procedure is used for quantifying and propagating uncertainty in systems governed by non-linear differential equations, where the discriminator distinguishes the real observation and the approximation provided by the generative network through the given physical laws expressed by PDEs and the generator tries to fool the discriminator.
To the best of our knowledge, none of the existing methods models the solution and test function in the weak solution form of the PDE as primal and adversarial networks as proposed in the present work.
We will show in the experiment section that the use of weak form  is more advantageous especially when the PDEs have singularities where classical solutions do not exist.
%%%%%%%%%%%%%%%%%%%%%%%%%%%%%%%%%%%%%%%%%%%%%%%%%%%%%%%%%%%%%%%%%%%%%%%%%%%%%%%%%%
\section{Proposed Method}
\label{sec:proposed}
To demonstrate the main idea, we first focus on the boundary value problems (BVP) \eqref{eq:soe}.
We consider the weak formulation of the PDE, and pose the weak solution as an operator norm minimization.
The weak solution and the test function are both parameterized as deep neural networks, where the parameters are learned by an adversarial training governed by the weak formulation.
Important implementation details are also provided.
Finally, we extend the proposed method to the IBVP \eqref{eq:parab}, where the PDEs are time-dependent.

%--------------------------------------------------------------------------------
\subsection{PDE and weak formulation}
In general, a solution $u\in C^2(\Omega)$ of a BVP \eqref{eq:soe} requires sufficient regularity of the problem and may not exist in the classical sense.
Instead, we consider the \textit{weak formulation} of \eqref{eq:soe} by multiplying both sides by a test function $\varphi \in H_0^1(\Omega;\mathbb{R})$ and integrating by parts:
\begin{equation}
\label{eq:weak_soe}
\begin{cases}
\langle \Au, \varphi \rangle \triangleq \int_\Omega \del[1]{\sum_{j=1}^d \sum_{i=1}^d a_{ij}\partial_j u\partial_{i}\varphi + \sum_{i=1}^d b_i \varphi\partial_{i}u + cu \varphi - f\varphi } \dif x= 0 & \\
\Bu = 0, \quad \mbox{on}\  \partial\Omega &
\end{cases}
\end{equation}
where $H_0^1(\Omega;\mathbb{R})$ denotes the Sobolev space,  a Hilbert space of functions who themselves and their weak partial derivatives are $L^2$ integrable on $\Omega$ with vanishing trace on the boundary $\partial \Omega$.
Note that the boundary terms of \eqref{eq:weak_soe} after integration by parts disappears due to $\varphi=0$ on $\partial \Omega$.
If $u \in H^1(\Omega;\mathbb{R})$ with possibly nonzero trace satisfies \eqref{eq:weak_soe} for all $\varphi \in H_0^1$, we say that $u$ is a \textit{weak solution} (or \textit{general solution}) of \eqref{eq:soe}.
%
%If certain conditions, such as elliptic regularity and those in the Sobolev's embedding theorem \cite{gilbarg2015elliptic}, hold, then the classical solution to \eqref{eq:soe} exists and coincides with the weak solution.
%
In general, the weak solution to \eqref{eq:soe} may exist while a classical one may not.
In this paper, we therefore seek for the weak solution characterized in \eqref{eq:weak_soe} so that we can provide an answer to a BVP \eqref{eq:soe} to the best extent even if it does not admit a solution in the classical sense.
%--------------------------------------------------------------------------------
\subsection{Induced operator norm minimization}
\label{sec:minmax}
A novel point of view for the weak solution $u$ can be interpreted as follows. %The weak formulation \eqref{eq:weak_soe} inspires a novel point of view of a weak solution $u$.
We can consider $\Au: H_0^1(\Omega) \to \mathbb{R}$ as a linear functional (operator) such that $\Au(\varphi)\triangleq \langle \Au, \varphi \rangle$ as defined in \eqref{eq:weak_soe}.
Then the operator norm of $\Au$ induced from $L^2$ norm is defined by
\begin{equation}
\label{eq:norm_op}
\|\Au\|_{op} \triangleq \max\{\langle \Au,\varphi \rangle / \|\varphi\|_2\ \vert\ \varphi \in H_0^1, \varphi\neq 0\},
\end{equation}
where $\|\varphi\|_2 = (\int_\Omega |\varphi(x)|^2\dif x)^{1/2}$.
Therefore, $u$ is a weak solution of \eqref{eq:soe} if and only if $\|\Au\|_{op} = 0$ and the boundary condition $\Bu=0$ is satisfied on $\partial \Omega$.
As $\|\Au \|_{op}\ge 0$, we know that a weak solution $u$ to \eqref{eq:soe} thus solves the following two equivalent problems in observation of \eqref{eq:norm_op}:
\begin{equation}
\label{eq:min_op}
\min_{u \in H^1} \|\Au\|_{op}^2 \quad \Longleftrightarrow \quad \min_{u \in H^1} \max_{\varphi \in H_0^1} |\langle\Au,\varphi \rangle|^2 / \|\varphi\|_2^2.
\end{equation}
This result is summarized in the following theorem.
\begin{theorem}
\label{thm:min_op}
Suppose $u^{*}$ satisfies the boundary condition $\mathcal{B}[u^*]=0$, then $u^*$ is a weak solution of the BVP \eqref{eq:soe} if and only if $u^*$ solves the problems in \eqref{eq:min_op} and $\|\mathcal{A}[u^*]\|_{op}=0$.
\end{theorem}
\begin{proof}
For any fixed $u\in H^1(\Omega)$, we can see that the maximum of $\langle \Au, \varphi \rangle$ is achievable over $Y\triangleq \{\varphi\in H^{1}_{0}(\Omega)\ \vert\ \|\varphi\|_{2}=1\}$ since $\langle \Au, \cdot \rangle$ is continuous and $Y$ is closed in $H^1_0(\Omega)$.
Denote $h(u)$ as the maximum of $\langle \Au, \varphi \rangle$ over $Y$, then $h(u) = \|\Au\|_{op}$ in \eqref{eq:norm_op}.
On the other hand, the space of functions $u\in H^{1}(\Omega)$ satisfying the boundary condition $\Bu=0$, denoted by $X$, is also closed in $H^{1}(\Omega)$.
Therefore, the minimum of $h(u)$ over $X$ is also achievable.
Hence the minimax problem \eqref{eq:min_op} is well-defined.

Now we show that $u^*$ is the solution of the minimax problem \eqref{eq:min_op} if and only if it is the weak solution of the problem \eqref{eq:soe}.
Suppose $u^{*}$, satisfying the boundary condition $\mathcal{B}[u^*]=0$, is the weak solution of the problem \eqref{eq:soe}, namely $u^{*}$ satisfies \eqref{eq:weak_soe} for all $\varphi \in Y$, then $\langle \mathcal{A}[u^*], \varphi \rangle \equiv 0$ for all $\varphi \in Y$.
Therefore, $\|\mathcal{A}[u^{*}]\|_{op}=0$, and $u^{*}$ is the solution of the minimax problem \eqref{eq:min_op}.
On the other hand, suppose a weak solution $\hat{u}$ of \eqref{eq:soe} exists.
Assume that $u^*$ is the minimizer of the problem \eqref{eq:min_op}, i.e., $u^*=\argmin_{u\in X}h(u)$, but not a weak solution of the problem \eqref{eq:soe}, then there exists $\varphi^* \in Y$ such that $\langle \mathcal{A}u^*, \varphi^*\rangle > 0$.
Therefore $h(u^*)=\max_{\varphi \in Y}|\langle \mathcal{A}[u^*],\varphi \rangle|>0$.
However, as we showed above, $h(\hat{u})=0$ since $\hat{u}$ is a weak solution of \eqref{eq:soe}, which contradicts to the assumption that $u^*$ is the minimizer of $h(u)$ over $X$.
Hence $u^*$ must also be a weak solution of \eqref{eq:soe}.%, i.e., $u^*$ satisfies \eqref{eq:weak_soe} for all $\varphi \in Y$.
\end{proof}
Theorem \ref{thm:min_op} implies that, to find the weak solution of \eqref{eq:soe}, we can instead seek for the optimal solution $u$ that minimizes \eqref{eq:min_op}.
%--------------------------------------------------------------------------------
\subsection{Weak adversarial network for solving PDE}
\label{sec:PDE-GAN}
The formulation \eqref{eq:min_op} inspires an adversarial approach to find the weak solution of \eqref{eq:soe}.
More specifically, we seek for the function $\ut: \mathbb{R}^d \to \mathbb{R}$, realized as a deep neural network with parameter $\theta$ to be learned, such that $\mathcal{A}[u_\theta]$ minimizes the operator norm \eqref{eq:min_op}.
On the other hand, the test function $\varphi$, is a deep adversarial network with parameter $\eta$, also to be learned, challenges $\ut$ by maximizing $\langle \Aut, \phieta \rangle$ modulus its own norm $\|\phieta\|_2$ for every given $\ut$ in \eqref{eq:min_op}.

To train the deep neural network $\ut$ and the adversarial network $\phieta$ such that they solve \eqref{eq:min_op}, we first need to formulate the objective functions of $\ut$ and $\phieta$.
%
%Since logarithm function is monotone and strictly increasing, we can for convenience reformulate \eqref{eq:min_op} and obtain the objective of $\ut$ and $\phieta$ in the \underline{int}erior of $\Omega$ as follows,
Since logarithm function is monotone and strictly increasing, we can for convenience reformulate \eqref{eq:min_op} and obtain the objective of $\ut$ and $\phieta$ in the \underline{int}erior of $\Omega$ as follows,
\begin{equation}\label{eq:Lint}
\Lint(\theta,\eta) \triangleq \log |\langle \Aut, \phieta \rangle |^2- \log \|\phieta\|_2^2.
\end{equation}
In addition, the weak solution $\ut$ also need to satisfy the boundary condition $\Bu=0$ on $\partial \Omega$ as in \eqref{eq:soe}.
Let $\{x_b^{(j)}\}^{N_b}_{j=1}$ be a set of $N_b$ collocation points on the \underline{b}oun\underline{d}a\underline{ry} $\partial\Omega$, then the squared error of $\ut$ for Dirichlet boundary condition $u=g$ on $\partial \Omega$ is given by
\begin{equation}
\label{eq:Lbdry}
\Lbdry(\theta) \triangleq (1/N_b) \cdot \textstyle\sum_{j=1}^{N_b} |\ut(x_b^{(j)}) - g(x_b^{(j)})|^2.
\end{equation}
If the Neumann boundary condition in \eqref{eq:soe} is imposed in the BVP \eqref{eq:soe}, then one can form the loss function $\Lbdry(\theta) = (1/N_b) \cdot \sum_{j=1}^{N_b} |\sum_{i=1}^d n_i(x_b^{(j)})\,\partial_i \ut(x_b^{(j)}) - g(x_b^{(j)})|^2$ instead, where $\vec{n}(x)=(n_1(x),\dots,n_d(x))$ is the outer normal direction at $x\in\partial\Omega$.
The total objective function is the weighted sum of the two objectives \eqref{eq:Lint} and \eqref{eq:Lbdry}, for which we seek for a saddle point that solves the minimax problem:
\begin{equation}
\label{eq:Ltotal}
\min_{\theta} \max_{\eta} L(\theta, \eta),\quad \mbox{where}\quad L(\theta, \eta) \triangleq\Lint(\theta,\eta) + \alpha \Lbdry(\theta),
\end{equation}
where $\alpha>0$ is user-chosen balancing parameter. In theory, the weak solution attains zero for both $\Lint$ and $\Lbdry$, so any choice of $\alpha$ would work. However,
different $\alpha$ values impact the performance of the training and we will give examples in Section \ref{sec:experiment}.
%
%--------------------------------------------------------------------------------
\subsection{Training algorithm for the weak adversarial network}
\label{subsec:training}
Given the objective function \eqref{eq:Ltotal}, the key ingredients in the network training are the gradients of $L(\theta,\eta)$ with respect to the network parameters $\theta$ and $\eta$.
Then $\theta$ and $\eta$ can be optimized by alternating gradient descent and ascent of $L(\theta,\eta)$ in \eqref{eq:Ltotal} respectively.

To obtain the gradients of $\Lint$ in \eqref{eq:Ltotal}, we first denote the integrand of $\langle \mathcal{A}[\ut],\phieta \rangle$ in \eqref{eq:weak_soe} as $I(x;\theta,\eta)$ for every given $\theta$ and $\eta$.
For instance, for the second-order elliptic PDE \eqref{eq:soe}, $I(x;\theta,\eta)$, $\nabla_\theta I(x;\theta,\eta)$, $\nabla_{\eta} I(x;\theta,\eta)$ are given below in light of the weak formulation \eqref{eq:weak_soe}:
\begin{equation}
\label{eq:int}
\begin{split}
& I(x;\theta,\eta) = \textstyle\sum_{j=1}^d \textstyle\sum_{i=1}^d a_{ij}(x)\partial_j \ut(x)\partial_{i}\phieta(x) + \sum_{i=1}^d b_i(x) \phieta(x) \partial_{i}\ut(x)\\ & \quad \qquad \qquad + c(x)\ut(x) \phieta(x) - f(x)\phieta(x) \\
& \nabla_\theta I(x;\theta,\eta) = \textstyle\sum_{j=1}^d \sum_{i=1}^d a_{ij}(x)\partial_j \nabla_\theta\ut(x)\partial_{i}\phieta(x) + \sum_{i=1}^d b_i(x) \phieta(x) \partial_{i}\nabla_\theta\ut(x) \\ & \qquad \qquad \qquad + c(x)\nabla_\theta\ut(x) \phieta(x) - f(x)\phieta(x) \\
& \nabla_\eta I(x;\theta,\eta) = \textstyle\sum_{j=1}^d \sum_{i=1}^d a_{ij}(x)\partial_j \ut(x)\partial_{i} \nabla_\eta \phieta(x) + \sum_{i=1}^d b_i(x) \nabla_\eta \phieta(x) \partial_{i}\ut(x) \\ & \qquad \qquad \qquad + c(x)\ut(x) \nabla_\eta \phieta(x) - f(x)\nabla_\eta \phieta(x) \\
\end{split}
\end{equation}
where $\nabla_\theta \ut$ and $\nabla_\eta \phieta$ are the standard gradients of the networks $\ut$ and $\phieta$ with respect to their network parameters $\theta$ and $\eta$.
Furthermore, the algorithm and numerical experiments conducted in this paper are implemented in the TensorFlow \cite{abadi2016tensorflow} framework. In this situation,  we take advantage of TensorFlow to calculate those derivatives automatically within the framework. To be more specific, due to the definition of $\Lint$ in \eqref{eq:Lint} and the integrands in \eqref{eq:int}, we can obtain that $\nabla_\theta \Lint(\theta,\eta) = 2 (\textstyle\int_\Omega I(x;\theta,\eta) \dif x)^{-1} (\int_\Omega \nabla_\theta I(x;\theta,\eta) \dif x)$.
Then we randomly sample $N_r$ collocation points $\{x_r^{(j)} \in \Omega \ \vert \ j \in [N_r]\}$ uniformly in the interior of the \underline{r}egion $\Omega$, and approximate the gradient $\nabla_\theta \Lint(\theta,\eta) \approx 2\cdot (\sum_{j=1}^{N_r} I(x_r^{(j)};\theta,\eta))^{-1} (\sum_{j=1}^{N_r}\nabla_\theta I(x_r^{(j)};\theta,\eta))$.
The gradients $\nabla_\eta \Lint$, $\nabla_\theta \Lbdry$ can be approximated similarly, and hence we omit the details here.
With the gradients of $\nabla_\theta L$ and $\nabla_\eta L$, we can apply alternating updates to optimize the parameters $\theta$ and $\eta$.
The resulting algorithm, termed as the weak adversarial network (WAN), is summarized in Algorithm \ref{alg:wan}.
%%%%%%%%%%%%%%%%%%%%%%%%%%%%%%%%%%%%%%%%%%%%%%%%%%%%%%%%%%%%%%%%%%%%%%%%%%%%%%%%%
%--------------------------------------------------------------------------------
\subsection{Efficiency and stability improvements of WAN}
\label{subsec:improve}
During our experiments, we observed that several small modifications can further improve the efficiency and/or stability of Algorithm \ref{alg:wan} in practice.
One of these modifications is that, to enforce $\phieta=0$ on $\Omega$, we can factorize $\phieta=w\cdot\veta$, where $w$ vanishes on $\partial \Omega$ and $\veta$ is allowed to take any value on $\partial \Omega$.
To obtain $w$ for the domain $\Omega$ in a given BVP \eqref{eq:soe}, we can set it to the signed distance function of $\Omega$, i.e., $w(x)=\text{dist}(x,\partial \Omega)\triangleq \inf\{|x-y|: y \in \partial \Omega\}$ if $x \in \Omega$ and $-\text{dist}(x,\partial \Omega)$ if $x\notin \Omega$.
This signed distance function can be obtained by the fast marching method.
Alternatively, one can pre-train $w$ as a neural network such that $w(x)>0$ for $x\in\Omega$ and $w(x)=0$ for $x\in\partial\Omega$.
To this end, one can parameterize $w_\xi: \Omega \to \mathbb{R}$ as a neural network and optimize its parameter $\xi$ by minimizing the loss function $\sum_{j=1}^{N_b}|w_\xi(x_b^{(j)})| - \varepsilon \sum_{j=1}^{N_r}\log w_\xi(x_r^{(j)})$.
In either way, we pre-compute such $w$ and fix it throughout Algorithm \ref{alg:wan} WAN, then the updates of parameters are performed for $\ut$ and $\veta$ only.
In this case, $\phieta=w\cdot\veta$ always vanishes on $\partial \Omega$ so we do not need to worry about the boundary constraint of $\phieta$ during the training.

In addition, our experiments show that in the training process for the weak solution neural network $u_{\theta}$, applying gradient descent directly to $|\langle \mathcal{A}[u_\theta],\varphi_\eta \rangle|^2/\|\varphi\|_2^2$ instead of the logarithm term appears to improve efficiency.  Finally, formulating the loss function for the boundary condition with absolute error rather than squared error appears empirically to be more efficient in some cases.
%In the training process for the neural network $u_{\theta}$, we apply gradient descent directly to $|\langle \mathcal{A}[u_\theta],\varphi_\eta \rangle|^2$ without the logarithm which appears to improve efficiency.
%
%In addition, we find that the training process will be more stable by adding the following term to the loss function \eqref{eq:Ltotal}:
%\begin{equation}
%\label{eq:Lw}
% L_w(\theta)\triangleq |\langle \Aut, w \rangle |^2.
%\end{equation}
%
%Therefore, the objective function for training $u_{\theta}$ becomes
%
%\begin{equation}
%\label{eq:loss_for_u}
%L(\theta, \eta) \triangleq \alpha \Lint(\theta,\eta)+\beta L_{w}(\theta) + \Lbdry(\theta),
%\end{equation}
%where $\alpha,\beta>0$ are user-chosen balancing parameters.
%
The computer code that implements these modifications for all test problems in Section \ref{sec:experiment} will be released upon request.
\begin{algorithm}[t!]
\caption{Weak Adversarial Network (WAN) for Solving High-dimensional static PDEs. }
\label{alg:wan}
\begin{algorithmic}
\STATE \textbf{Input:} {$N_r/N_b$: number of region/boundary collocation points; $K_u/K_{\varphi}$: number of solution/adversarial network parameter updates per iteration.}
\STATE \textbf{Initialize:} Network architectures $\ut,\phieta:\Omega \to \mathbb{R}$ and parameters $\theta,\eta$.

\WHILE{not converged}
\STATE {Sample collocation points $\{x^{(j)}_r \in \Omega : j\in[N_r]\}$ and $\{x^{(j)}_b \in \partial \Omega : j\in[N_b]\}$}
\STATE \texttt{\# update weak solution network parameter}
\FOR{$k=1,\dots,K_u$}
\STATE {Update $\theta \leftarrow \theta - \tau_\theta \nabla_\theta L$ where $\nabla_\theta L$ is approximated using $\{x^{(j)}_r\}$ and $\{x^{(j)}_b\}$.}
\ENDFOR
\STATE \texttt{\# update test function network parameter}
\FOR{$k=1,\dots,K_{\varphi}$ }
\STATE {Update $\eta \leftarrow \eta + \tau_\eta \nabla_\eta L$ where $\nabla_\eta L$ is approximated using $\{x^{(j)}_r\}$.}
\ENDFOR
\ENDWHILE
\STATE \textbf{Output:} {Weak solution $\ut(\cdot)$ of \eqref{eq:soe}.}
\end{algorithmic}
\end{algorithm}

\subsection{Weak adversarial network for PDEs involving time}
\label{subsec:parabolic}
In this subsection, we consider extending the proposed weak adversarial network method to solve IBVPs with time-dependent PDEs.
We provide two approaches for such case: one is to employ semi-discretization in time and iteratively solve $u(x,t_n)$ from a time-independent PDE for each $t_n$, where Algorithm \ref{alg:wan} directly serves as a subroutine; the other one is to treat $x$ and $t$ jointly and consider the weak solution and test functions in the whole region $\Omega\times[0,T]$ without any discretization.

\subsubsection{Semi-discretization in time}
The weak adversarial network approach can be easily applied to time-dependent PDEs, such as the parabolic equation \eqref{eq:parab}, by discretizing the time and solving an elliptical-type static PDE for each time point.
To this end, we partition $[0,T]$ into $N$ uniform segments using time points $0=t_0<t_1<\dots<t_N=T$, and apply the Crank-Nicolson scheme \cite{crank1996practical} in classical finite difference method for \eqref{eq:parab} at each time $t_n$ for $n=0,\dots,N-1$ to obtain
\begin{equation}
\label{eq:slice-parab}
u(x,t_{n+1})-u(x,t_n)=\frac{h}{2} \del[2]{\mathcal{L}(x,t_{n+1}; u(x,t_{n+1}))+f(x,t_{n+1})+\mathcal{L}(x,t_{n}; u(x,t_{n}))+f(x,t_n)}
\end{equation}
where $h=T/N$ is the time step size in discretization, $t_n=nh$, and
\begin{equation}
\mathcal{L}(x,t;u)\triangleq \textstyle\sum_{i=1}^d \partial_{i} \del[2]{\textstyle\sum_{j=1}^d a_{ij}(x,t)\partial_{j}u(x,t)} - \textstyle\sum_{i=1}^d b_i(x,t) \partial_{i} u(x,t) - c(x,t)u(x,t).
\end{equation}
More precisely, we start with $u(x,t_0)=u(x,0)=h(x)$, and solve for $u(x,t_1)$ from \eqref{eq:slice-parab} for $n=1$. % with boundary condition $u(x,t_1)=g(x,t_1)$ on $\partial \Omega$.
Since \eqref{eq:slice-parab} is an elliptical-type PDE in $u(x,t_1)$ with boundary value $u(x,t_1)=g(x,t_1)$ on $\partial \Omega$, we can apply Algorithm \ref{alg:wan} directly and obtain $u(x,t_1)$ as the parameterized neural network $u_{\theta_1}(x)$ with parameter $\theta_1$ output by Algorithm \ref{alg:wan}.
Following this procedure, we can solve \eqref{eq:slice-parab} for $u(x,t_n) = u_{\theta_n}(x)$ for $n=2,3,\dots,N$ in order.
This process is summarized in Algorithm \ref{alg:wan_parab_slice}.
Other types of time discretization can be employed and the IBVP can be solved with similar idea.
%idea presented above, and hence
%We omit the discussions on this.
%
\begin{algorithm}[t!]
\caption{Solving parabolic PDE \eqref{eq:parab} with semi-discretization in time and Algorithm \ref{alg:wan} as subroutine}
\label{alg:wan_parab_slice}
\begin{algorithmic}
\STATE \textbf{Input:} {$N_r,N_b,K_u,K_{\varphi}$ as in Algorithm \ref{alg:wan}. $N$: number of time points; $h=T/N$.}
\STATE \textbf{Initialize:} Network architectures $\ut,\phieta:\Omega \to \mathbb{R}$ and parameters $\theta,\eta$ for each $t_n$. Set $u(x,t_0)=u(x,0)=h(x)$.
\FOR{$n=0,\cdots,N-1$}
\STATE {Solve for $u(x,t_{n+1})=u_{\theta_{n+1}}(x)$ from the elliptical equation \eqref{eq:slice-parab} using Algorithm \ref{alg:wan}}
%\STATE {Freeze the parameter of neural network $u_{\theta}$ and denote this frozen neural network as $u_{\hat{\theta}}$}
\ENDFOR
\STATE \textbf{Output:} {Weak solution $\ut(\cdot,t_n)$ of \eqref{eq:parab} for $n=1,\dots,N$.}
\end{algorithmic}
\end{algorithm}

\subsubsection{Solving PDE with space and time variables jointly}
The proposed weak adversarial network approach can also be generalized to solve the IBVP \eqref{eq:parab} with space and time variables jointly.
In this case, the \textit{weak formulation} of \eqref{eq:parab} can be obtained by multiplying both sides of \eqref{eq:parab} by a test function $\varphi(\cdot, t)\in H^1_0(\Omega)$ a.e. in $[0,T]$ and integrating by parts:
\begin{equation}
\label{eq:weak_parab_xt}
\begin{aligned}
 0 = \langle \Au, \varphi \rangle &\textstyle\triangleq\int_\Omega \del[1]{u(x,T)\varphi(x,T)-h(x)\varphi(x,0)}\dif x-\textstyle\int_{0}^T \textstyle\int_{\Omega}u\partial_t\varphi\dif x\dif t \\
 &\quad+ \textstyle\int_{0}^T \textstyle\int_\Omega \del[1]{\textstyle\sum_{j=1}^d \textstyle\sum_{i=1}^d a_{ij}\partial_j u\partial_{i}\varphi + \textstyle\sum_{i=1}^d b_i \varphi \partial_{i}u  + cu \varphi - f\varphi } \dif x\dif t
\end{aligned}
\end{equation}
Following the same idea presented in Section \ref{sec:minmax}--\ref{sec:PDE-GAN}, we parameterize the weak solution $u$ and test function $\varphi$ as deep neural networks $u_\theta,\varphi_\eta: \Omega \times [0,T] \to \mathbb{R}$ with parameters $\theta$ and $\eta$ respectively.
Then we form the objective function in the saddle-point problem of $\theta$ and $\eta$ as
\begin{equation}
\label{eq:loss_for_u_parab}
L(\theta, \eta) \triangleq \Lint(\theta,\eta) +\gamma \Linit(\theta)+ \alpha \Lbdry(\theta),
\end{equation}
where $\alpha,\gamma>0$ are user-chosen balancing parameters.
In \eqref{eq:loss_for_u_parab}, the loss function $\Lint$ of the interior of $\Omega\times [0,T]$ has the same form as \eqref{eq:Lint} but with $\langle \mathcal{A}[u_\theta],\varphi_\eta \rangle$ defined in \eqref{eq:weak_parab_xt} and $\|\varphi_\eta\|_2^2 \triangleq \int_{0}^T \int_\Omega |\varphi(x,t)|^2\dif x \dif t$; $\Linit$ of the initial value condition in $\Omega$ and $\Lbdry$ of the boundary value condition on $\partial\Omega \times [0,T]$ are given by
\begin{align}
\Linit(\theta) &\triangleq (1/N_{a}) \cdot \textstyle\sum_{j=1}^{N_{a}} |\ut(x_a^{(j)},0) - h(x_a^{(j)})|^2
\label{eq:L_init_parab} \\
\Lbdry(\theta) &\triangleq (1/N_{b}) \cdot \textstyle\sum_{j=1}^{N_{b}} |\ut(x_b^{(j)},t_b^{(j)}) - g(x_b^{(j)},t_b^{(j)})|^2
\label{eq:L_bdry_parab}
\end{align}
where $\{x_a^{(j)}:j\in[N_a]\} \subset \Omega$ are $N_a$ collocation points for the initial condition and $\{(x_b^{(j)},t_b^{(j)}):j\in[N_b]\} \subset \partial \Omega \times [0,T]$ are $N_b$ collocation points for the boundary condition.
%
% then use the weak adversarial network to find the weak solution of \eqref{eq:parab}. To formulate the objective function $L(\theta, \eta)$ for the training neural network $u_{\theta}$ and $\varphi_{\eta}$, we first reformulate \eqref{eq:min_op_parab} with logarithm function and get the loss $L_{int}$ which has the same form of \eqref{eq:Lint}. For the boundary value condition, we again use the same expression in section \ref{sec:minmax} to denote $L_{bdry}$ with $\{(x^{j},t^{j})\}^{N_b}_{j=1}$ being the set of $N_b$ collocation points on the boundary $\partial\Omega\times[0,T]$. We formulate a new loss term, which denotes by $L_{init}$, to enforce the initial value condition. We assume that $\{x^{j}\}^{N_{a}}_{j=1}$ is a set of collocation data on $\Omega$, then the loss $L_{init}$ is

%
Similar as in Section \ref{subsec:improve}, we factorize $\varphi_{\eta}=w\cdot v_{\eta}$ where $w: \Omega_T\rightarrow \mathbb{R}$ is set to a function which vanishes on $\partial\Omega$ in advance.
The training process is similar as above, which is summarized in Algorithm \ref{alg:wan_parab}.
\begin{algorithm}[t!]
\caption{Weak Adversarial Network (WAN) for Solving high-dimensional PDEs in whole space $\Omega\times[0,T]$}
\label{alg:wan_parab}
\begin{algorithmic}
\STATE \textbf{Input:} {$N_r/N_b/N_{a}$: number of region/boundary/initial collocation points; $K_u/K_{\varphi}$.
$\Omega_T \triangleq \Omega\times[0,T]$.}
\STATE \textbf{Initialize:} Network architectures $\ut,\phieta:\Omega_T \to \mathbb{R}$ and parameters $\theta,\eta$.

\WHILE{not converged}
\STATE {Sample points $\{(x^{(j)}_r,t_r^{(j)}): j\in[N_r]\}\subset \Omega_T$, $\{(x^{(j)}_b,t_b^{(j)})  : j\in[N_b]\}\subset \partial \Omega\times[0,T]$, $\{x^{(j)}_{a} : j\in[N_{a}]\} \subset \Omega$}
\STATE \texttt{\# update weak solution network parameter}
\FOR{$k=1,\dots,K_u$}
\STATE {Update $\theta \leftarrow \theta - \tau_\theta \nabla_\theta L$ where $\nabla_\theta L$ in \eqref{eq:loss_for_u_parab} is approximated using $\{(x^{(j)}_r,t_r^{(j)})\}$ and $\{(x^{(j)}_b,t^{(j)}_b)\}$ and $\{x^{(j)}_{a}\}$.}
\ENDFOR
\STATE \texttt{\# update test function network parameter}
\FOR{$k=1,\dots,K_{\varphi}$ }
\STATE {Update $\eta \leftarrow \eta + \tau_\eta \nabla_\eta L$ where $\nabla_\eta L$ in \eqref{eq:loss_for_u_parab} is approximated using $\{(x^{(j)}_r, t^{(j)}_r)\}$.}
\ENDFOR
\ENDWHILE
\STATE \textbf{Output:} {Weak solution $\ut(x,t)$ in $\Omega_T$.}
\end{algorithmic}
\end{algorithm}
%%%%%%%%%%%%%%%%%%%%%%%%%%%%%%%%%%%%%%%%%%%%%%%%%%%%%%%%%%%%%%%%%%%%%%%%%%%%%%%%%%
\section{Numerical Experiments}
\label{sec:experiment}
\subsection{Experiment setup}
In this section, we conduct a series of numerical experiments of the proposed algorithms (Algorithms \ref{alg:wan}--\ref{alg:wan_parab}) on BVP and IBVP with high-dimensional linear and nonlinear PDEs defined on regular and irregular domains.
To quantitatively evaluate the accuracy of a solution $u_\theta$, we use the $L_2$ relative error $\|u_\theta-u^*\|_2/\|u^*\|_2$, where $u^*$ is the exact solution of the problem and $\|u\|_2^2 = \int_{\Omega} |u|^2 \dif x$.
To compute this error in high dimensional domain $\Omega$, we use a regular mesh grid of size $100 \times 100$ for $(x_1, x_2)$, and sampled one point $x$ for each of these grid points (i.e., for each grid point $(x_1, x_2)$, randomly draw values of $(x_3,\dots,x_d)$ of $x$ within the domain $\Omega$). These points are sampled in advance and then used for all comparison algorithms to compute their errors. They are different from those sampled during training processes.
In all experiments, we set both of the primal network (weak solution $u_\theta$) and the adversarial network (test function $\phieta$) in the proposed algorithms as fully-connected feed-forward networks.
Unless otherwise noted, $u_{\theta}$ network is set to have 6 hidden layers, with 40 neurons per hidden layer.
The activation functions of $u_{\theta}$ are set to tanh for layers 1, 2, 4  and 6, and elu for the problem in Section \ref{subsubsec:weak_vs_strong} and softplus for all other problems for layers 3 and 5.
We do not apply activation function in the last, output layer.
For the network $\phieta$, it consists of $8$ hidden layers with 50 neurons per hidden layer.
The activation functions are set to tanh for layers 1 and 2, softplus for layers 3, 5, and 8, sinc for layers 2, 5, and 7, and again no activation in the last layer.
The parameters $\theta$ and $\eta$ of the two networks are updated alternately based on \eqref{eq:Ltotal} for BVP and \eqref{eq:loss_for_u_parab} for IBVP using AdaGrad algorithm \cite{duchi2011adaptive}.
%
%The code is available on Github at \url{https://github.com/yaohua32/wan}.
%
Notations of all model and algorithm parameters are summarized in Table \ref{tab:notation} for quick reference.
The values of these parameters are given in the description of each experiment below.
%
% with Intel CPU 2.2GHz, Nvidia Tesla T4 and 16GB of memory.
%
\begin{table}[t]
    \caption{List of model and algorithm parameters.}
    \label{tab:notation}
    \centering
    \begin{tabular}{c|l}
        \hline
        \textbf{Notation} & \textbf{Stands for ...} \\
        \hline
        \hline
        $d$ & Dimension of $\Omega \subset \mathbb{R}^d$ \\
        \hline
        $K_{\varphi}$ & Inner iteration to update test function $\varphi_\eta$\\
        \hline
        $K_u$ & Inner iteration to update weak solution $u_\theta$\\
        \hline
        $\tau_{\eta}$ & Learning rate for network parameter $\eta$ of test function $\varphi_\eta$\\
        \hline
        $\tau_{\theta}$& Learning rate for network parameter $\theta$ of weak solution $u_\theta$\\
        \hline
        $N_{r}$ & Number of sampled collocation points in the \underline{r}egion $\Omega$\\
        \hline
        $N_{b}$ & Number of sampled collocation points on the \underline{b}oundary $\partial \Omega$ or $\partial \Omega\times [0,T]$\\
        \hline
        $N_{a}$ & Number of sampled collocation points in $\Omega=\Omega\times\{0\}$ at initi\underline{a}l time \\
        \hline
        $\alpha$ & Weight parameter of $\Lbdry(\theta)$ on the boundary $\partial\Omega$ \\
        \hline
        $\gamma$ & Weight parameter of $\Linit(\theta)$ for the initial value condition\\
        \hline
    \end{tabular}
\end{table}

\subsection{Experimental results}
\subsubsection{Weak form versus strong form}
\label{subsubsec:weak_vs_strong}
In the first test, we show that Algorithm \ref{alg:wan} based on the weak formulation of PDEs can be advantageous for problems with singularities.
Consider the following Poisson equation with Dirichlet boundary condition:
\begin{equation}\label{eq:eq-weak}
\begin{cases}
-\Delta u = f, &\quad \text{in}\ \Omega\\
u= g, &\quad \text{on}\  \partial\Omega \\
\end{cases}
\end{equation}
where $\Omega = (0,1)^2 \subset \mathbb{R}^2$, $f\equiv -2$ in $\Omega$, and $g(x_1,0)=g(x_1,1) = x_1^2$ for $0\le x_1 \le \frac{1}{2}$, $g(x_1,0)=g(x_1,1)=(x_1-1)^2$ for $\frac{1}{2}\le x_1 \le 1$, and $g(0,x_2)=g(1,x_2)=0$ for $0\le x_2 \le 1$ on $\partial \Omega$.
This problem does not admit a strong (classical) solution, but only a unique weak solution $u^*(x)=u^*(x_1,x_2)=x_1^2$ when $0 \le x_1\leq \frac{1}{2}$ and $u^*(x)= (x_1-1)^2$ when $\frac{1}{2} \le x_1 \le 1$, which is shown in Figure \ref{subfig:nonsmooth_u_true}.

We apply the proposed Algorithm \ref{alg:wan} to \eqref{eq:eq-weak} with $K_\varphi =1, K_u = 2$, $\tau_\eta = 0.04, \tau_\theta = 0.015$, $N_r=10^4$, $N_b = 4\times 100$ ($100$ collocation points on each side of $\Omega$), and $\alpha=10,000\times N_b$ for 100,000 iterations, after which we obtain an approximation $u_{\text{WAN}}$ shown in Figure \ref{subfig:nonsmooth_u_WAN}.
For comparison, we also apply two state-of-the-art deep learning based methods, the physics-informed neural networks (PINN) \cite{raissi2019physics}\footnote{PINN implementation obtained from \url{https://github.com/maziarraissi/PINNs}.} and the deep Ritz method (DRM) \cite{weinan2018deep}\footnote{DRM implementation obtained from \url{https://github.com/ZeyuJia/DeepRitzMethod}.}, to the same problem \eqref{eq:eq-weak}.
PINN \cite{raissi2019physics} is based on the strong form of \eqref{eq:eq-weak} where the loss function sums the squared errors in the PDE and the boundary condition at sampled points inside $\Omega$ and on $\partial \Omega$, respectively.
DRM \cite{weinan2018deep} is designed to solve a class of PDEs that can be reformulated as equivalent energy minimization problems. The constraint due to the boundary condition is formed as a penalty term and added to the energy in DRM \cite{weinan2018deep}.
We parameterize the solution $u$ using the default fully connected network activation function tanh for PINN and the residual network (ResNet) structure with activation function $\max(x^3,0)$ for DRM.
The weight of the boundary term in the loss function is set to $1$ for PINN and $500$ for DRM, which seem to yield optimal solution quality for these methods.
Same as WAN, the network $u$ has 6 hidden layers and 40 neurons per hidden layer in both PINN and DRM. The numbers $N_r$ and $N_b$ of sample points in $\Omega$ and on $\partial\Omega$ respectively are also the same as WAN.
We use Adam optimizer with step size $0.001$ as suggested in PINN and DRM.
%
% \begin{align}
%  L_{PINN}&= \frac{1}{N_r}\sum^{N_r}_{i=1}\bigg(-\Delta u_{PINN}(x_i)-f(x_i)\bigg)^2+ \frac{\alpha}{N_b}\sum^{N_b}_{j=1}(u_{PINN}(x_j)-g(x_j))^2
%  \label{eq:loss_PINN}
% \\
% L_{DRM}&= \frac{1}{N_r}\sum^{N_r}_{i=1}\bigg(\frac{1}{2}|\nabla u_{DRM}(x_i)|^2-f(x_i)u_{DRM}(x_i)\bigg)+ \frac{\alpha}{N_b}\sum^{N_b}_{j=1}(u_{DRM}(x_j)-g(x_j))^2 \label{eq:loss_DRM}
% \end{align}
%
%\begin{equation*}
%L_{WAN}= \frac{1}{N_r}\bigg(\sum^{N_r}_{i=1}\nabla u_{WAN}(x_i)\nabla v(x_i)+\sum^{N_r}_{i=1} f(x_i)v(x_i)\bigg)^2/\sum^{N_r}_{i=1}|v(x_i)|^2+\frac{\alpha}{N_b}\sum^{N_b}_{j=1}(u_{WAN}(x_j)-g(x_j))^2
%\end{equation*}
%
The results after 100,000 iterations of these two comparison methods, denoted by $u_{\text{PINN}}$ and $u_{\text{DRM}}$, are shown in Figures \ref{subfig:nonsmooth_u_PINN} and \ref{subfig:nonsmooth_u_DRM}, respectively.
The comparison of relative error versus time is shown in Figure \ref{subfig:nonsmooth_error_time}, and the change process of objective functions over the number of iterations by PINN, DRM, and WAN are shown in Figures \ref{subfig:nonsmooth_loss_PINN}, \ref{subfig:nonsmooth_loss_DRM}, and \ref{subfig:nonsmooth_loss_WAN}, respectively.
(We plot these figures separately since the objective functions are defined differently in these methods).
We also plot the absolute value of residual $|\Delta u_{\text{PINN}} - f|$ for PINN in Figure \ref{subfig:nonsmooth_f}, which shows that the solution of PINN based on strong form satisfies the PDE in $\Omega$ closely.
However, the pointwise absolute errors $|u-u^*|$ with $u$ obtained by PINN, DRM, and WAN shown in Figures \ref{subfig:nonsmooth_abs_PINN}, \ref{subfig:nonsmooth_abs_DRM} and \ref{subfig:nonsmooth_abs_WAN} (note the significantly lower scale in color bar in Figure \ref{subfig:nonsmooth_abs_WAN}) imply that WAN can better capture the singularity of solution due to its use of weak form of the PDE.
We tested a variety of network parameters and activation functions for PINN and DRM, but they all yield solutions with similar patterns as shown in Figures \ref{subfig:nonsmooth_u_PINN} and \ref{subfig:nonsmooth_u_DRM}.
For example, we used larger weight on the boundary term to enforce better alignment of $u_{\text{PINN}}$ with $g$ on $\partial \Omega$, but this resulted in even worse matching of $\Delta u_{\text{PINN}}$ and $f$ in $\Omega$, and vice versa.
We believe it is because the solution $u_{\text{PINN}}$ based on the strong form tends to enforce smoothness in $\Omega$ and aligns with the boundary value $g$ on $\partial\Omega$ by violating both the PDE and the boundary condition slightly, but this results in a smooth solution severely biased from $u^*$ as shown in Figure \ref{subfig:nonsmooth_u_PINN}.
DRM also suffers the issue of problem singularity due to the heavy penalty term $\int_{\Omega} |\nabla u|^2 \dif x$ on the gradient $\nabla u$ in the objective function.
On the other hand, the solution $u_{\text{WAN}}$ obtained by Algorithm \ref{alg:wan} can capture the singularities at the center line and faithfully recover the solution $u^*$ as shown in Figure \ref{subfig:nonsmooth_u_WAN}.
Furthermore, we also tested the algorithm reported in \cite{xu2018deep}, which is also based on the strong form of PDEs, on this problem and obtained similar results as PINN.
%\ye{1. Need to clearly state how the objective function of WAN is computed: did you compute $L(\theta,\eta)$? If yes we should avoid calling it ``loss function'' and it's not necessarily decreasing either. 2. Double check Figure 1i, do you mean the relative errors of PINN and WAN are always above 1? }\zang{1. Yes, I computed the $L(\theta, \eta)$ with the term $L_{int}$ without taking log. 2. For PINN, the answer is yes. For DRM, the relative error is less than 1 in the early stage, and then increased.}

%%%%%%%%%%%%%%%%%%%%%%%%%%%%%%%%%%%%%%%%%%%%%%%%%%%%%%%%%%%%%%%%%%%%%%%%%%%%%%%%%%%%%%%%%%%%%%%%%%
\begin{figure}[ht!]
\centering
\begin{subfigure}[b]{.245\textwidth}
\includegraphics[width=\textwidth]{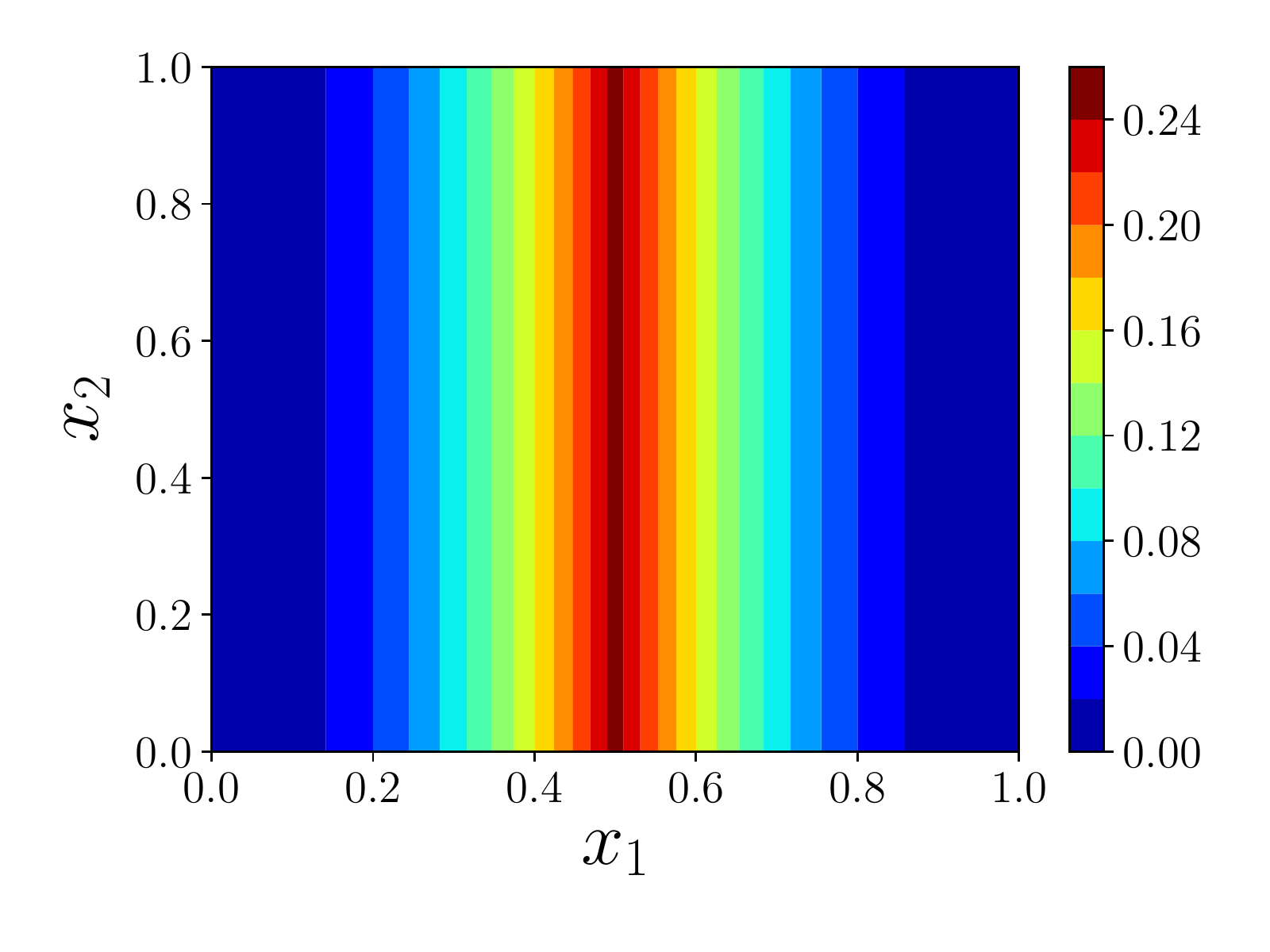}
\caption{True $u^*$}
\label{subfig:nonsmooth_u_true}
\end{subfigure}
\begin{subfigure}[b]{.245\textwidth}
\includegraphics[width=\textwidth]{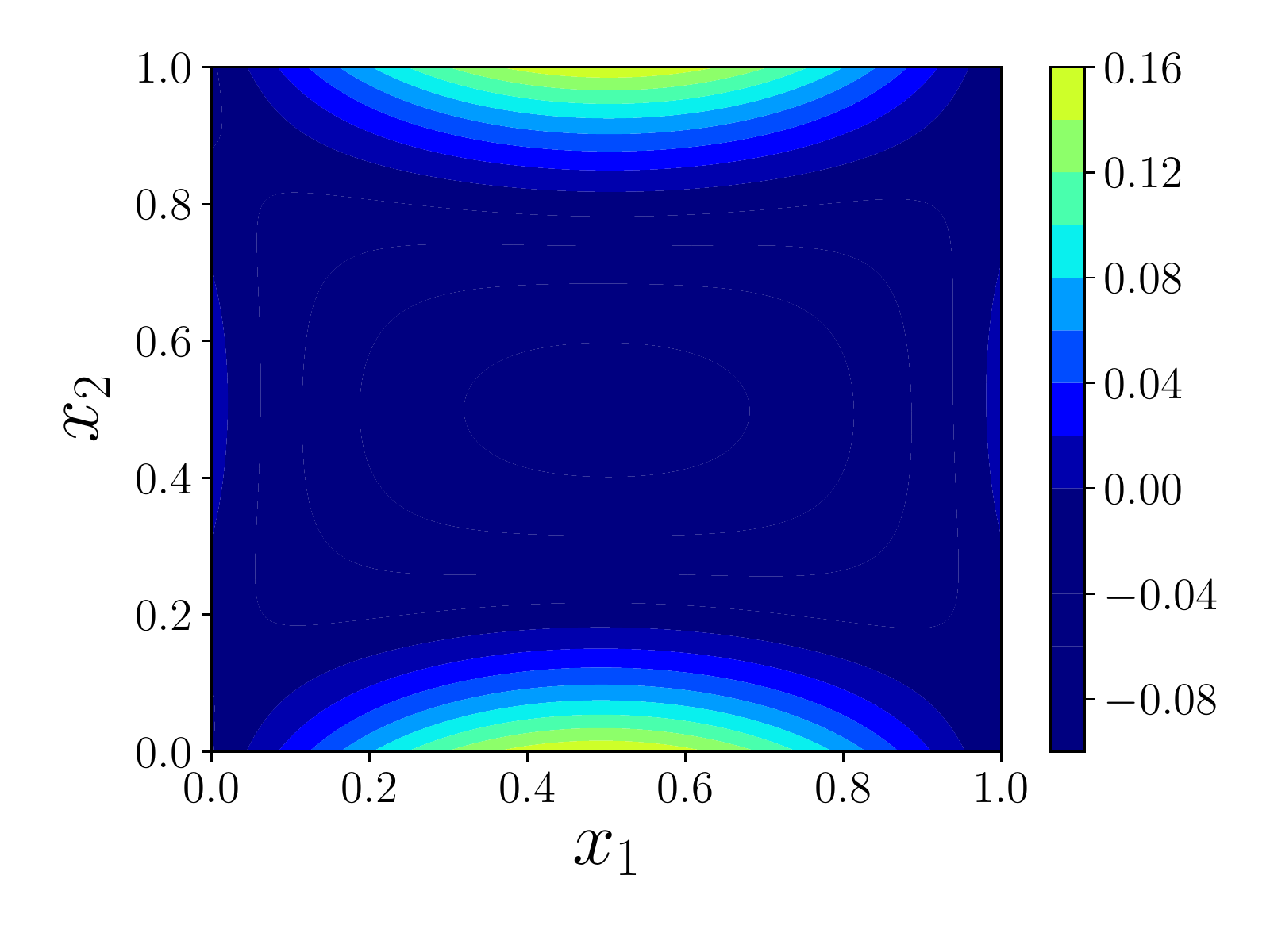}
\caption{$u_{\text{PINN}}$}
\label{subfig:nonsmooth_u_PINN}
\end{subfigure}
\begin{subfigure}[b]{.245\textwidth}
\includegraphics[width=\textwidth]{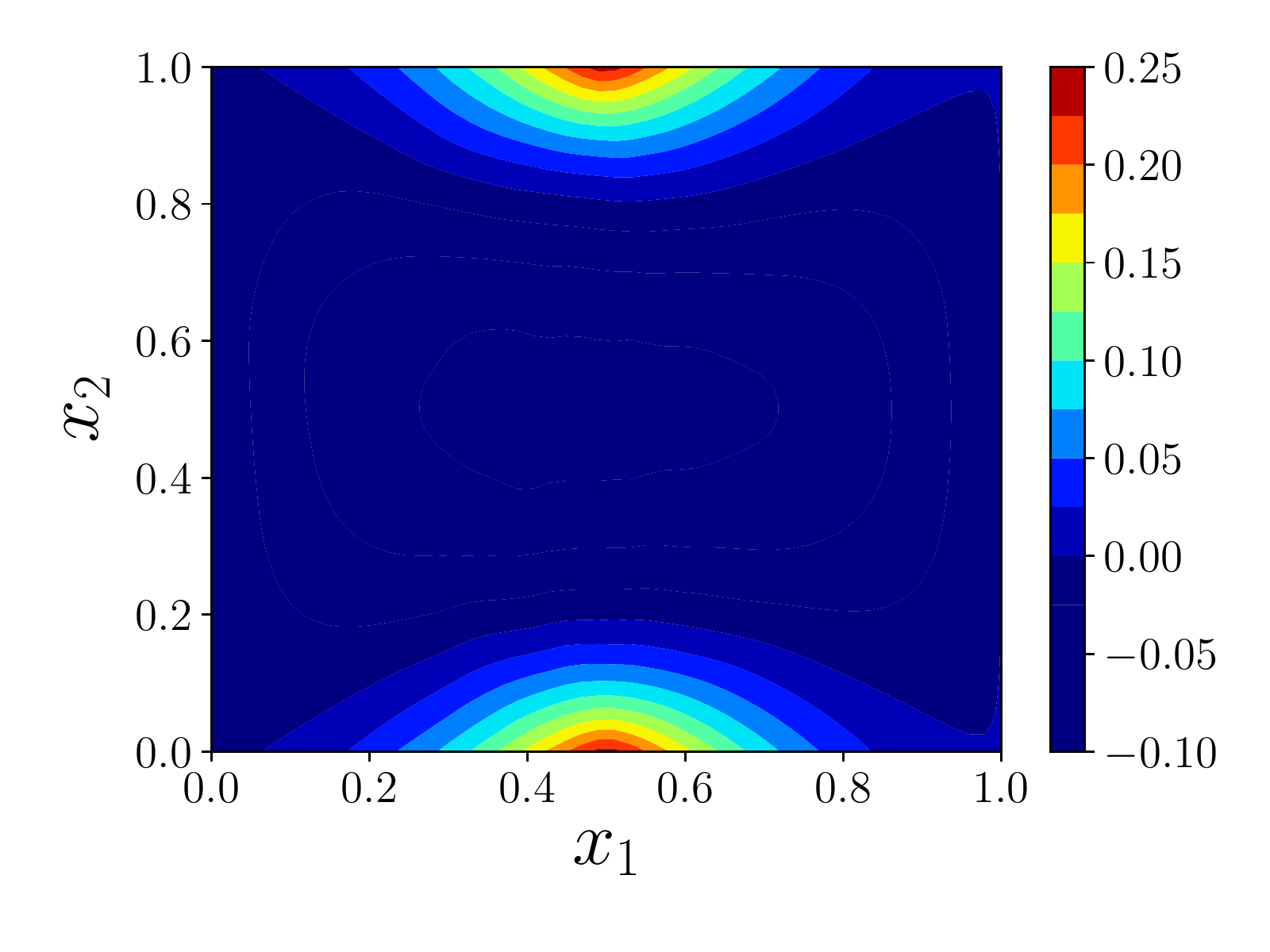}
\caption{$u_{\text{DRM}}$}
\label{subfig:nonsmooth_u_DRM}
\end{subfigure}
\begin{subfigure}[b]{.245\textwidth}
\includegraphics[width=\textwidth]{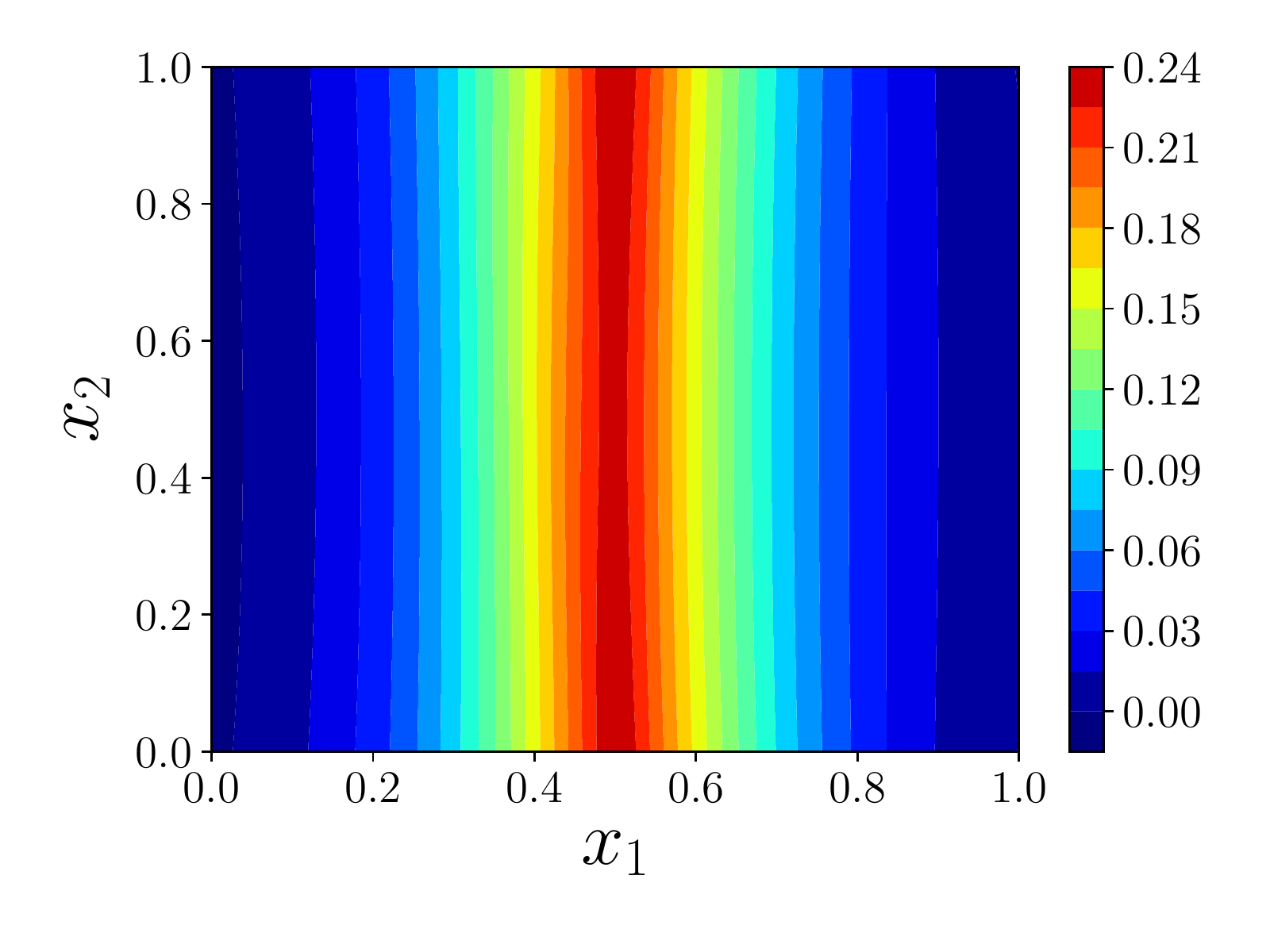}
\caption{$u_{\text{WAN}}$}
\label{subfig:nonsmooth_u_WAN}
\end{subfigure} \\
\begin{subfigure}[b]{.245\textwidth}
\includegraphics[width=\textwidth]{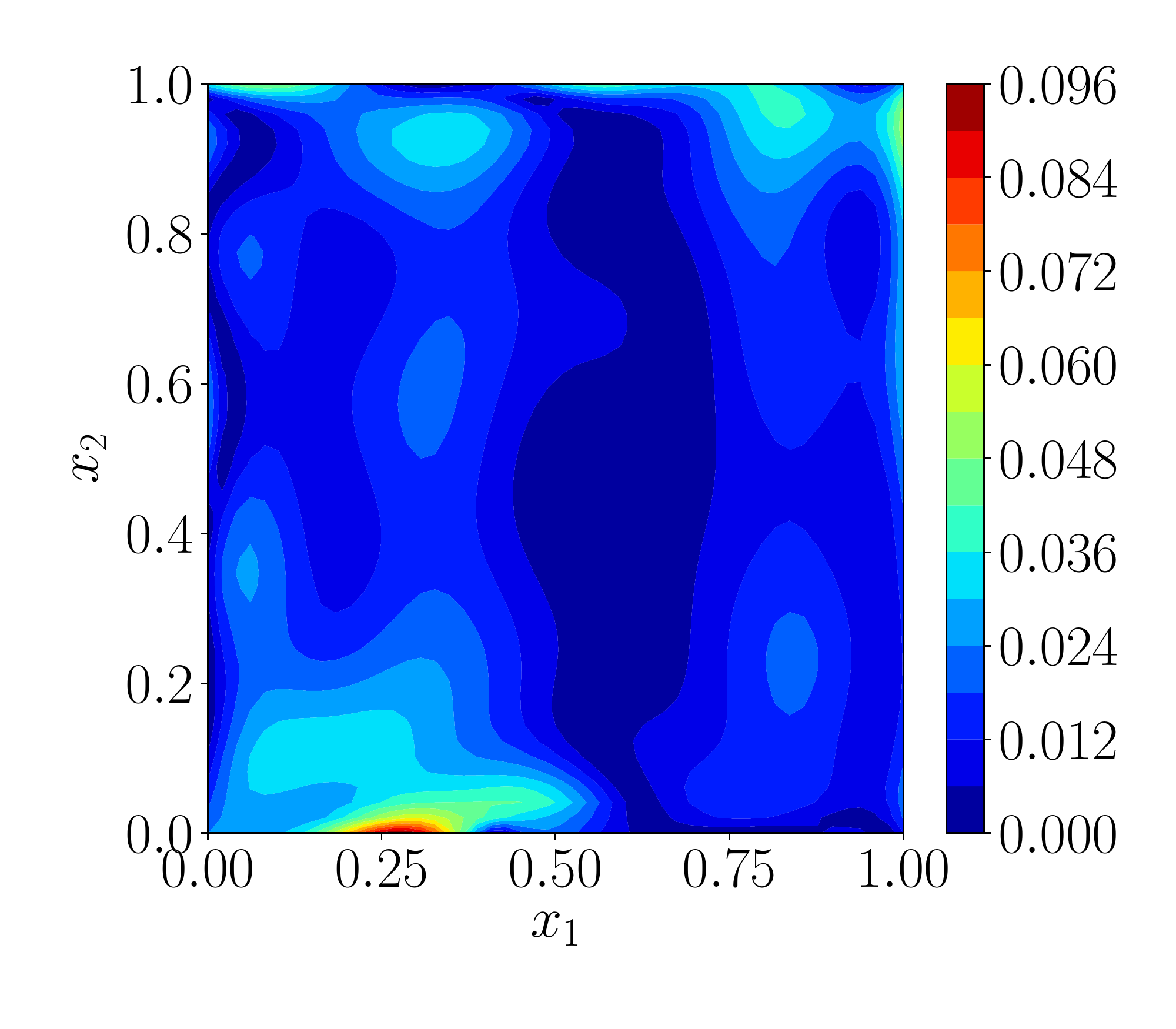}
\caption{$|-\Delta u_{\text{PINN}}-f|$}
\label{subfig:nonsmooth_f}
\end{subfigure}
\begin{subfigure}[b]{.245\textwidth}
\includegraphics[width=\textwidth]{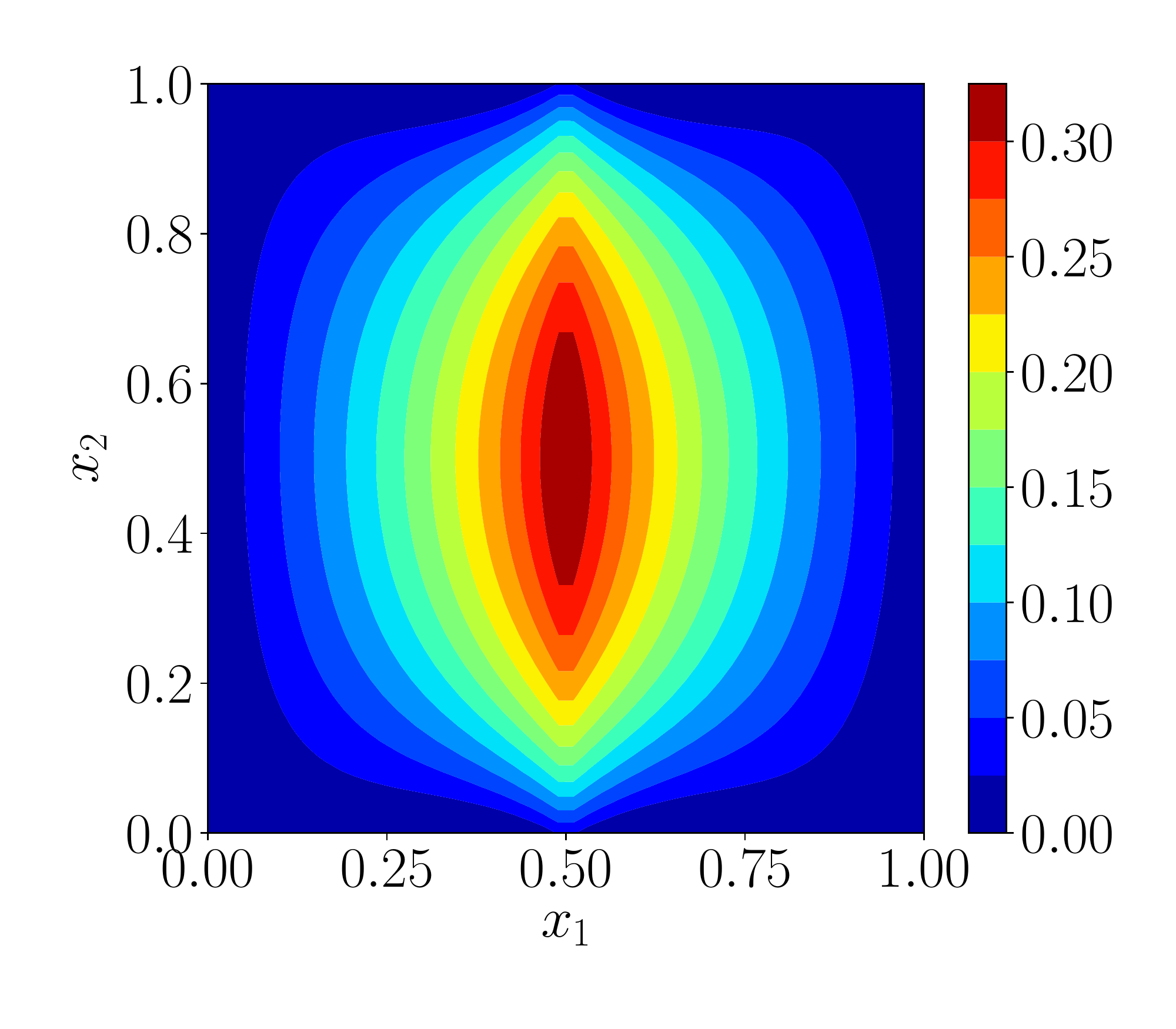}
\caption{$|u_{\text{PINN}}-u^*|$}
\label{subfig:nonsmooth_abs_PINN}
\end{subfigure}
\begin{subfigure}[b]{.245\textwidth}
\includegraphics[width=\textwidth]{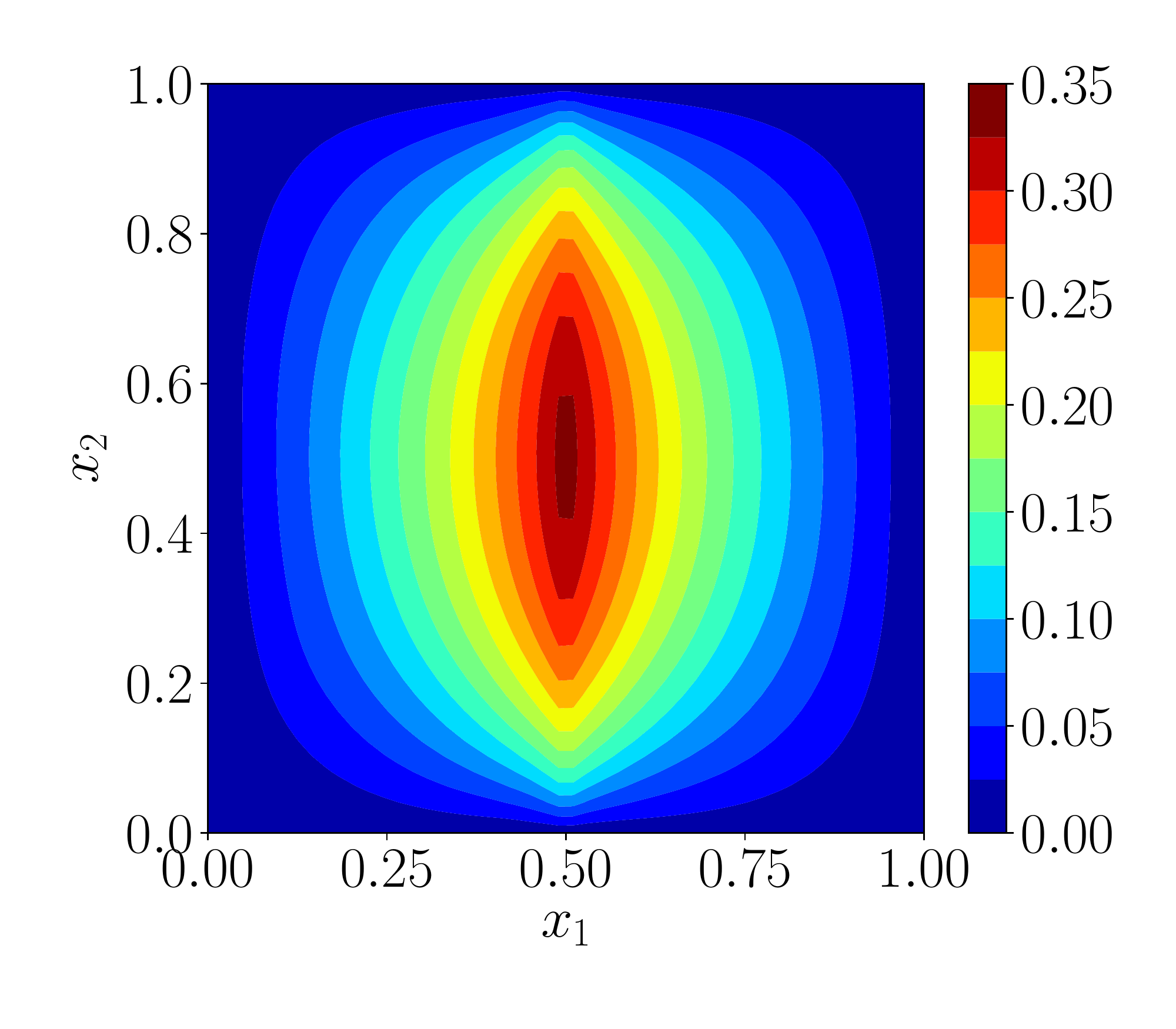}
\caption{$|u_{\text{DRM}}-u^*|$}
\label{subfig:nonsmooth_abs_DRM}
\end{subfigure}
\begin{subfigure}[b]{.245\textwidth}
\includegraphics[width=\textwidth]{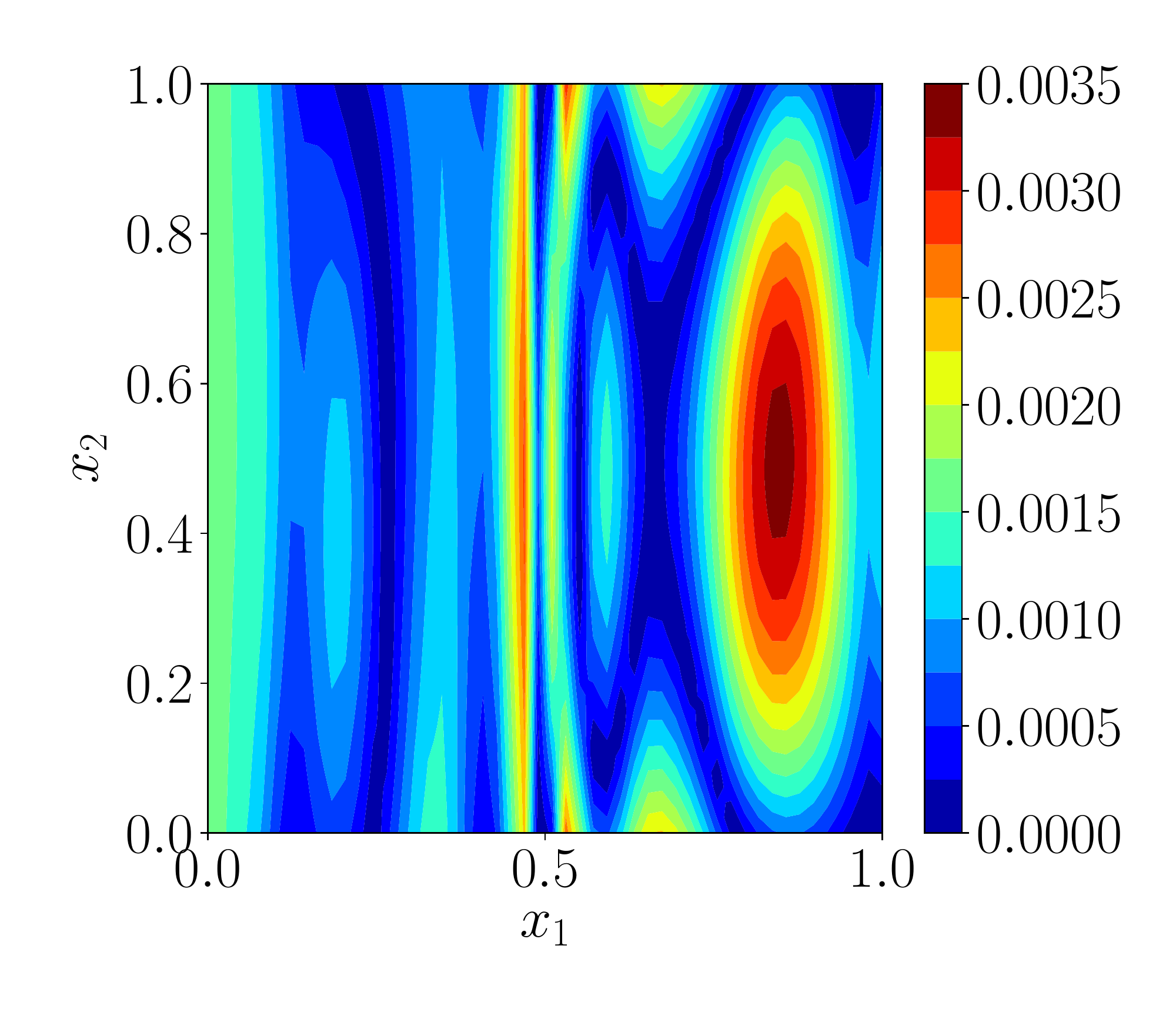}
\caption{$|u_{\text{WAN}}-u^*|$}
\label{subfig:nonsmooth_abs_WAN}
\end{subfigure} \\
\begin{subfigure}[b]{.245\textwidth}
\includegraphics[width=\textwidth]{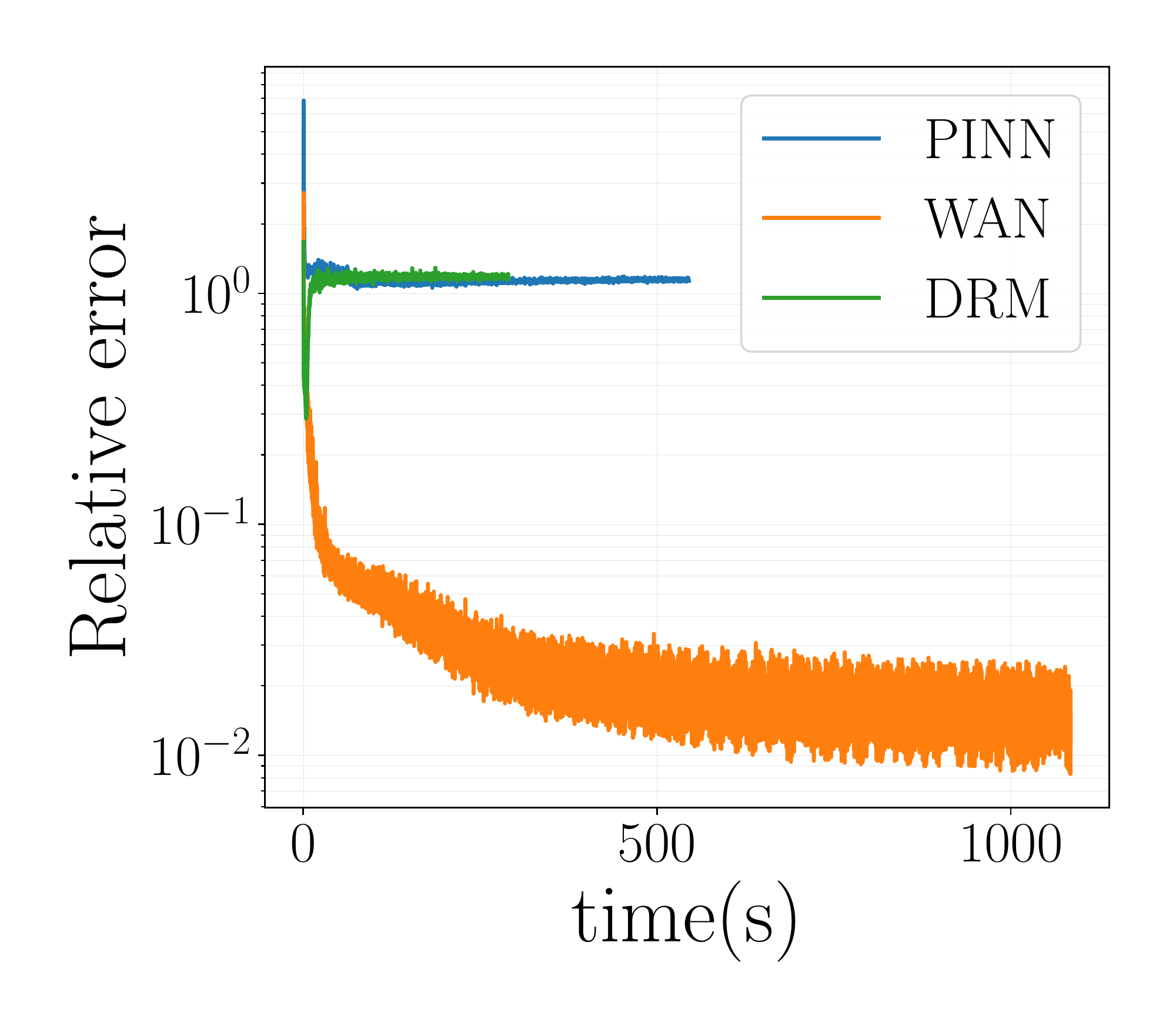}
\caption{Error vs time (s)}
\label{subfig:nonsmooth_error_time}
\end{subfigure}
\begin{subfigure}[b]{.245\textwidth}
\includegraphics[width=\textwidth]{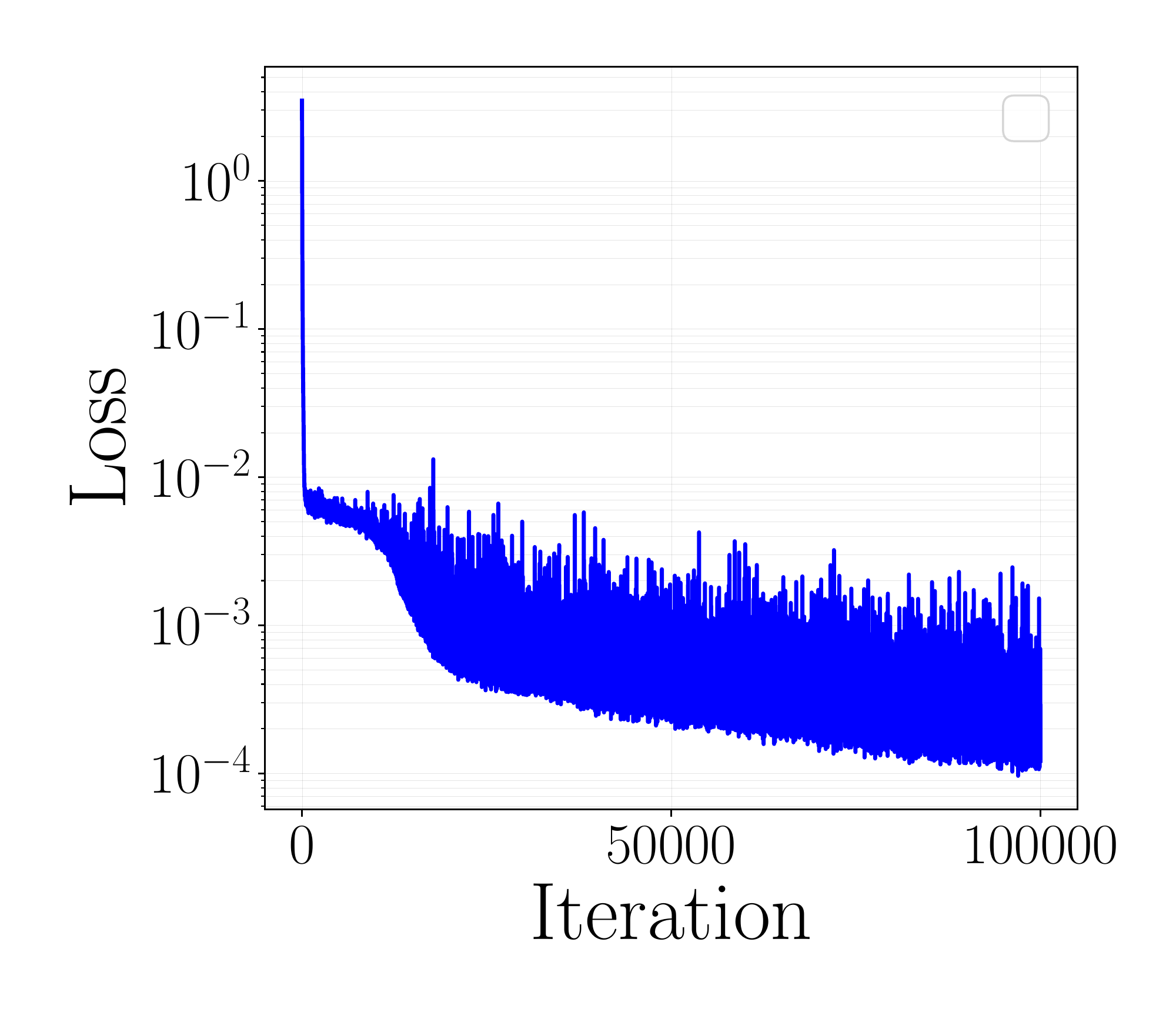}
\caption{Loss vs iter by PINN}
\label{subfig:nonsmooth_loss_PINN}
\end{subfigure}
\begin{subfigure}[b]{.245\textwidth}
\includegraphics[width=\textwidth]{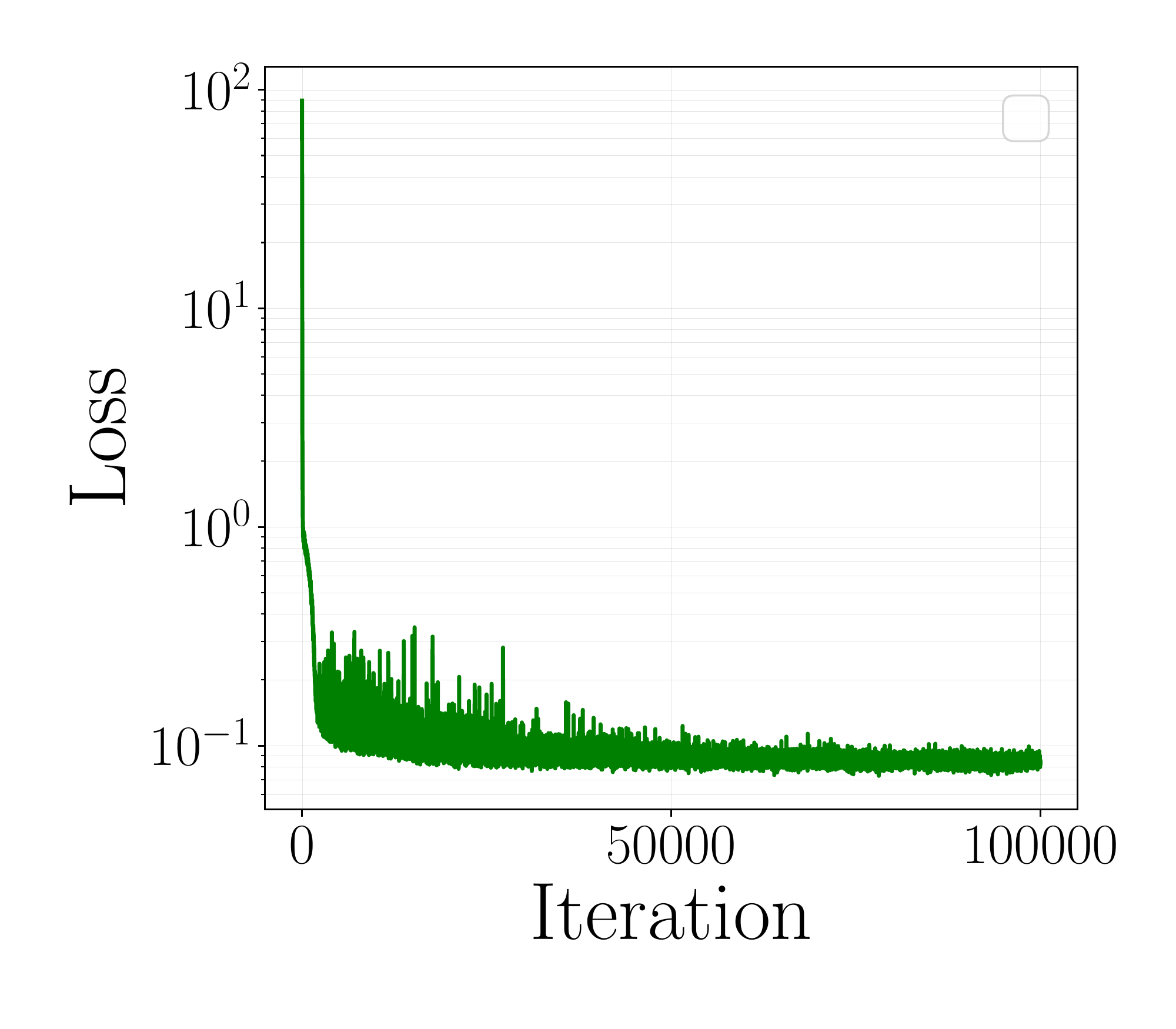}
\caption{Loss vs iter by DRM}
\label{subfig:nonsmooth_loss_DRM}
\end{subfigure}
\begin{subfigure}[b]{.245\textwidth}
\includegraphics[width=\textwidth]{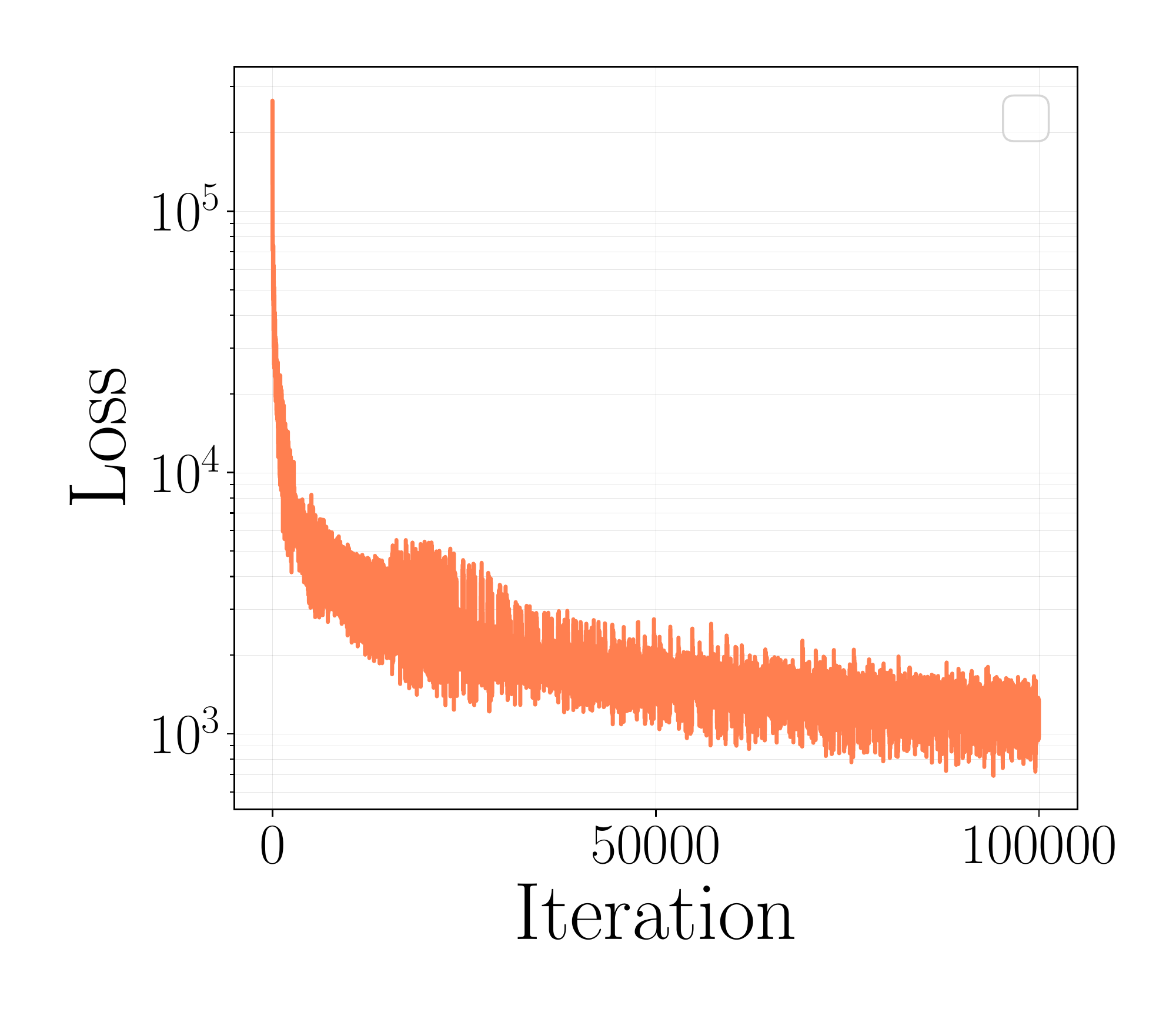}
\caption{Loss vs iter by WAN}
\label{subfig:nonsmooth_loss_WAN}
\end{subfigure}
\caption{Results of \eqref{eq:eq-weak}.
(a) The true solution $u^*$;
(b) $u_{\text{PINN}}$ obtained by PINN \cite{raissi2019physics} based on the strong form;
(c) $u_{\text{DRM}}$ obtained by DRM \cite{weinan2018deep};
(d) $u_{\text{WAN}}$ obtained by the proposed algorithm \ref{alg:wan} WAN based on the weak form;
(e) $|-\Delta u_{\text{PINN}}-f|$ by PINN;
(f), (g), (h) Pointwise absolute error $|u-u^*|$ by PINN, DRM, and WAN respectively;
(i) Relative error versus computation time (in seconds) of PINN, DRM, and WAN;
(j), (k), (l) the objective function versus iteration number of PINN, DRM, and WAN respectively.}
\label{fig:continue}
\end{figure}

\subsubsection{High dimensional smooth problem}
We also check how competitive the proposed Algorithm \ref{alg:wan} WAN is on a simple smooth problem that admits strong solution.
We again consider a Poisson equation with Dirichlet boundary condition:
\begin{equation}\label{eq:smooth_poisson}
\begin{cases}
-\Delta u = \frac{\pi^2}{4}\sum^{d}_{i=1}\sin(\frac{\pi}{2} x_i), &\quad \text{in}\ \Omega= (0,1)^d \\
u= \sum^{d}_{i=1}\sin(\frac{\pi}{2} x_i), &\quad \text{on}\  \partial\Omega \\
\end{cases}
\end{equation}
with problem dimension $d=5$.
Problem \eqref{eq:smooth_poisson} has a strong (classical) solution $u^*(x)=\sum^{d}_{i=1}\sin(\frac{\pi}{2} x_i)$ in $\bar{\Omega}$.
For comparison, we also apply the state-of-the-art methods PINN \cite{raissi2019physics} and DRM \cite{weinan2018deep} which are particularly suitable for such problems.
For sake of fair comparison, in WAN, we set $K_\varphi =1, K_u = 1$, $\tau_\eta = 0.015, \tau_\theta = 0.001$, $N_r=10^4$, $N_b = 2d\times 30$ ($30$ points uniformly sampled on each of the $2d$ sides of $\Omega$), and $\alpha=10,000$.
We use Adam optimizer for updating $\theta$ and AdaGrad optimizer for $\eta$.
For the PINN and the DRM, we follow the same parameter settings as in Section \ref{subsubsec:weak_vs_strong}, and use the same $N_b$ and $N_r$ as WAN here.
We run all methods for $20,000$ iterations, and show the pointwise absolute error $|u-u^*|$ in Figure \ref{subfig:smooth_abserr}.
We also show their change process of relative error versus time (s) in Figure \ref{subfig:smooth_error_time} and the objective functions versus iteration number in Figure \ref{subfig:smooth_loss_iter}.
The objective functions are computed in the same way as above.
From Figure \ref{subfig:smooth_loss_iter}, we can see that all methods have almost reached their limit after 20,000 iterations as the objective functions do not have substantial improvements.
This can also be seen from Figure \ref{subfig:smooth_error_time}, where the relative errors of all methods are lower than 1\% which means their results are very close to the true solution $u^*$.
Moreover, we can see that PINN obtains the highest accuracy after 20,000 iterations, closely followed by WAN.
DRM also obtains a satisfactory accuracy level with the fastest convergence.
It is worth noting that PINN requires computation of $\Delta u$ (i.e., second-order partial derivatives of the deep network $u$ with respect to its input $x$) which is a bit more expensive than that of first-order gradient $\nabla u$.
\begin{figure}[!ht]
\begin{subfigure}[b]{\textwidth}
\centering
\includegraphics[width=.8\linewidth]{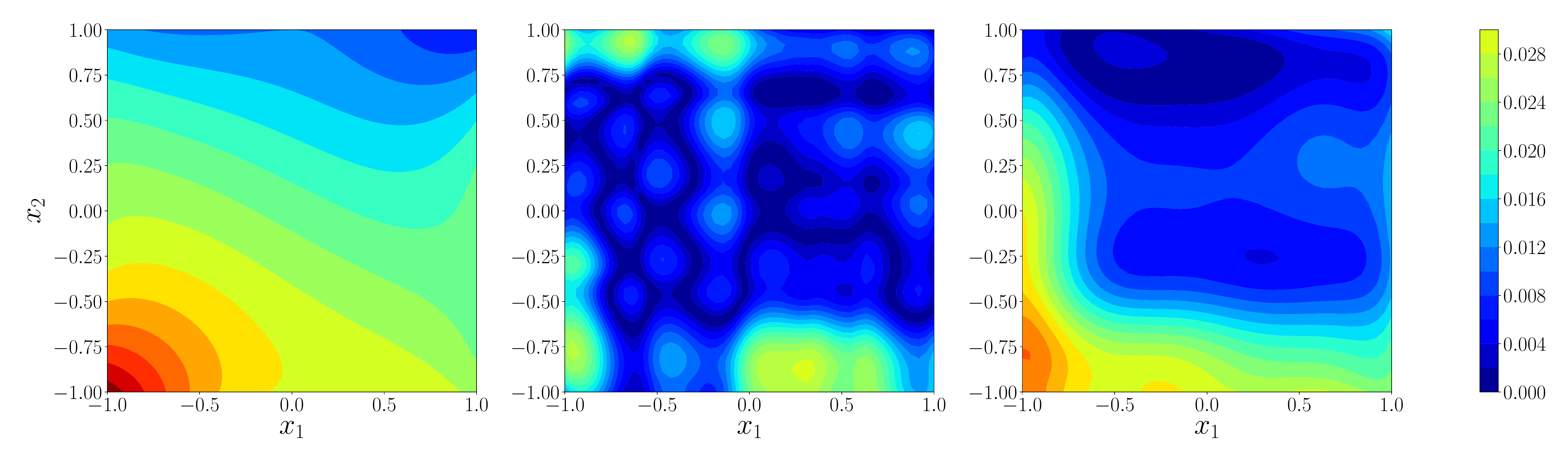}
\caption{$|u-u^*|$ with $u$ obtained by PINN (left), DRM (middle), and WAN (right).}
\label{subfig:smooth_abserr}
\end{subfigure}
\begin{subfigure}[b]{0.24\textwidth}
\centering
\includegraphics[width=\textwidth]{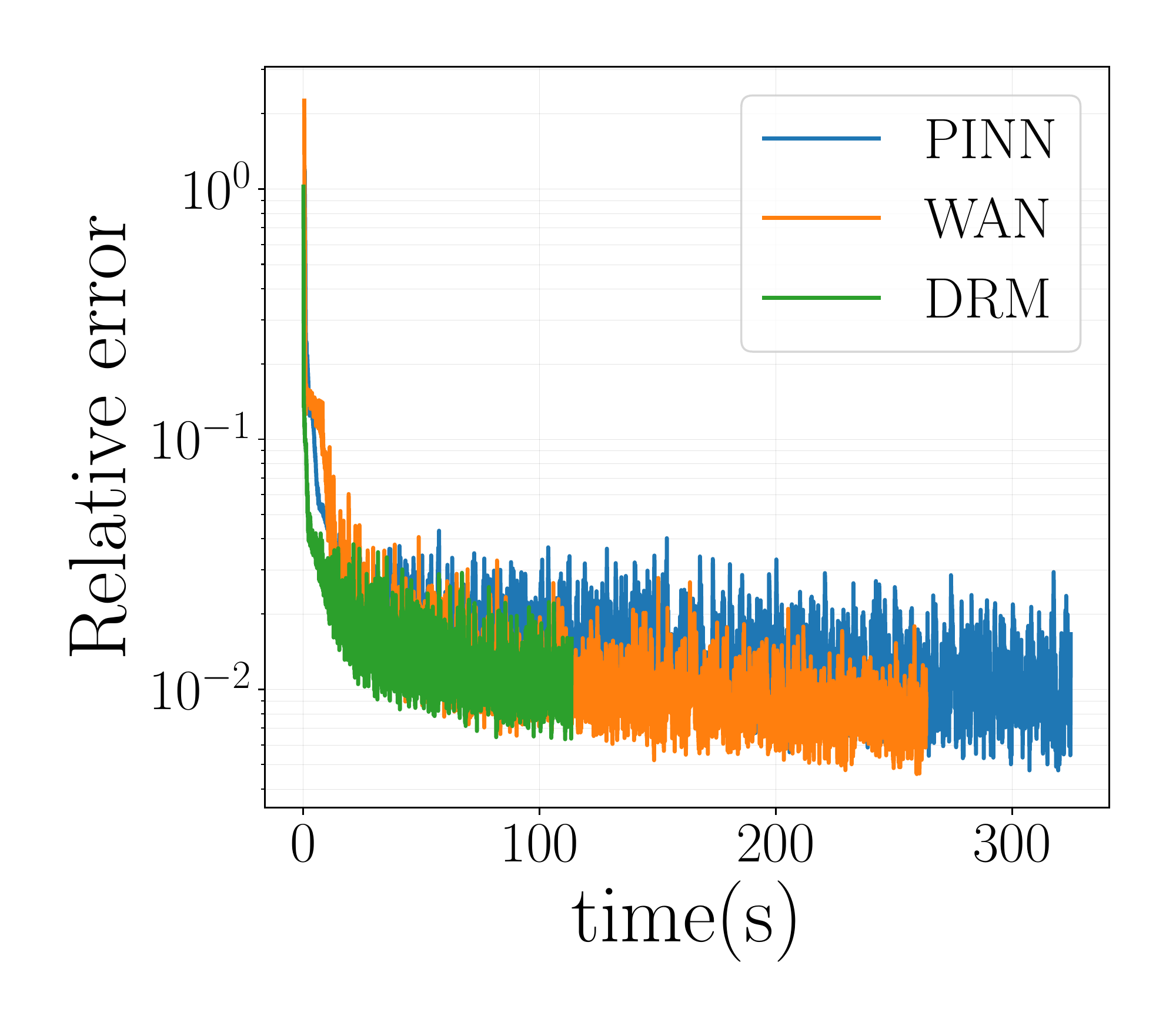}
\caption{Error vs time (s).}
\label{subfig:smooth_error_time}
\end{subfigure}
\begin{subfigure}[b]{.74\textwidth}
\includegraphics[width=.32\textwidth]{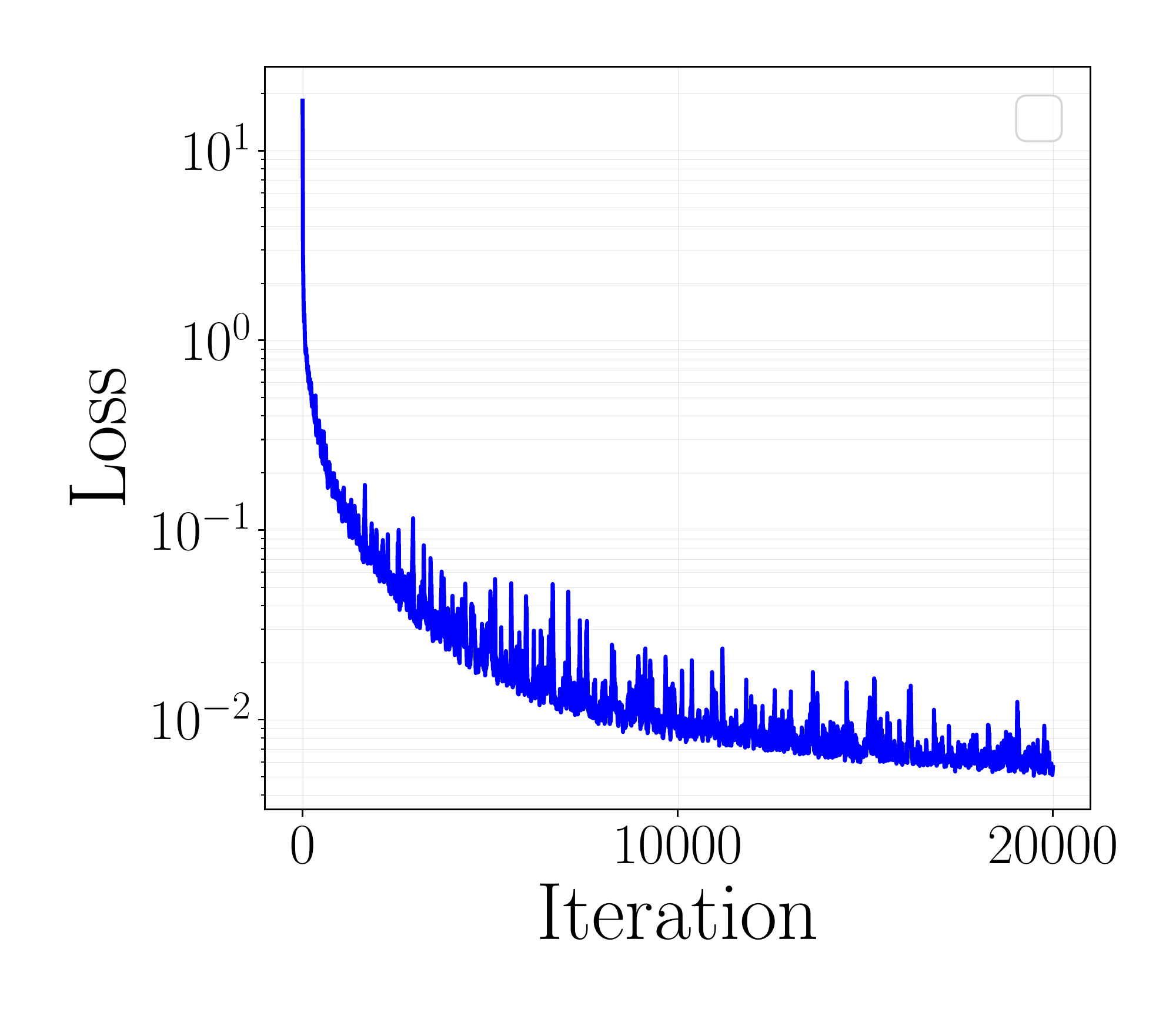}
\includegraphics[width=.32\textwidth]{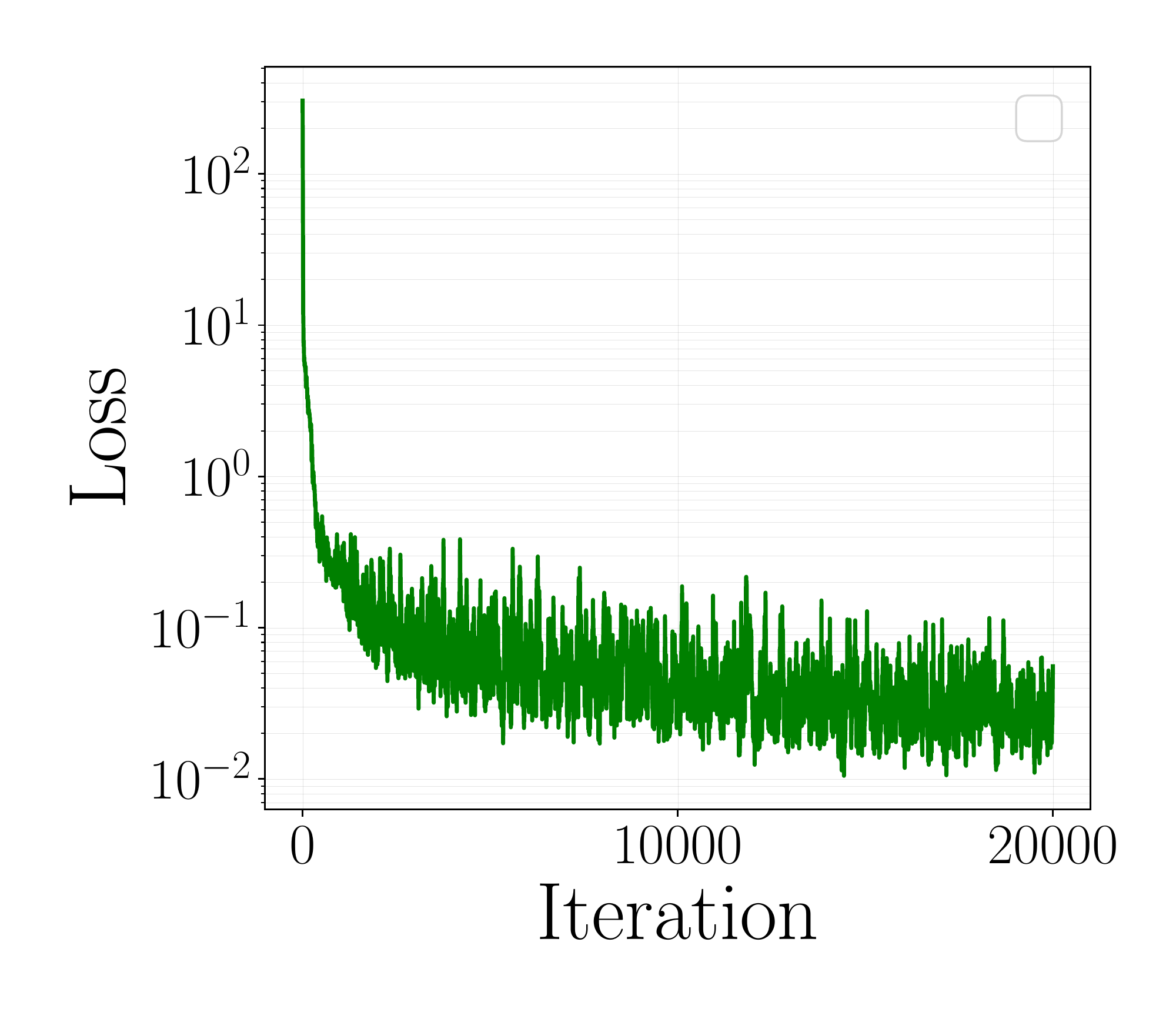}
\includegraphics[width=.32\textwidth]{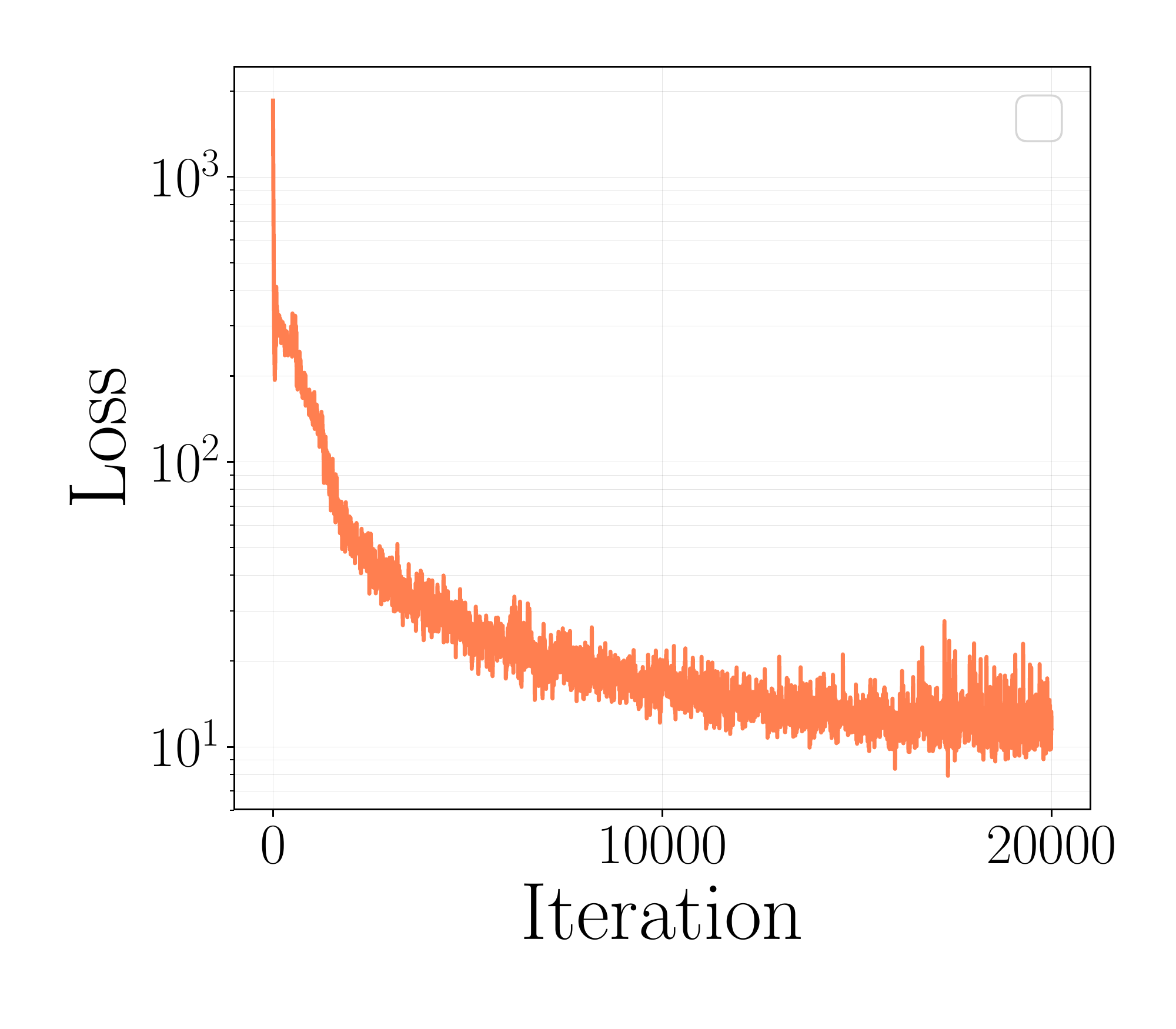}
\caption{Loss vs iteration by PINN (left), DRM (middle), and WAN (right).}
\label{subfig:smooth_loss_iter}
\end{subfigure}
\caption{Comparison of the solutions to \eqref{eq:smooth_poisson} based on the PINN, DRM, and WAN.
(a) The pointwise absolute error of $|u-u^*|$ with $u$ obtained by PINN (left), DRM (middle), and the proposed WAN (right).
(b) The change process of relative error versus time in seconds.
(c) The change process of the objective function versus iteration number by PINN (left), DRM (middle), and the proposed WAN (right).}
\label{fig:strong_solution_test}
\end{figure}
%%%%%%%%%%%%%%%%%%%%%%%%%%%%%%%
%\begin{figure}[!ht]
%
%\begin{subfigure}[b]{\textwidth}
%\centering
%\includegraphics[width=.8\linewidth]{fig/example_2/abs_example_1step_5d.pdf}
%\caption{$|u-u^*|$ with $u$ obtained by PINN (left), DRM (middle), and WAN (right).}
%\label{subfig:smooth_abserr_1step}
%\end{subfigure}
%
%\begin{subfigure}[b]{0.24\textwidth}
%\centering
%\includegraphics[width=\textwidth]{fig/example_2/time_example_1step_list_0d.pdf}
%\caption{Error vs time (s).}
%\label{subfig:smooth_error_time_1step}
%\end{subfigure}
%
%\begin{subfigure}[b]{.74\textwidth}
%\includegraphics[width=.32\textwidth]{fig/example_2/loss_example_1step_PINN_5d.pdf}
%\includegraphics[width=.32\textwidth]{fig/example_2/loss_example_1step_ritz_5d.pdf}
%\includegraphics[width=.32\textwidth]{fig/example_2/loss_example_1step_WAN_5d.pdf}
%\caption{Error vs iteration by PINN (left), DRM (middle), and WAN (right).}
%\label{subfig:smooth_loss_iter_1step}
%\end{subfigure}
%
%\caption{Comparison of the solutions to \eqref{eq:smooth_poisson} based on the PINN, DRM, and WAN. (\zang{One step of gradient descent in each iteration for the PINN and the DRM.})
%(a) The pointwise absolute error of $|u-u^*|$ with $u$ obtained by PINN (left), DRM (middle), and the proposed WAN (right).
%(b) Progress of relative error versus time in seconds.
%(c) Progress of objective function versus iteration number by PINN (left), DRM (middle), and the proposed WAN (right).}
%\label{fig:strong_solution_test_1step}
%\end{figure}

\subsubsection{High-dimensional nonlinear elliptic PDEs with Dirichlet boundary condition}
\label{subsubsec:nonlinear}
In this example, we apply Algorithm \ref{alg:wan} to a \textit{nonlinear} elliptic PDEs with Dirichlet boundary condition. We test the problem with different dimensions $d=5,10,15,20,25$ as follows,
\begin{equation}\label{eq:nonl_cube}
\begin{cases}
-\nabla\cdot (a(x)\nabla u )+\frac{1}{2}|\nabla u|^2 = f(x) \quad &\text{in}\ \ \Omega \triangleq (-1,1)^d,\\
u(x) = g(x) \quad &\text{on}\ \ \partial\Omega
\end{cases}
\end{equation}
where $a(x)= 1+|x|^2$ in $\Omega$, $f(x)=4 \rho_1^2 (1+|x|^2)\sin{\rho_0^2}-4\rho_0^2\cos{(\rho_0^2)}-(\pi+1)(1+|x|^2)\cos{(\rho_0^2)}+2\rho_1^2\cos^2(\rho_0^2)$ in $\bar{\Omega}$, and $g(x)=\sin(\frac{\pi}{2} x_1^2+\frac{1}{2}x_2^2)$ on $\partial \Omega$, with $\rho_0^2 \triangleq \frac{\pi}{2} x_1^2+\frac{1}{2}x_2^2, \rho_1^2 \triangleq \frac{\pi^2}{4} x_1^2+\frac{1}{4}x_2^2$.
The exact solution of \eqref{eq:nonl_cube} is $u^*(x)=\sin(\frac{\pi}{2} x_1^2+\frac{1}{2}x_2^2)$ in $\Omega$, the cross section $(x_1,x_2)$ of which is shown in the left panel of Figure \ref{subfig:nonlinear_u}.
In this test, we set $K_{\varphi}=1$, $K_u = 2$, $\tau_{\eta}=0.04$, $\tau_{\theta}=0.015$, $N_r=4,000d$, $N_b=40d^2$, and $\alpha=10,000\times N_b$ for $d=5,10$, and $20,000\times N_b$ for $d=15,20$, and $25,000\times N_b$ for $d=25$.
The solution $u_\theta$ after 20,000 iterations for $d=20$ case is shown in the right panel of Figure \ref{subfig:nonlinear_u}, and the absolute pointwise error $|u_\theta-u^*|$ is shown in Figure \ref{subfig:nonlinear_abserr}.
We show the progresses of the relative error versus iteration in Figure \ref{subfig:nonlinear_error_iter}.
After 20,000 iterations, the relative error reaches $0.44\%, 0.62\%, 0.52\%, 0.66\%, 0.69\%$ for $d=5,10,15,20,25$ cases, respectively.
As we can see, the Algorithm \ref{alg:wan} can solve high-dimensional nonlinear PDEs accurately.
\begin{figure}[!t]
\centering
\begin{subfigure}[b]{.46\textwidth}
\includegraphics[width=\textwidth]{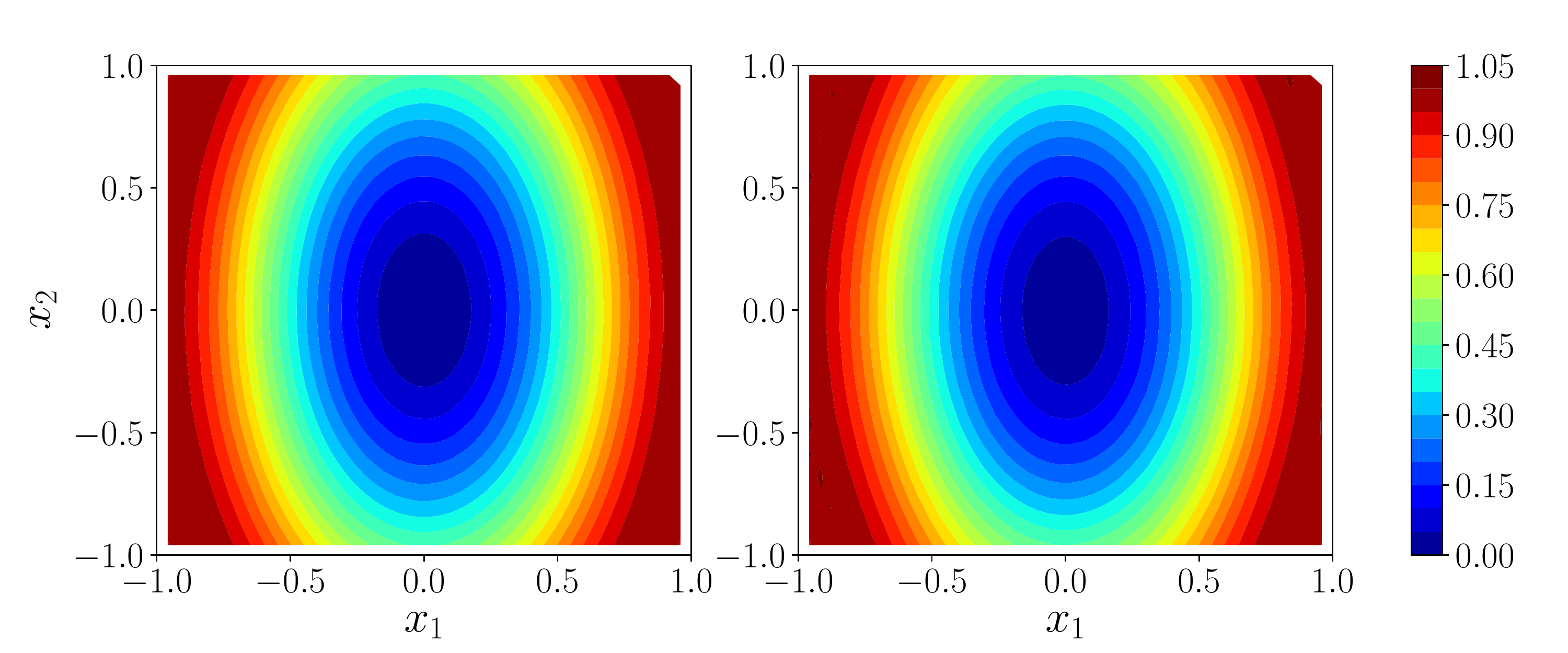}
\caption{True $u^*$ (left) vs estimated $u_\theta$ (right)}
\label{subfig:nonlinear_u}
\end{subfigure}
\begin{subfigure}[b]{.245\textwidth}
\includegraphics[width=\textwidth]{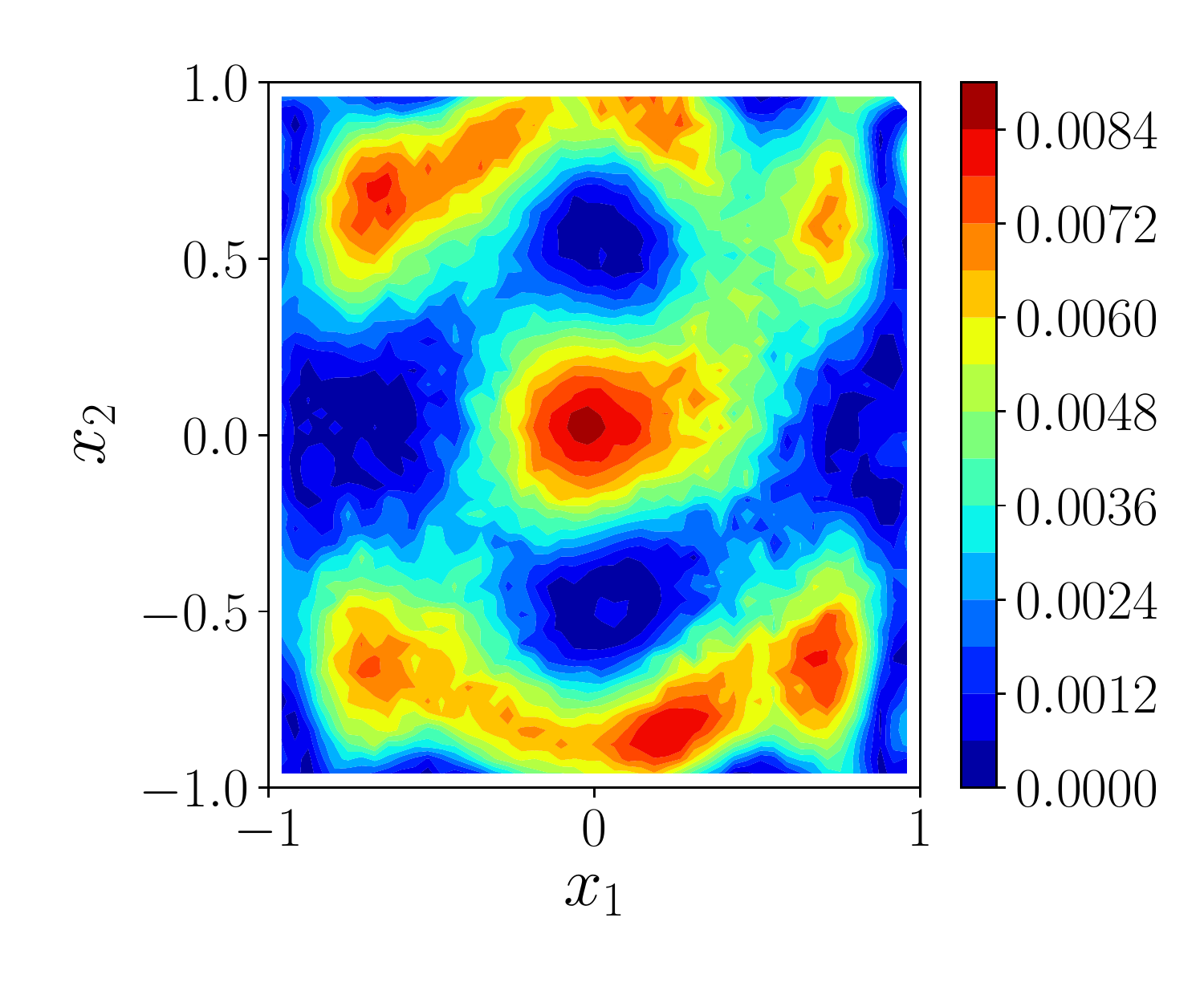} \vspace{-20pt}
\caption{$|u_\theta - u^*|$}
\label{subfig:nonlinear_abserr}
\end{subfigure}
\begin{subfigure}[b]{.23\textwidth}
\includegraphics[width=\textwidth]{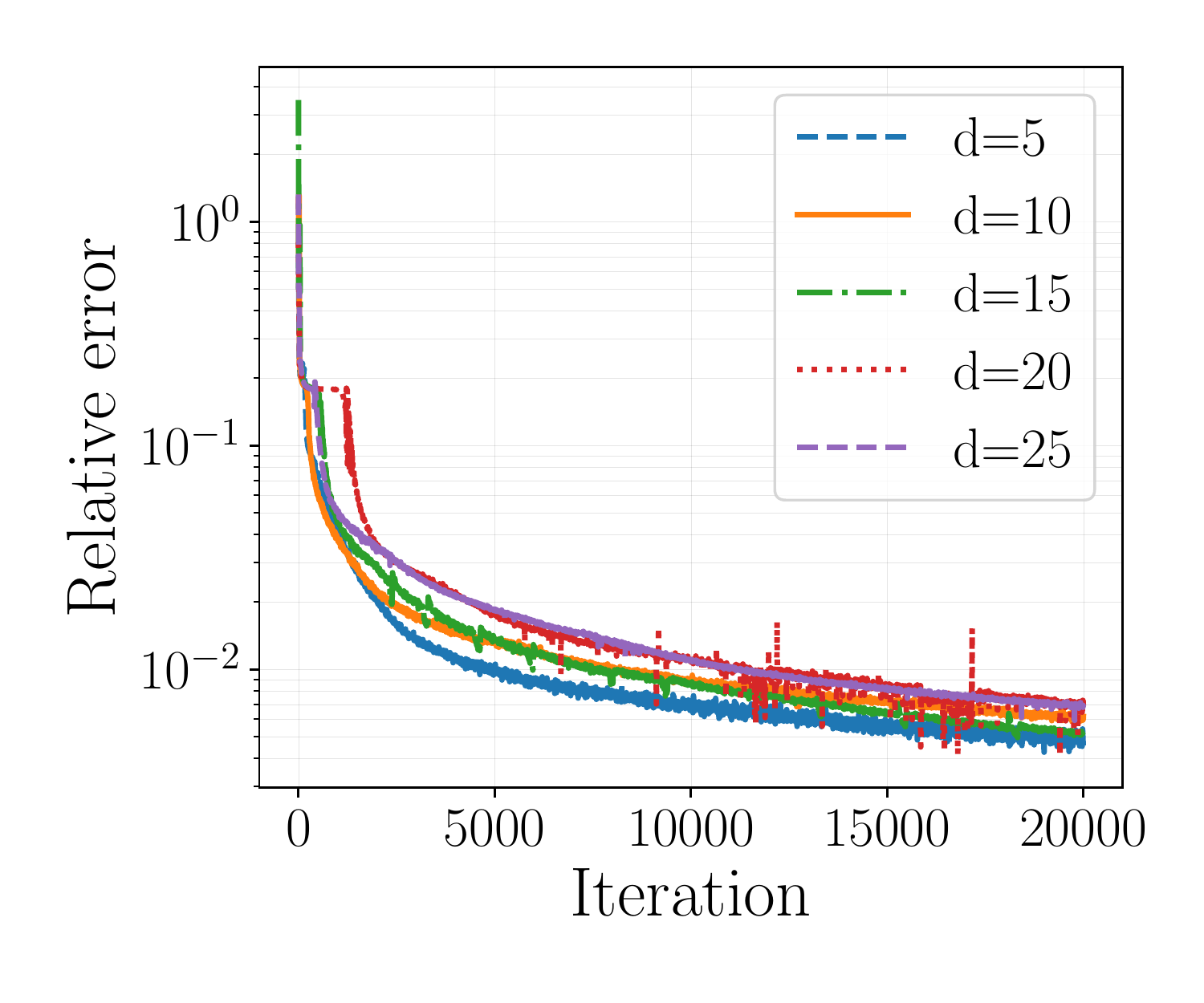} \vspace{-18pt}
\caption{Error vs iteration}
\label{subfig:nonlinear_error_iter}
\end{subfigure}
\caption{Result for the BVP \eqref{eq:nonl_cube} with nonlinear elliptical PDE and Dirichlet boundary condition. (a) True solution $u^*$ and the approximation $u_\theta$ obtained by Algorithm \ref{alg:wan} after 20,000 iterations for $d=20$;
(b) The absolute difference $|u_\theta-u^*|$ for $d=20$;
(c) Relative errors versus iteration numbers for $d=5,10,15,20,25$ cases. For display purpose, images (a) and (b) only show the slices of $x_3=\dots=x_d=0$.}
\label{fig:nonlinear}
\begin{subfigure}[b]{.46\textwidth}
\includegraphics[width=\textwidth]{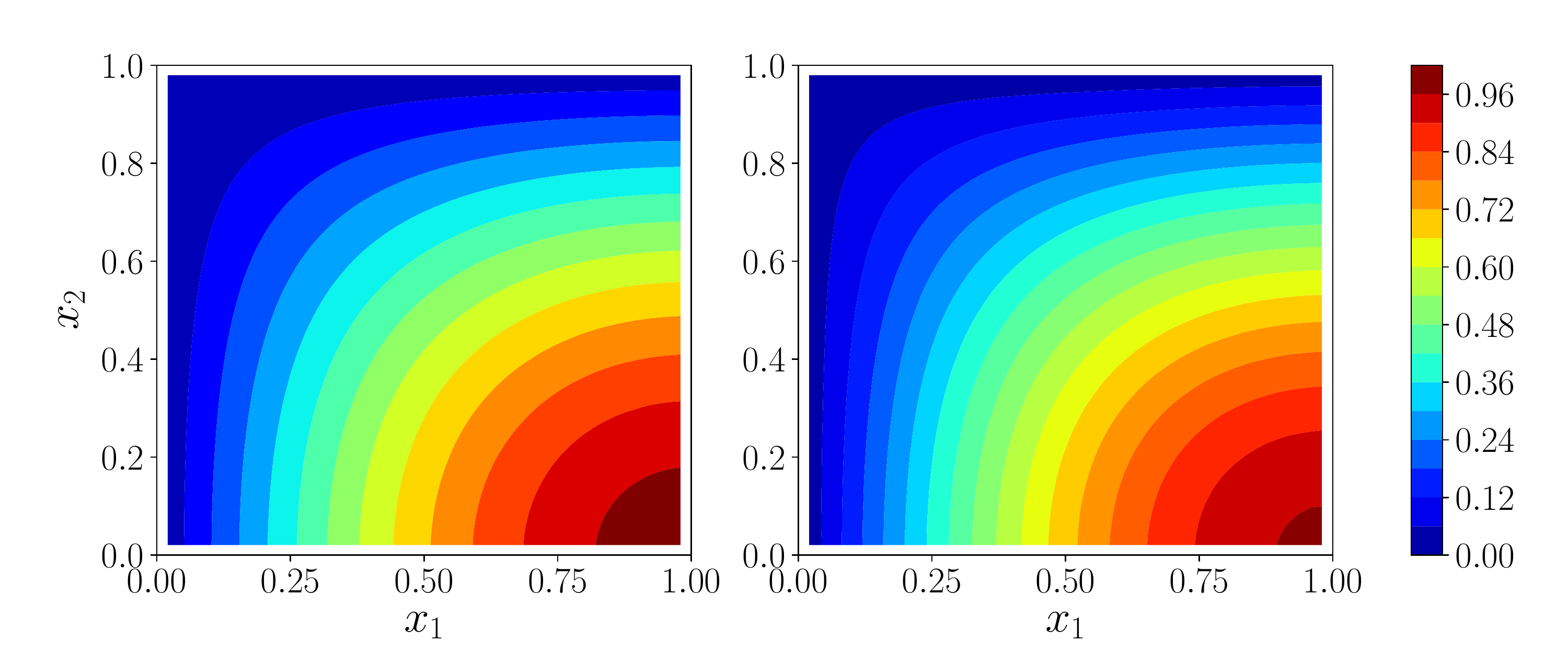}
\caption{True $u^*$ (left) vs estimated $u_\theta$ (right)}
\label{subfig:neumann_u}
\end{subfigure}
\begin{subfigure}[b]{.245\textwidth}
\includegraphics[width=\textwidth]{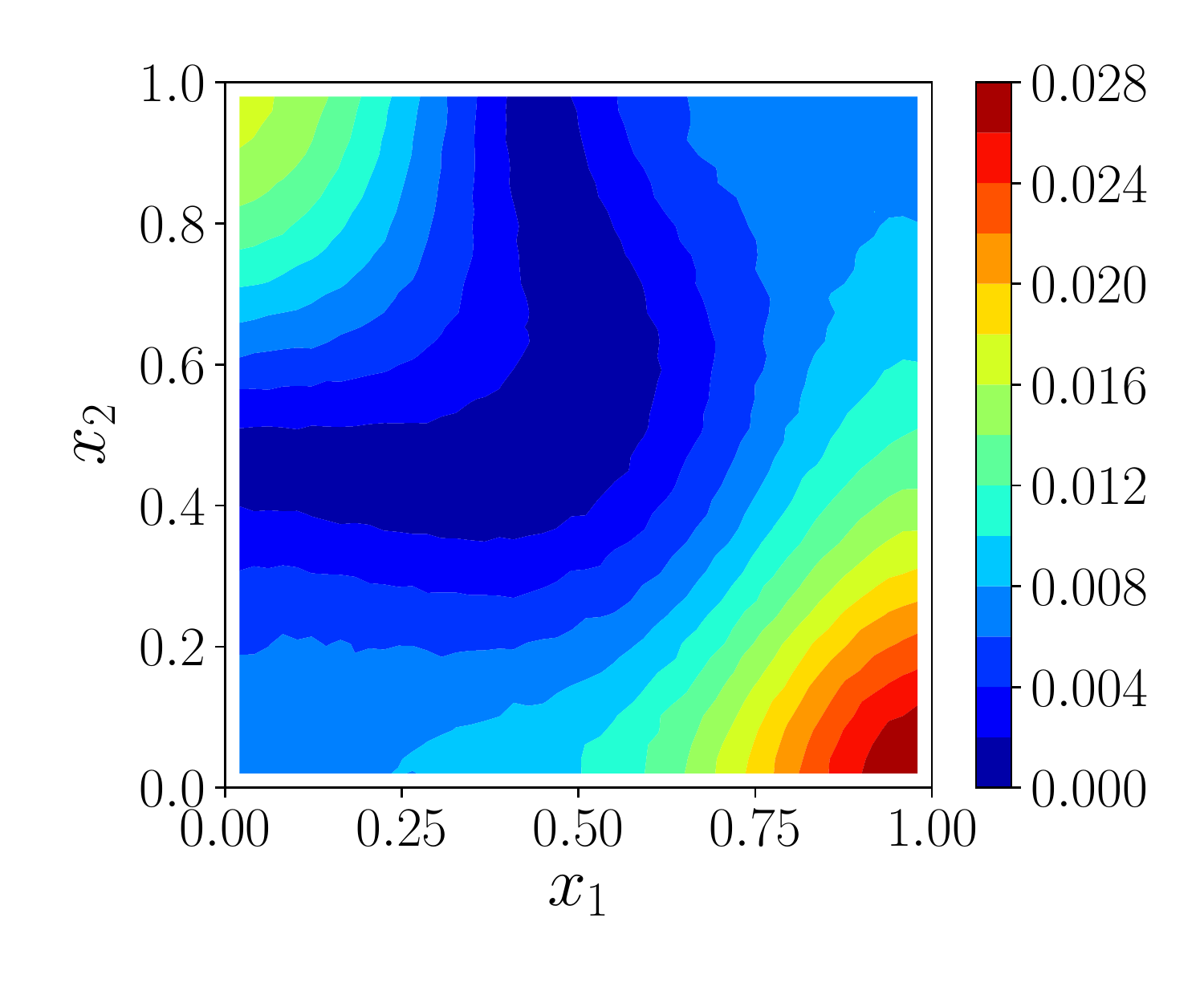} \vspace{-20pt}
\caption{$|u_\theta - u^*|$}
\label{subfig:neumann_abserr}
\end{subfigure}
\begin{subfigure}[b]{.23\textwidth}
\includegraphics[width=\textwidth]{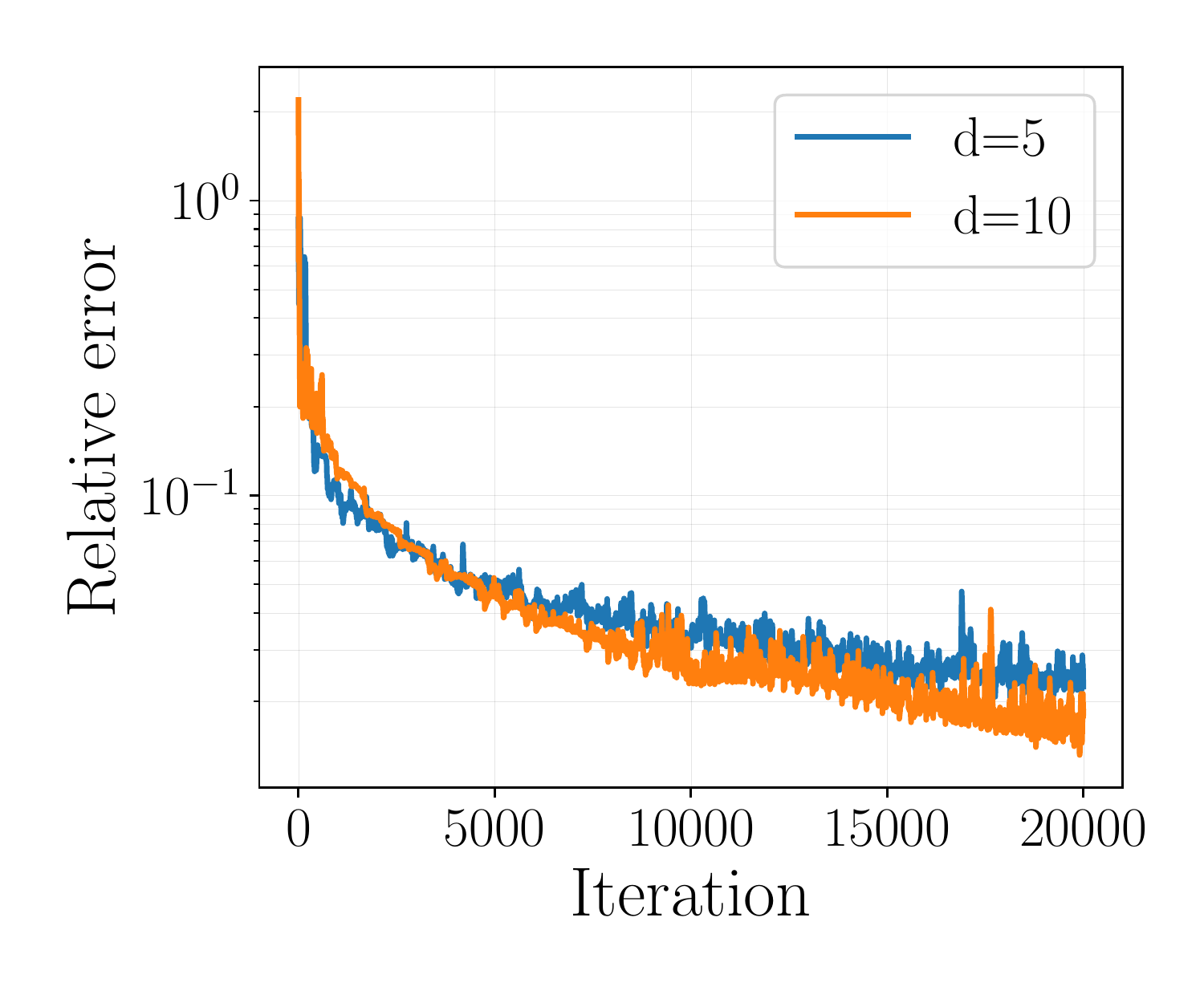} \vspace{-18pt}
\caption{Error vs iteration}
\label{subfig:neumann_error_iter}
\end{subfigure}
\caption{Result for the BVP \eqref{eq:neum_cube} with Neumann boundary condition. (a) True solution $u^*$ and the approximation $u_\theta$ obtained by Algorithm \ref{alg:wan} after 20,000 iterations for $d=10$;
(b) The absolute difference $|u_\theta-u^*|$ for $d=10$;
(c) Relative errors versus iteration numbers for $d=5,10$ cases. For display purpose, images (a) and (b) only show the slices of $x_3=\dots=x_d=0$.}
\label{fig:neumann}
\begin{subfigure}[b]{.46\textwidth}
\includegraphics[width=\textwidth]{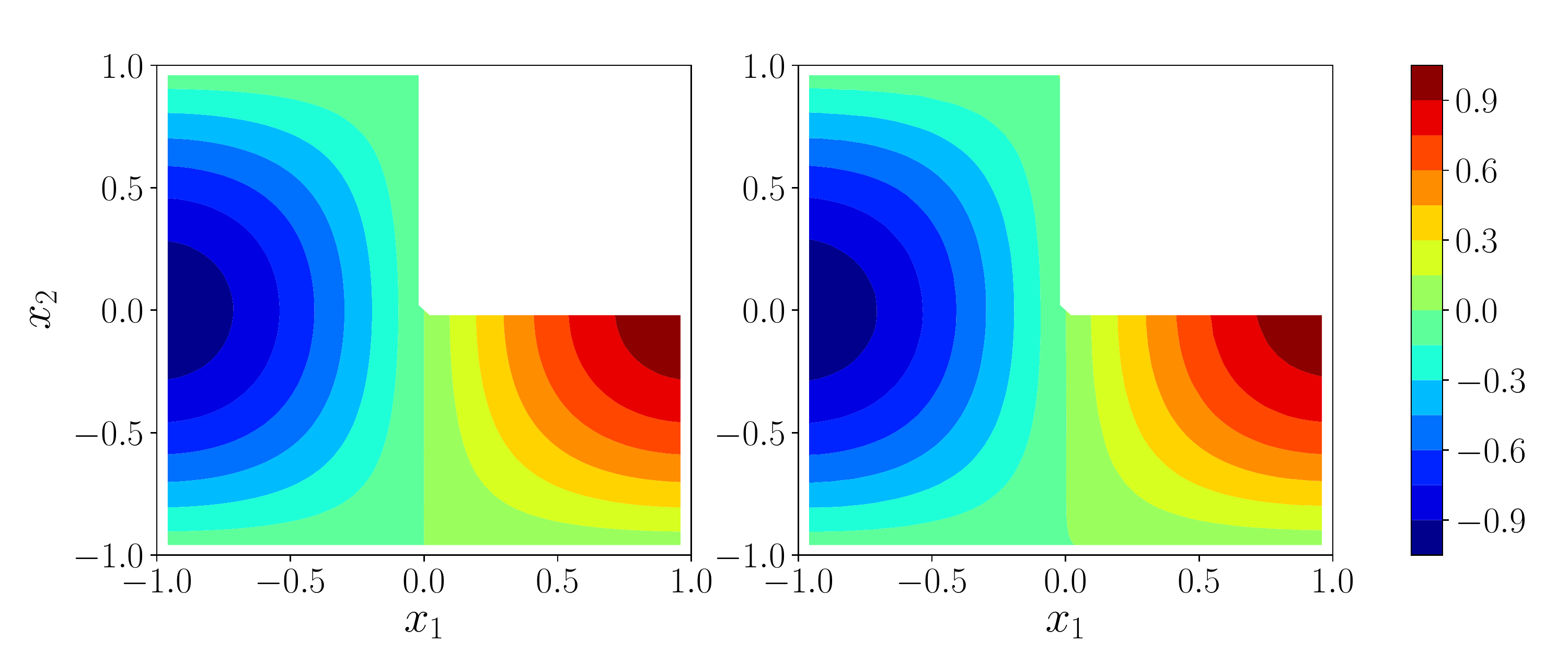}
\caption{True $u^*$ (left) vs estimated $u_\theta$ (right)}
\label{subfig:nonconvex_u}
\end{subfigure}
\begin{subfigure}[b]{.245\textwidth}
\includegraphics[width=\textwidth]{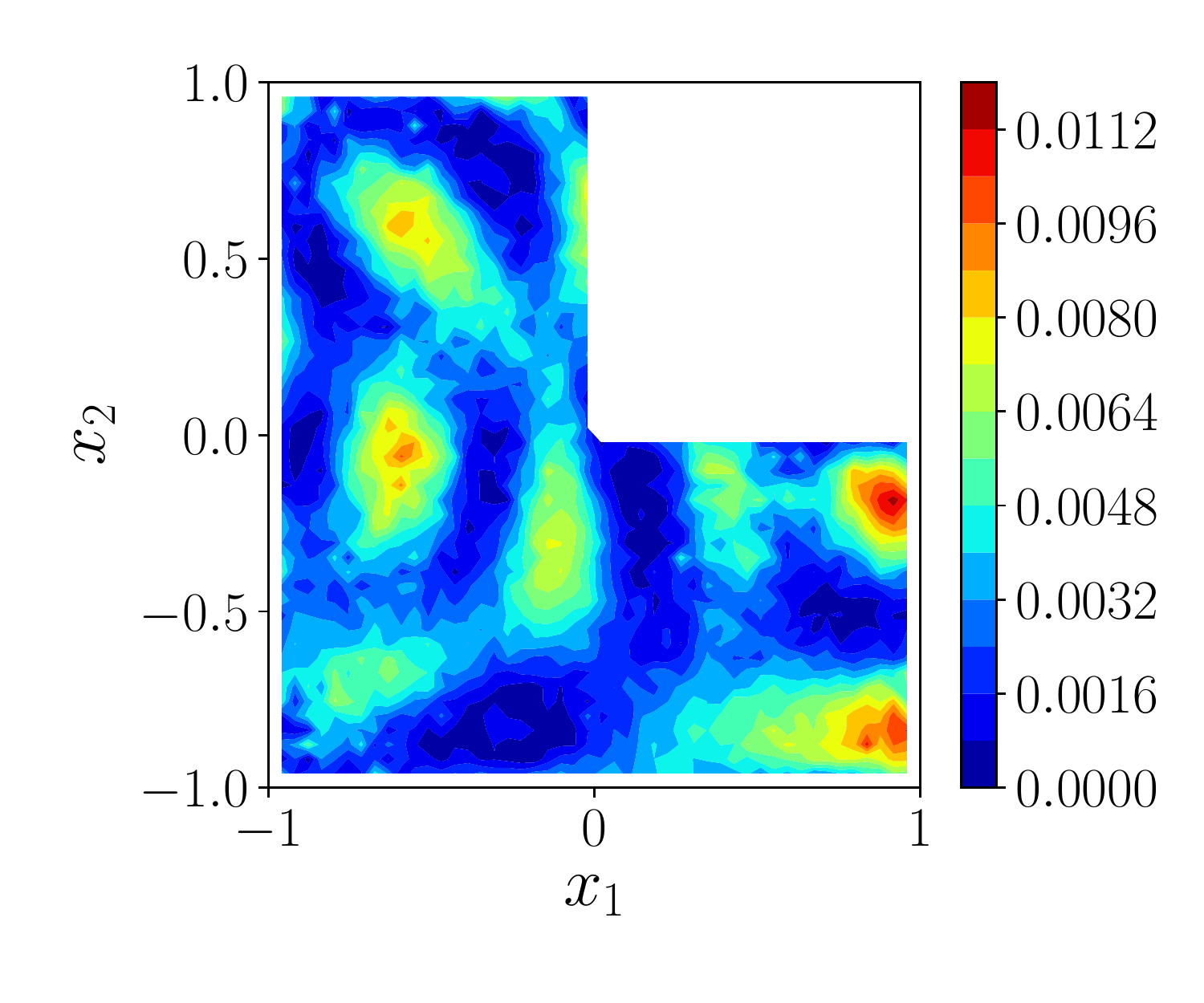} \vspace{-20pt}
\caption{$|u_\theta - u^*|$}
\label{subfig:nonconvex_abserr}
\end{subfigure}
\begin{subfigure}[b]{.23\textwidth}
\includegraphics[width=\textwidth]{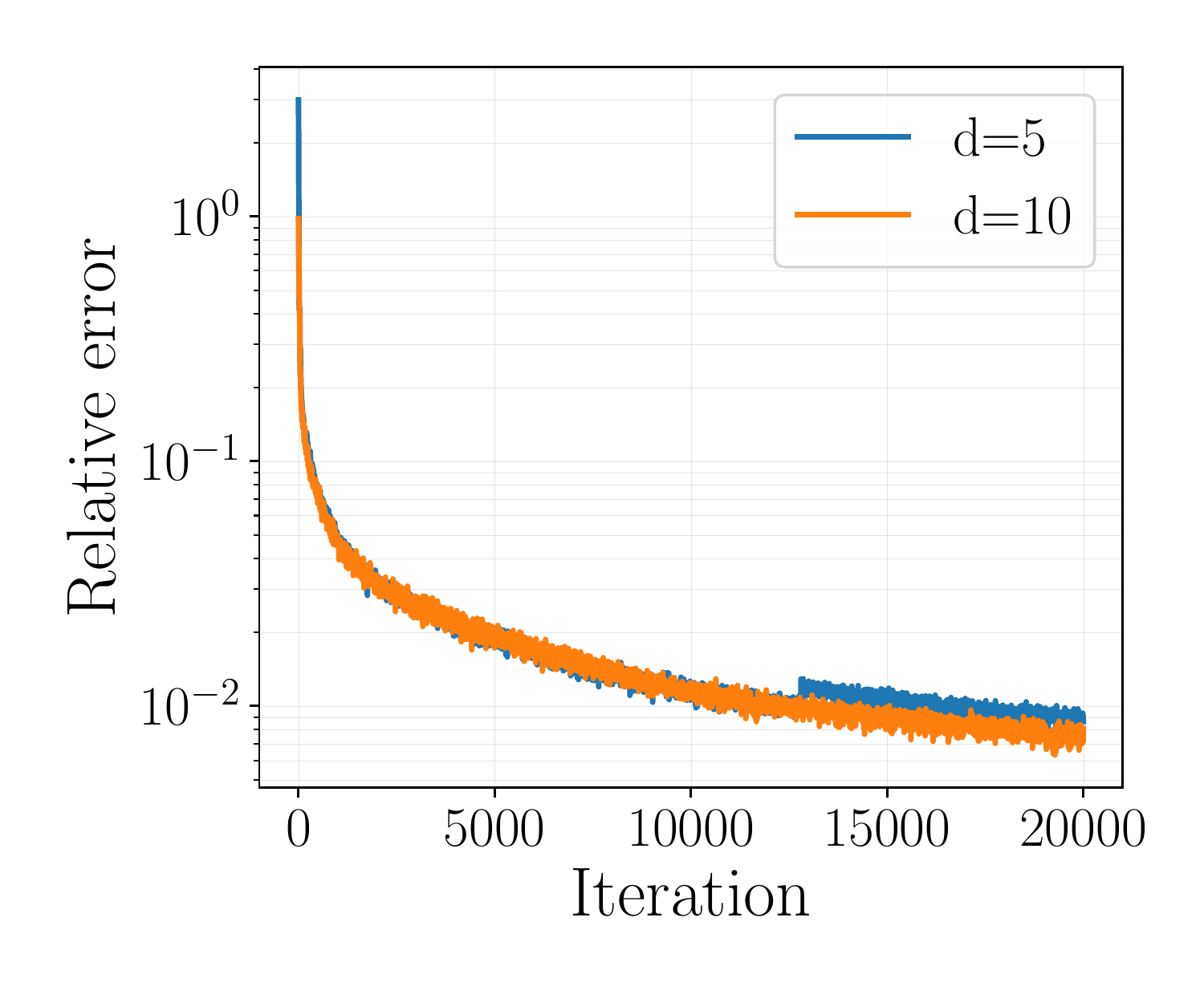} \vspace{-18pt}
\caption{Error vs iteration}
\label{subfig:nonconvex_error_iter}
\end{subfigure}
\caption{Result for the BVP \eqref{eq:poisson_L} with Poisson equation and Dirichlet boundary condition on \textit{nonconvex} domain. (a) True solution $u^*$ and the approximation $u_\theta$ obtained by Algorithm \ref{alg:wan} after 20,000 iterations for $d=10$;
(b) The absolute difference $|u_\theta-u^*|$ for $d=10$;
(c) Relative errors versus iteration numbers for $d=5,10$ cases. For display purpose, images (a) and (b) only show the slices of $x_3=\dots=x_d=0$.}
\label{fig:nonconvex}
\end{figure}

\subsubsection{High-dimensional elliptic PDEs with Neumann boundary condition}
In this experiment, we show that the proposed Algorithm \ref{alg:wan} can be applied to high-dimensional PDEs with \textit{Neumann} boundary condition.
Consider the following boundary value problem:
\begin{equation}\label{eq:neum_cube}
\begin{cases}
-\Delta u+2u= f \quad &\text{in}\ \ \Omega \triangleq (0,1)^d,\\
\partial u/\partial \vec{n} = g \quad &\text{on}\ \ \partial\Omega
\end{cases}
\end{equation}
where $\vec{n}(x) \in \mathbb{R}^d$ is the outer normal at $x\in \partial \Omega$.
We set $f(x)=(\frac{\pi^2}{2}+2)\sin{(\frac{\pi}{2} x_1)}\cos{(\frac{\pi}{2} x_2)}$ in $\Omega$ and
$g(x) = \sbr[1]{\frac{\pi}{2}\cos{\del[1]{\frac{\pi}{2}x_1}}\cos{\del[1]{\frac{\pi}{2}x_2}},\ -\frac{\pi}{2}\sin{\del[1]{\frac{\pi}{2}x_1}}\sin{\del[1]{\frac{\pi}{2}x_2}},\ 0,\cdots,0 }\cdot\vec{n}$ on $\partial \Omega$.
The exact solution of \eqref{eq:neum_cube} in this case is $u^*(x)=\sin{(\frac{x_1}{2}x_1)}\cos{(\frac{\pi}{2}x_2)}$ in $\Omega$ as shown in the left panel of Figure \ref{subfig:neumann_u}.
In this test, we set $K_\varphi=2$, $K_u=5$, $\tau_{\eta}=0.05$, $\tau_\theta=0.02$, $N_r=8\times 10^4$, $N_b=2d\times400$, and $\alpha=1000\times N_b$.
After 20,000 iterations, the relative error of $u_\theta$ to $u^*$ is $2.03\%$ and $1.31\%$ for $d=5$ and $10$ respectively.
The solution is shown in the right panel of Figure \ref{subfig:neumann_u}, whose pointwise absolute error is shown in Figure \ref{subfig:neumann_abserr}.
The progresses of relative error versus iteration number for $d=5,10$ are shown in Figure \ref{subfig:neumann_error_iter}.
This test shows that the proposed algorithm can also work effectively for BVPs with Neumann boundary conditions.

%%%%%%%%%%%%%%%%%%%%%%%%%%%%%%%%%%%%%%%%%%%%%%%%%%%%%%%%%%%%%%%%%%%%%%%%%%%%%%%%%
\subsubsection{Poisson equation on irregular nonconvex domain}
We now consider a BVP with Poisson equation and Dirichlet boundary condition on irregular \textit{nonconvex} domain for problem dimension $d=5,10$ as follows:
\begin{equation}
\label{eq:poisson_L}
\begin{cases}
-\nabla\cdot (a(x)\nabla u )=f(x) \quad & \text{in}\ \ \Omega \triangleq (-1,1)^d \setminus [0,1)^d \\
u(x)= g(x) \quad & \text{on}\ \ \partial\Omega
\end{cases}
\end{equation}
where $a(x)=1+|x|^2$ and $f(x)= \frac{\pi^2}{2}\cdot(1+|x|^2)\sin(\tilde{x}_1)\cos(\tilde{x}_2) +\pi x_2\sin(\tilde{x}_1)\sin(\tilde{x}_2) -\pi x_1\cos(\tilde{x}_1)\cos(\tilde{x}_2)$ in $\Omega$ and $g(x)=\sin(\tilde{x}_1)\cos(\tilde{x}_2)$ on $\partial \Omega$, with $\tilde{x}_i\triangleq (\pi/2)\cdot x_i$ for $i=1,2$.
The true solution is $u^* = \sin(\tilde{x}_1)\cos(\tilde{x}_2)$ in $\Omega$ as shown in the left panel of Figure \ref{subfig:nonconvex_u}.
In this test, $K_{\varphi},K_u,\tau_{\eta},\tau_\theta,N_r,N_b$ are set the same as those in problem (\ref{eq:nonl_cube}), and $\alpha=10,000\times N_b$ and $20,000\times N_b$ for $d=5,10$ respectively.
The solution $u_\theta$ after 20,000 iterations for $d=10$ case is shown in the right panel of Figure \ref{subfig:nonconvex_u}, and the absolute pointwise error $|u_\theta-u^*|$ is shown in Figure \ref{subfig:nonconvex_abserr}.
Again, for both values of problem dimension $d$, we show the progresses of the relative error versus iteration in Figure \ref{subfig:nonconvex_error_iter}.
After 20,000 iterations, the relative error reaches $0.86\%$ and $0.80\%$ for $d=5,10$ cases, respectively.
As we can see, the proposed Algorithm \ref{alg:wan} can easily handle PDEs defined on irregular domains.

\subsubsection{Solving high dimensional parabolic equation involving time}
Next, we consider solving the following \textit{nonlinear} diffusion-reaction equation involving time:
\begin{equation}\label{eq:exp-parabolic}
\begin{cases}
u_t-\Delta u-u^2=f(x,t), &\quad \text{in}\ \Omega\times[0,T] \\
u(x,t)= g(x,t), &\quad \text{on}\  \partial\Omega\times[0,T] \\
u(x,0)= h(x), &\quad \text{in} \ \Omega
\end{cases}
\end{equation}
where $\Omega=(-1,1)^d \subset \mathbb{R}^{d}$. We first give an example of solving the IBVP \eqref{eq:exp-parabolic} in dimension $d=5$ using Algorithm \ref{alg:wan_parab_slice} which discretizes time and uses the Crank-Nicolson scheme \eqref{eq:slice-parab}.
In this test, we set
%\begin{equation*}
$f(x,t)= (\pi^2-2)\sin{({\frac{\pi}{2} x_1})}\cos({\frac{\pi}{2} x_2})e^{-t}-4\sin^2{({\frac{\pi}{2} x_1})}\cos({\frac{\pi}{2} x_2})e^{-2t}$ in $\Omega\times[0,T]$,
%\end{equation*}
$g(x,t)=2\sin({\frac{\pi}{2} x_1})\cos({\frac{\pi}{2} x_2})e^{-t}$ on $\partial\Omega\times[0,T]$ and $h(x)=2\sin({\frac{\pi}{2} x_1})\cos({\frac{\pi}{2} x_2})$ in $\Omega$. In this case, the exact solution of the IBVP \eqref{eq:exp-parabolic} is
$ u(x,t)=2\sin({\frac{\pi}{2} x_1})\cos({\frac{\pi}{2} x_2})e^{-t}$.
We take $T=1$ and discretize the time interval $[0,1]$ into $N=10$ equal segments, and then solve the IBVP using Algorithm \ref{alg:wan_parab_slice} with the setup of $K_{\varphi}, K_u, \tau_{\eta}, \tau_{\theta}, N_r, N_b, \alpha$ at each time step are the same as those in Section \ref{subsubsec:nonlinear}.
Figure \ref{subfig:parabolic_u} shows the exact solution $u^*$ (left) and the solution $u_\theta$ (right) obtained by Algorithm \ref{alg:wan_parab_slice} at final time $T$.
Figure \ref{subfig:parabolic_abserr} shows the point-wise absolute error $|u_\theta(x,T)-u^*(x,T)|$.
%
%We can see that the approximation is very close to the exact solution at final time $T$. In fact, t
The relative error reaches $2.8\%$ after $10,000$ iterations. %, which shows the high accuracy of this method.
The small error implies that the solution obtained by Algorithm \ref{alg:wan_parab_slice} is a close approximation to the true solution $u^*$.
\begin{figure}[!h]
\centering
\begin{subfigure}[b]{0.5\textwidth}
\centering
\includegraphics[width=\textwidth]{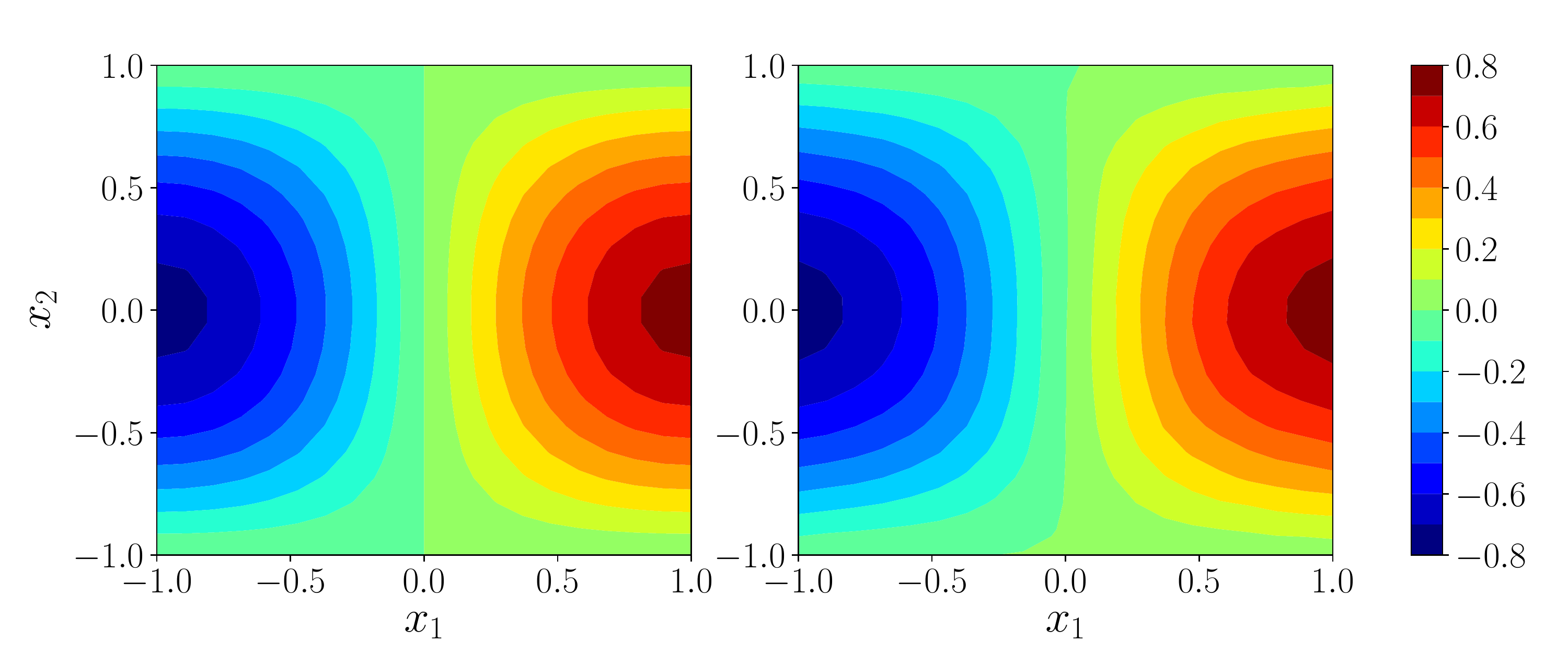}
\caption{True $u^*(x,T)$ (left) vs estimated $u_\theta(x,T)$ (right)}
\label{subfig:parabolic_u}
\end{subfigure}
\begin{subfigure}[b]{0.28\textwidth}
\centering
\includegraphics[width=\textwidth]{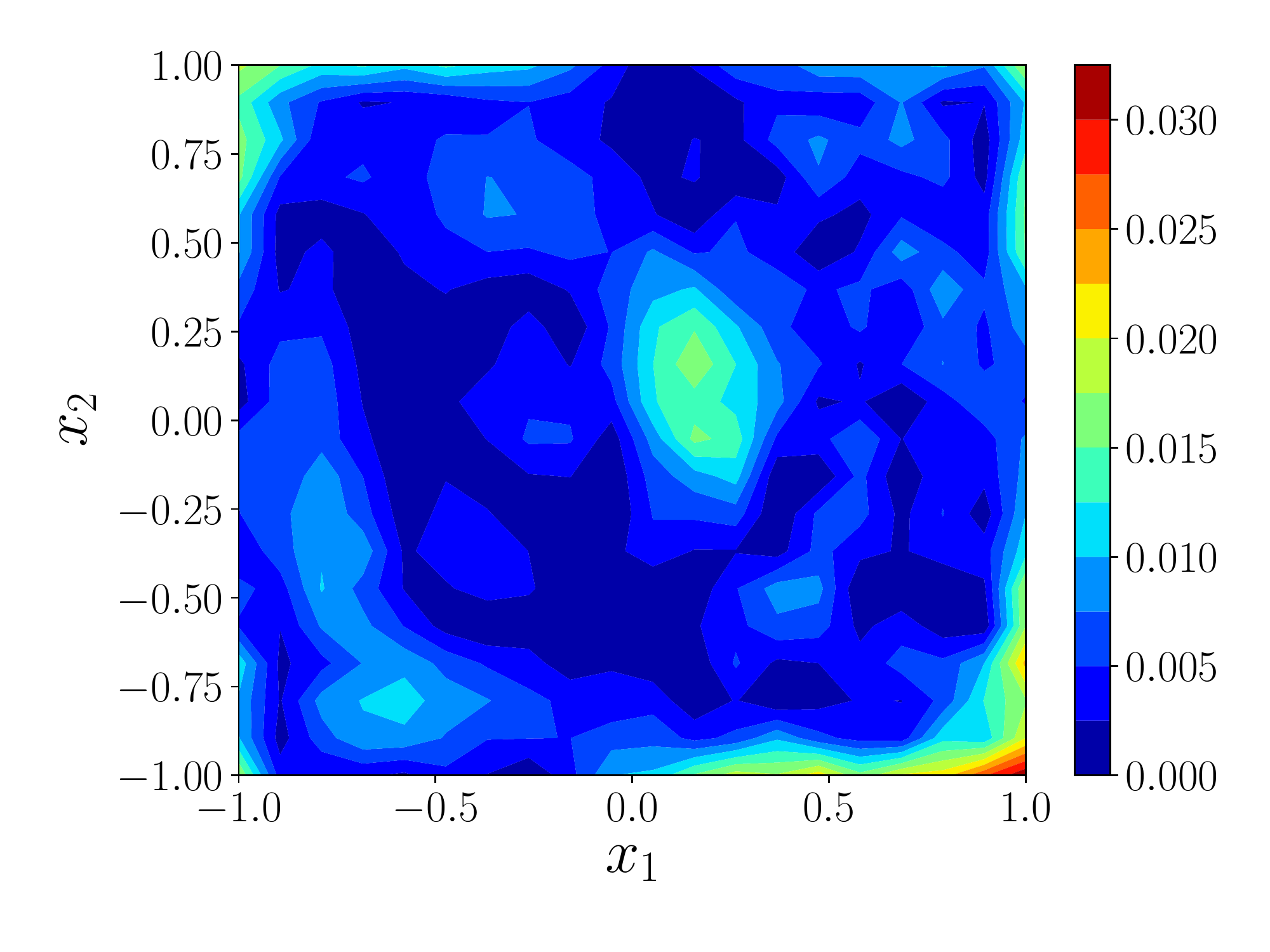} \vspace{-15pt}
\caption{$|u_{\theta}(x,T) - u^*(x,T)|$}
\label{subfig:parabolic_abserr}
\end{subfigure}
\caption{(a) Exact solution $u^*(x,T)$ (left) and the approximation $u_\theta(x,T)$ (right) obtained by Algorithm \ref{alg:wan_parab_slice} at final time $T=1$ for the IBVP \eqref{eq:exp-parabolic} with problem dimension $d=5$. (b) Pointwise absolute error $|u(x,T)-u^*(x,T)|$. All images only show the 2D slice of $x_3=x_4=x_5=0$ for display purpose.}
\label{fig:parabolic}
\end{figure}

We also considered solving the diffusion-reaction equation \eqref{eq:exp-parabolic} for space dimension $d=5,10$ using Algorithm \ref{alg:wan_parab} by dealing with $(x,t)$ jointly without discretization. In this experiment, we set $f(x,t)= (\frac{\pi^2}{2}-2)\sin{(\frac{\pi}{2}x_1)}e^{-t}-4\sin^2\del[1]{\frac{\pi}{2}x_1}e^{-2t}$ in $\Omega\times [0,T]$, $g(x,t)=2\sin({\frac{\pi}{2} x_1})e^{-t}$ on $\partial\Omega\times[0,T]$, and $h(x)=2\sin({\frac{\pi}{2} x_1})$ in $\Omega$.
The exact solution is $u^*(x,t) = 2\sin({\frac{\pi}{2} x_1})e^{-t}$ in $\Omega \times [0,T]$.
In this test, we set $K_{\varphi},K_u,\tau_{\eta},\tau_\theta,N_r,N_b,\alpha$ the same as in Section \ref{subsubsec:nonlinear} , and $N_a=N_b$ and $\gamma=\alpha$.
The solution $u_\theta$ after 20,000 iterations for $d=10$ case is shown in the right panel of Figure \ref{subfig:parab_u}, and the absolute pointwise error $|u_\theta-u^*|$ is shown in Figure \ref{subfig:parab_abserr}.
As has done before, we show the progresses of the relative error versus iteration in Figure \ref{subfig:parab_error_iter}.
After 20,000 iterations, the relative error reaches $0.78\%$ and $0.66\%$ for $d=5,10$ cases, respectively.
Clearly, the Algorithm \ref{alg:wan_parab} can solve high-dimensional nonlinear PDEs involving time accurately.
\begin{figure}
\centering
\begin{subfigure}[b]{.46\textwidth}
\includegraphics[width=\textwidth]{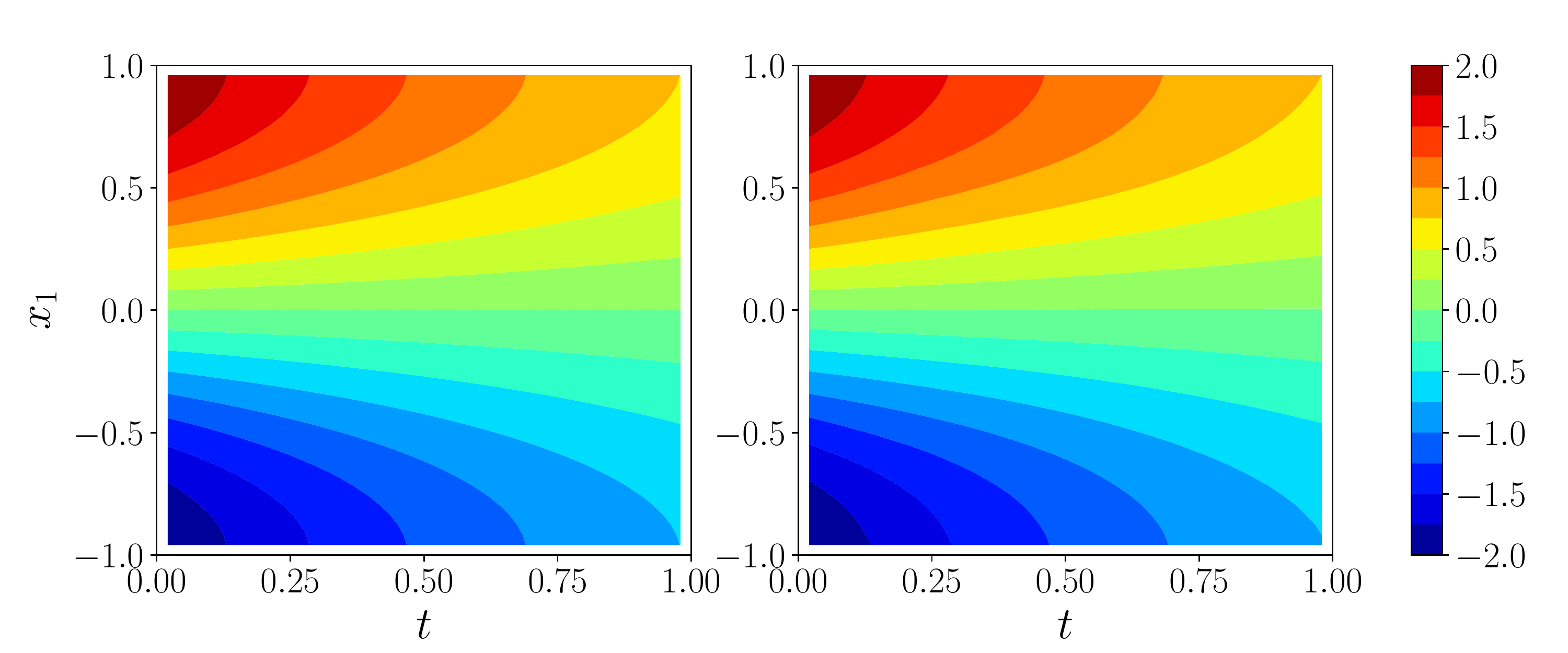}
\caption{True $u^*(x,t)$ (left) vs estimated $u_\theta(x,t)$ (right)} \label{subfig:parab_u}
\end{subfigure}
\begin{subfigure}[b]{.245\textwidth}
\includegraphics[width=\textwidth]{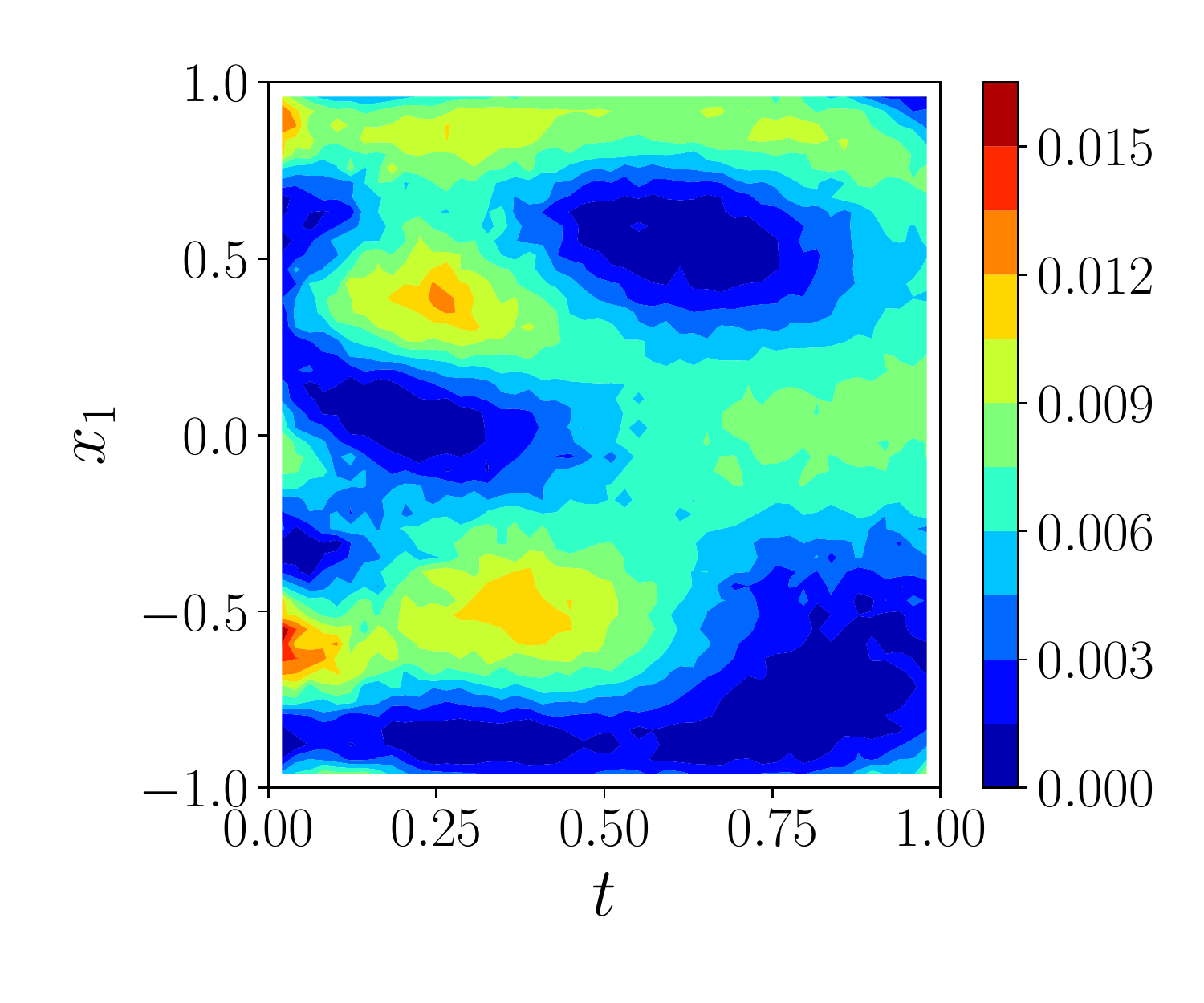} \vspace{-20pt}
\caption{$|u_\theta(x,t)-u^*(x,t)|$} \label{subfig:parab_abserr}
\end{subfigure}
\begin{subfigure}[b]{.23\textwidth}
\includegraphics[width=\textwidth]{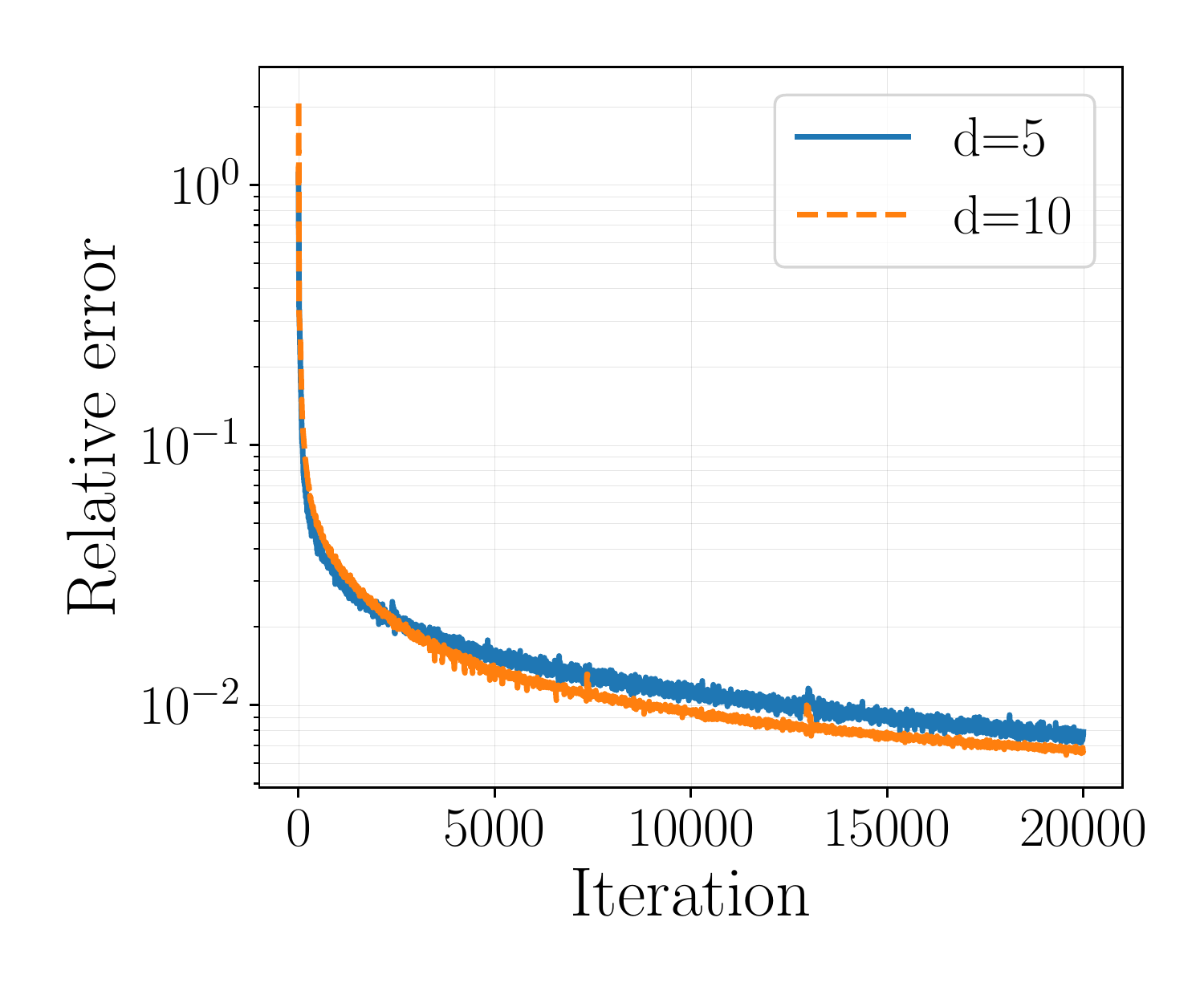} \vspace{-18pt}
\caption{Error vs iteration} \label{subfig:parab_error_iter}
\end{subfigure}
\caption{Result for the IBVP \eqref{eq:exp-parabolic} with a \textit{nonlinear} diffusion-reaction equation. (a) True solution $u^*$ and the approximation $u_\theta$ obtained by Algorithm \ref{alg:wan_parab} after 20,000 iterations for $d=5$;
(b) The absolute difference $|u_\theta-u^*|$ for $d=5$;
(c) Relative errors versus iteration numbers for $d=5,10$ cases. For display purpose, images (a) and (b) show the slices of $x_2=\dots=x_d=0$ for $d\ge 2$.}
\label{fig:parab}
\end{figure}

\subsubsection{Stability and scalability}
In this last set of tests, we evaluate the performance of Algorithm \ref{alg:wan} with different parameter settings and problem dimensionalities.
First, we test the effects of different numbers of collocation points $N_r$ and $N_b$ on the problem \eqref{eq:nonl_cube} with $d=5$.
We set $K_\varphi=1$, $K_u=2$, $\tau_\eta=0.04$ and $\tau_\theta=0.015$ for this test, and then use different numbers of $N_r$ in the region $\Omega$ and $N_b$ on the boundary $\partial\Omega$ (each time with one fixed and the other one varying).
In particular, we run Algorithm \ref{alg:wan} and show the relative error versus running time (the iteration stops when the $L_2$ relative error reaches $1\%$) for varying $N_b$ with fixed $N_r=500$ (Figure \ref{subfig:N_Nr500}) and $N_r=16,000$ (Figure \ref{subfig:N_Nr16000}), and then for varying $N_r$ with fixed $N_b=5$ (Figure \ref{subfig:N_Nb5}) and $N_b=10$ (Figure \ref{subfig:N_Nb20}).
As we can see, in all cases, Algorithm \ref{alg:wan} stably makes progresses towards the weak solution.
Note that more sampled points (larger $N_r$ or $N_b$) do not always yield improvement as shown in Figure \ref{fig:N}.
We suspect that it is due to the severe non-convexity of the problem (from the PDE formulation, the boundary condition, and the deep neural network parameterization) which contains many local minima.
In this case, more sampling points would generally lower the stochastic error of our gradient evaluations, but also reduce the chance for the iterates to escape from local minima.
In terms of real-world performance, more sampled points may improve convergence rate in terms of iterations, but also increase per-iteration computational cost due to more computations of backpropagations.

We also test the effect of network architectures of $u_\theta$.
We try a number of different combinations of layer and neuron numbers for $u_\theta$.
With fixed layer numbers $3$ and $9$, we show the progresses of training with different number of per-layer neurons in Figures \ref{subfig:struct_K3} and \ref{subfig:struct_K9} respectively.
Similarly, with fixed per-layer neuron numbers $10$ and $20$, we show the same training process with varying numbers of layers in Figures \ref{subfig:struct_n10} and \ref{subfig:struct_n20} respectively.
In all of these tests, we can see the relative error of $u_\theta$ gradually decays towards $0$ except the case when the number of neuron is $5$, which indicates sufficient neurons are required to accurately approximate the solution.
In most cases, more layers and/or neurons yield faster decay of relative error, but this is not always the case.
We know that more layers and/or neurons increase representation capacity of the neural network $u_\theta$, but they can introduce much more parameters to train, yield longer training time, and may result in overfitting of the representation.
We plan to investigate this problem in more depth in our future work.

To show the scalability of our method, we plotted the total computation time (in seconds) of Algorithm \ref{alg:wan} applied to BVP \eqref{eq:nonl_cube} for $d=5,10,15,20,25$ in Figure \ref{subfig:struct_t_vs_d}.
This figures shows the times when $u_\theta$ first hits $1\%$ relative error to $u^*$ for these problem dimensions.
It appears that, with the parameter setting we selected, the computation time increases approximately linearly in problem dimension $d$.
This shows that Algorithm \ref{alg:wan} has great potential in scalability for high dimensional PDEs empirically.
\begin{figure}[t!]
\begin{subfigure}[b]{.245\textwidth}
\includegraphics[width=\textwidth]{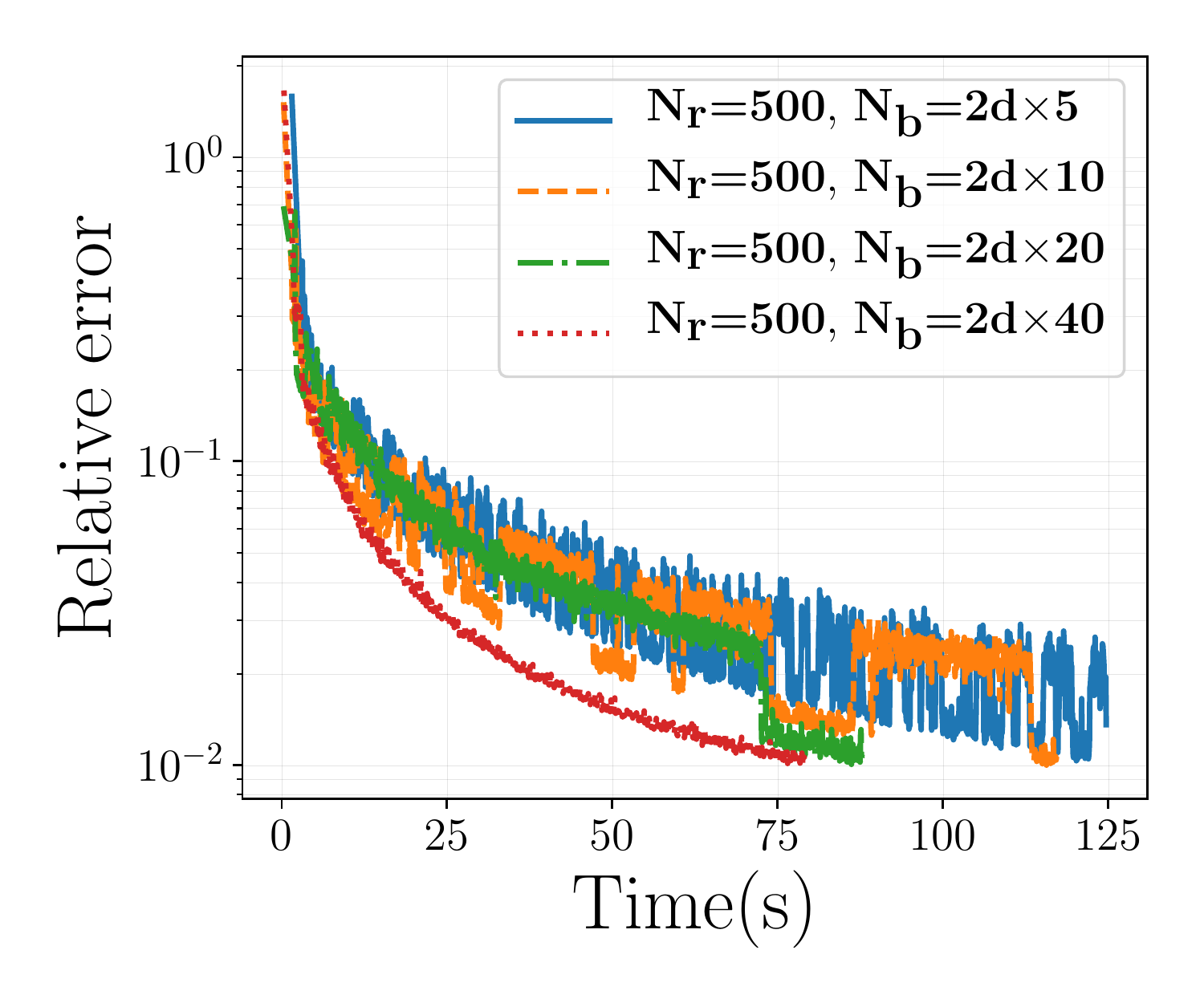} \caption{$N_r=500$} \label{subfig:N_Nr500}
\end{subfigure}
\begin{subfigure}[b]{.245\textwidth}
\includegraphics[width=\textwidth]{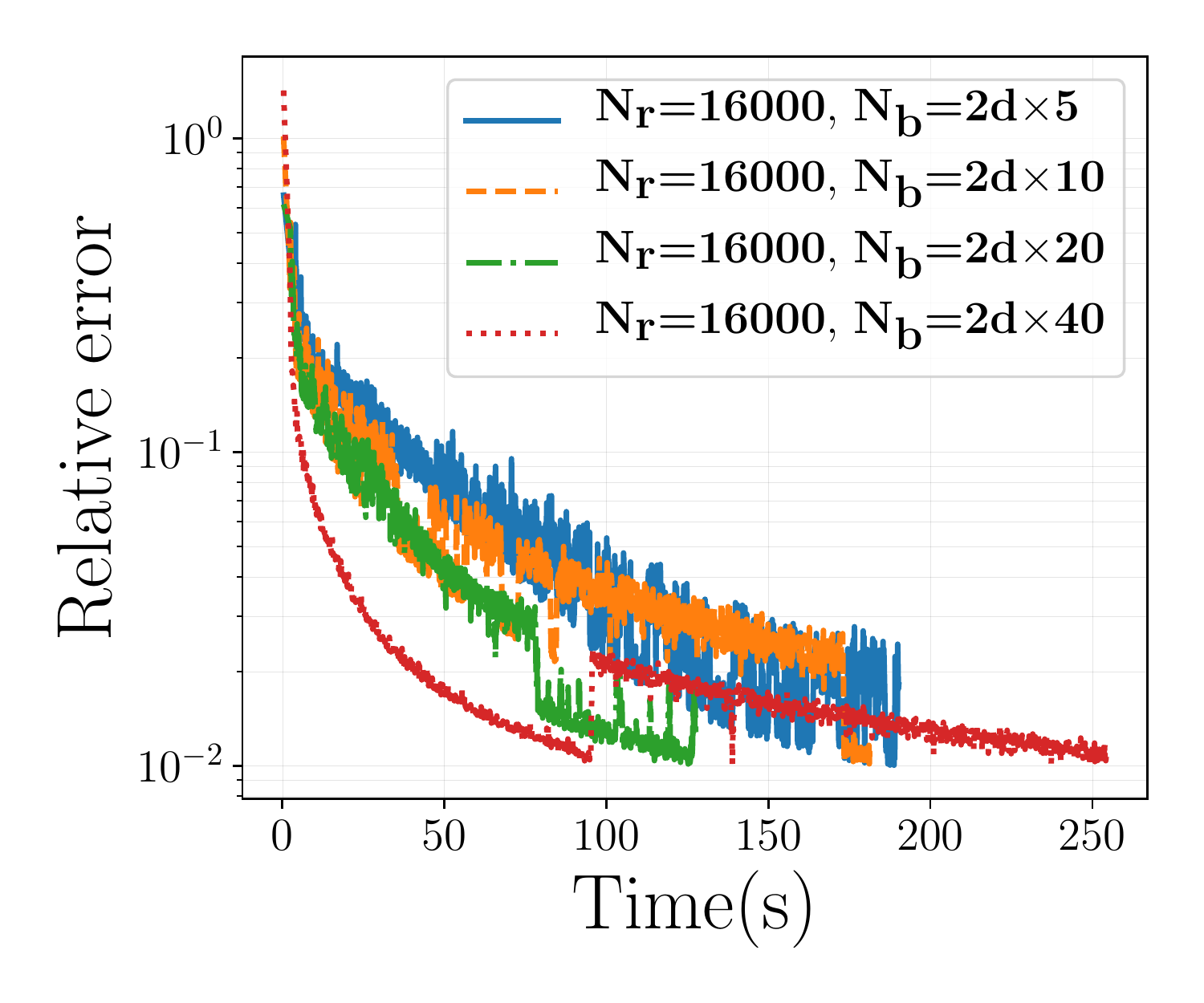} \caption{$N_r=16,000$} \label{subfig:N_Nr16000}
\end{subfigure}
\begin{subfigure}[b]{.245\textwidth}
\includegraphics[width=\textwidth]{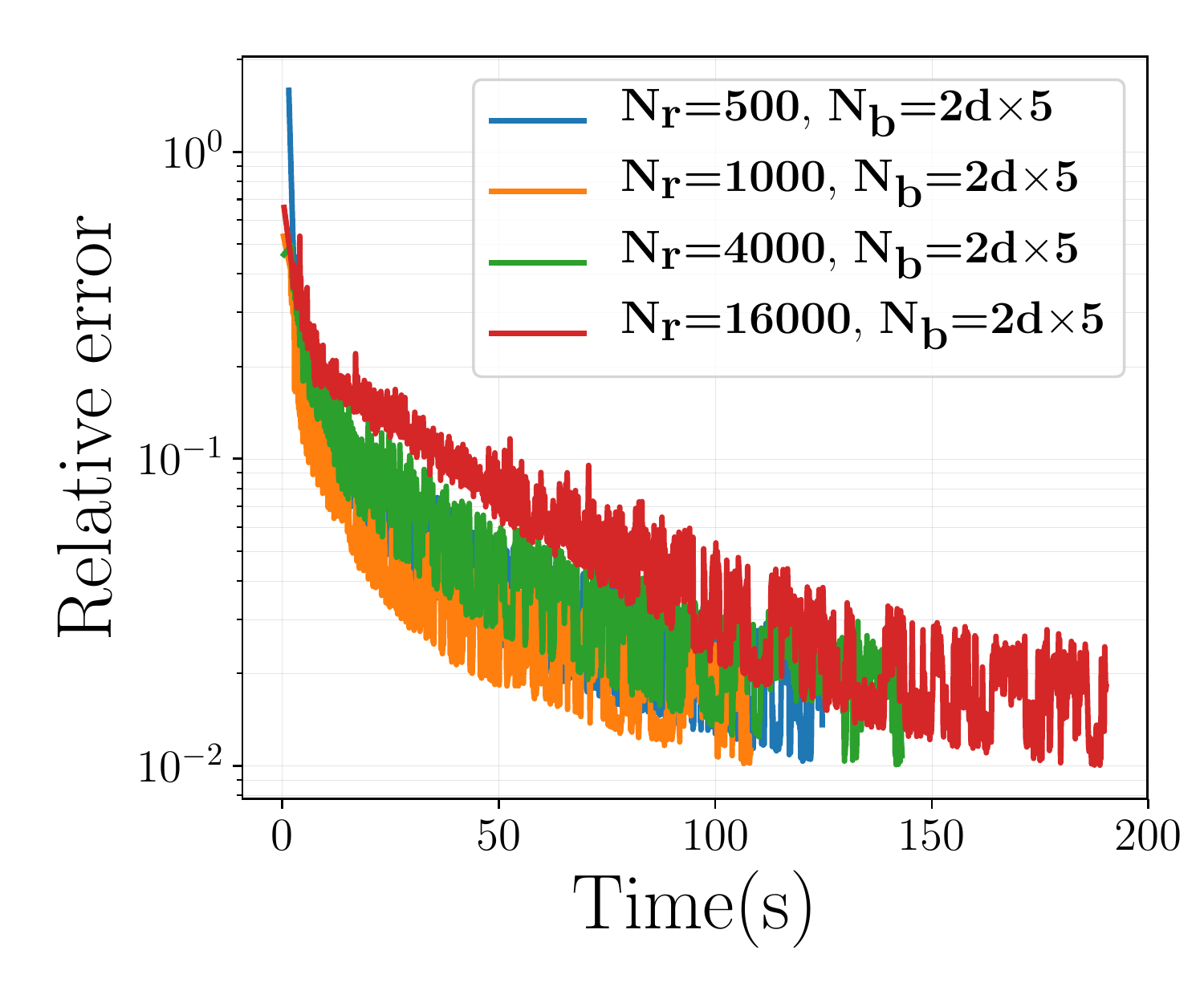} \caption{$N_b=2d\times 5$} \label{subfig:N_Nb5}
\end{subfigure}
\begin{subfigure}[b]{.245\textwidth}
\includegraphics[width=\textwidth]{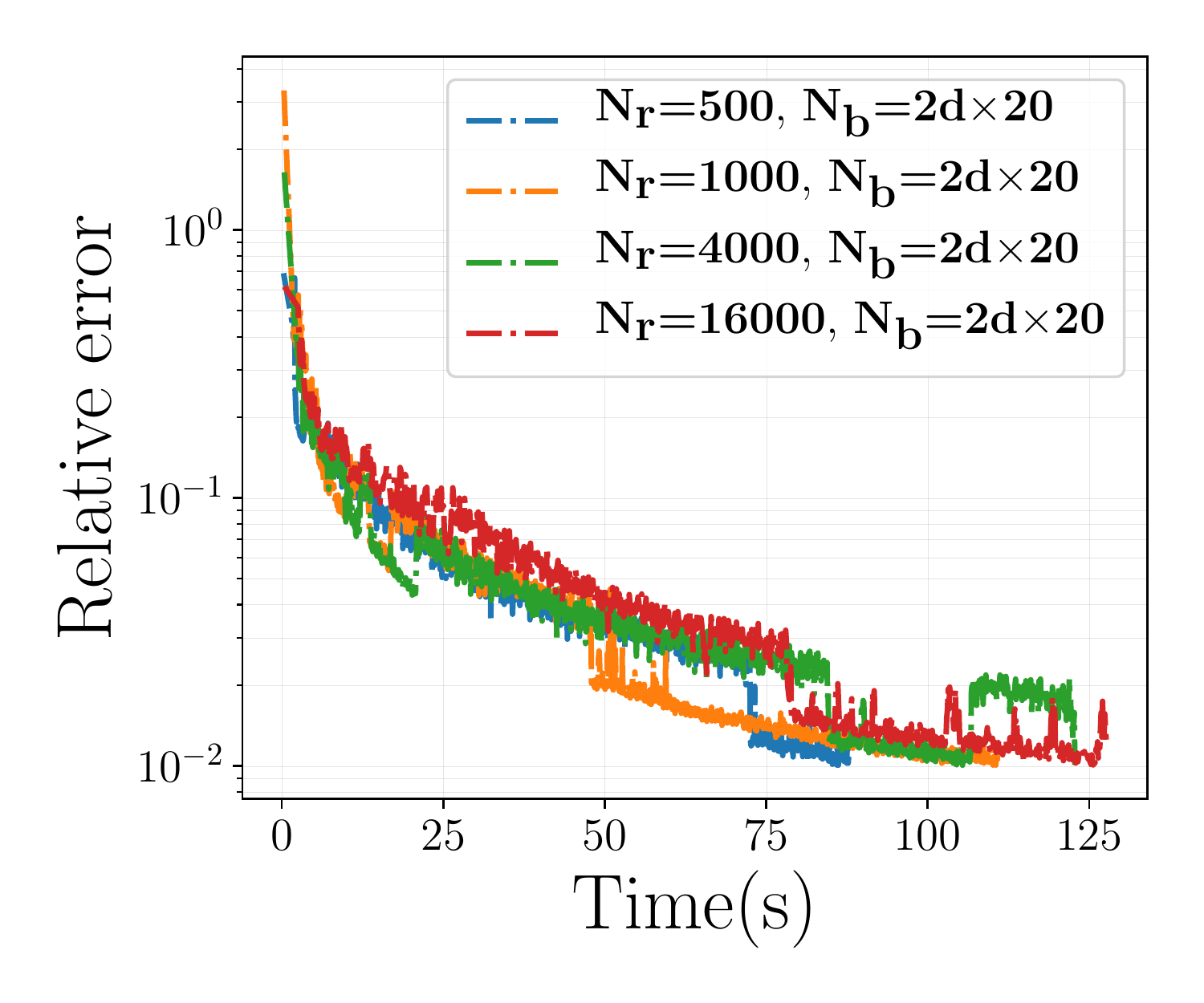} \caption{$N_b=2d\times 20$} \label{subfig:N_Nb20}
\end{subfigure}
\caption{Effects of the numbers of sampled region collocation points $N_r$ and boundary collocation points $N_b$ on the nonlinear elliptical PDE \eqref{eq:nonl_cube} with $d=5$. The progresses of relative error versus running time are shown with varying $N_b$ for (a) $N_r=500$ and (b) $N_r=16,000$, and with varying $N_r$ for (c) $N_b=2d\times 5$ and (d) $N_b=2d\times 20$.}
\label{fig:N}
\begin{subfigure}[b]{.196\textwidth}
\includegraphics[width=\textwidth]{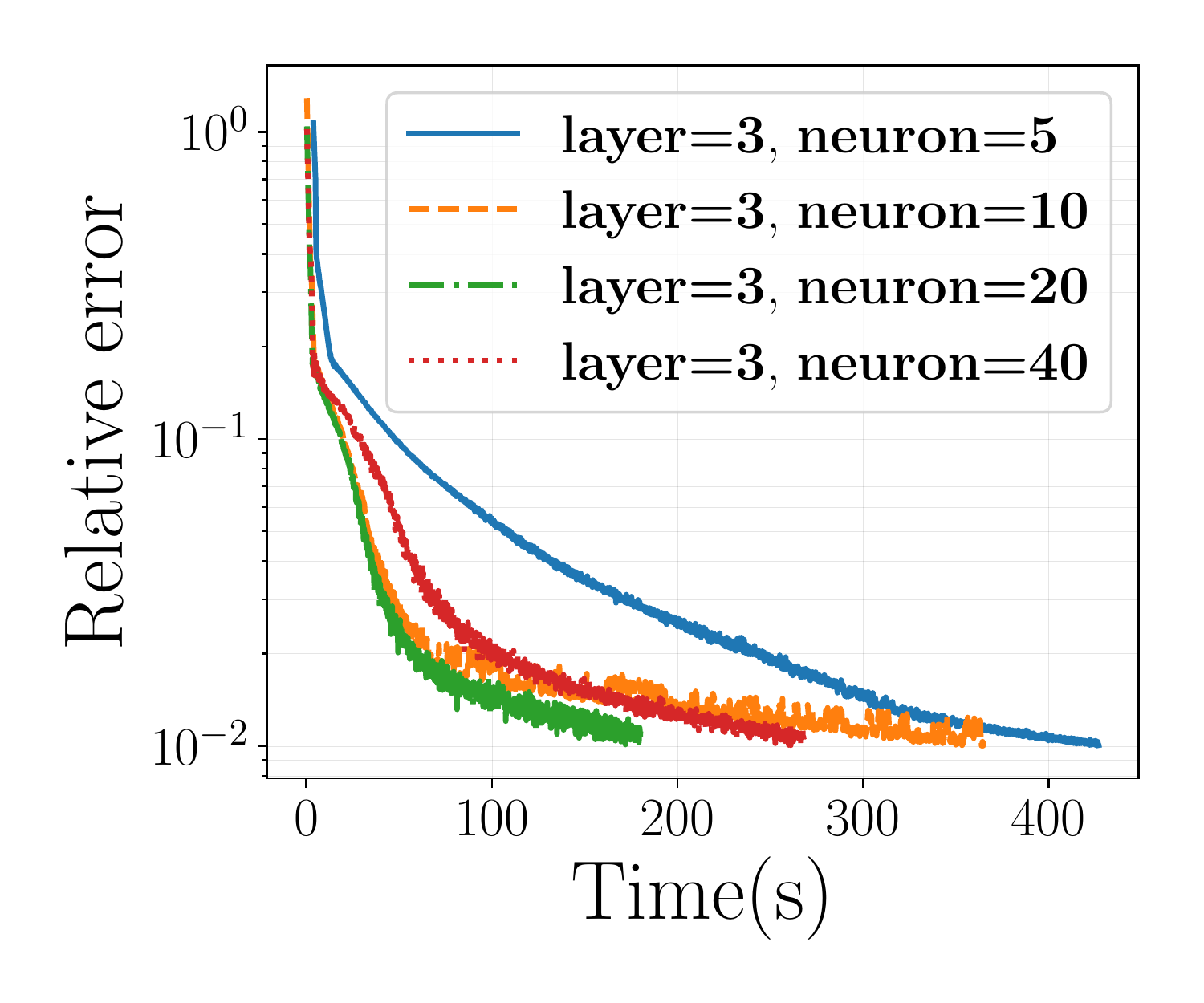} \caption{3 layers} \label{subfig:struct_K3}
\end{subfigure}
\begin{subfigure}[b]{.196\textwidth}
\includegraphics[width=\textwidth]{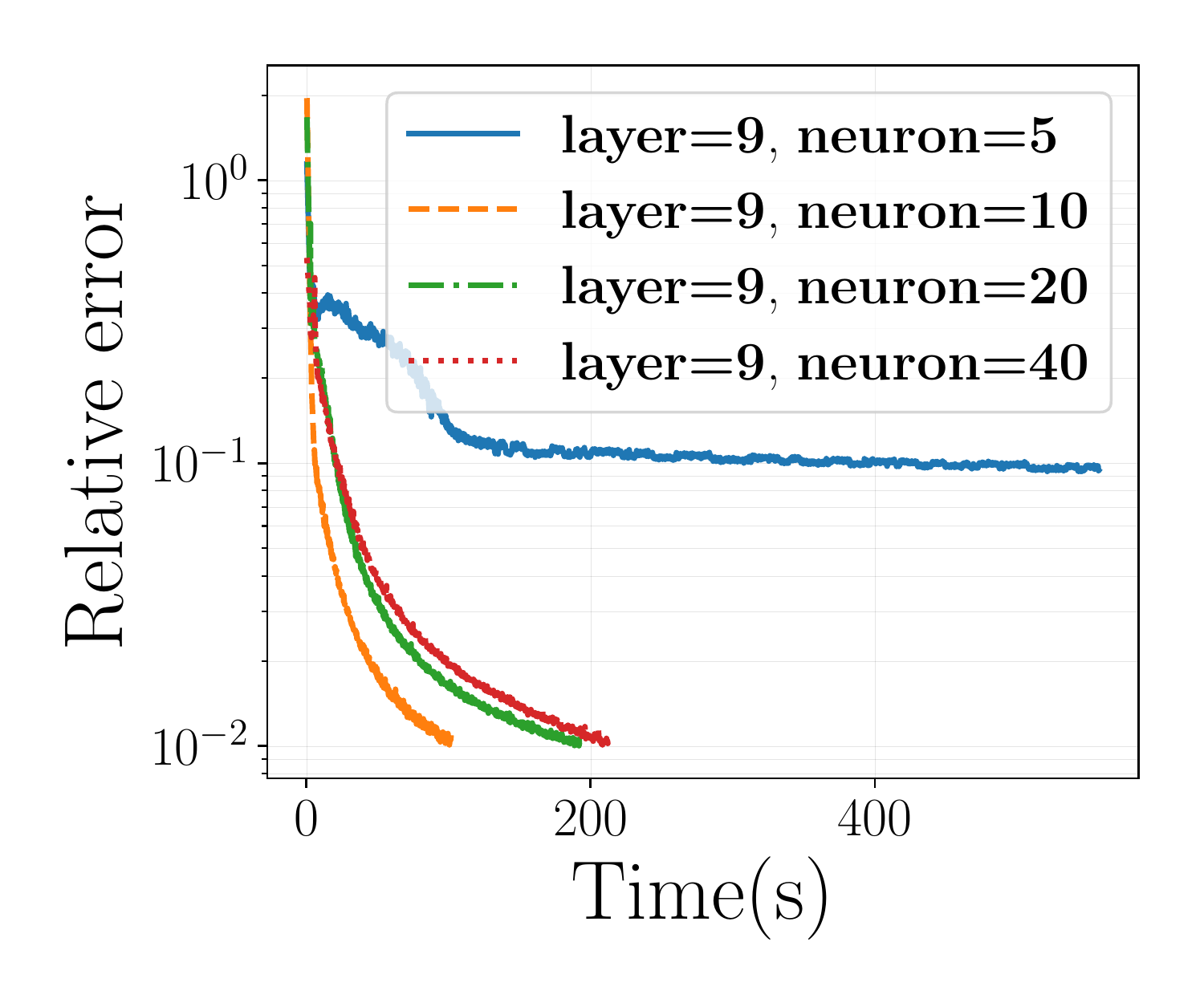} \caption{9 layers} \label{subfig:struct_K9}
\end{subfigure}
\begin{subfigure}[b]{.196\textwidth}
\includegraphics[width=\textwidth]{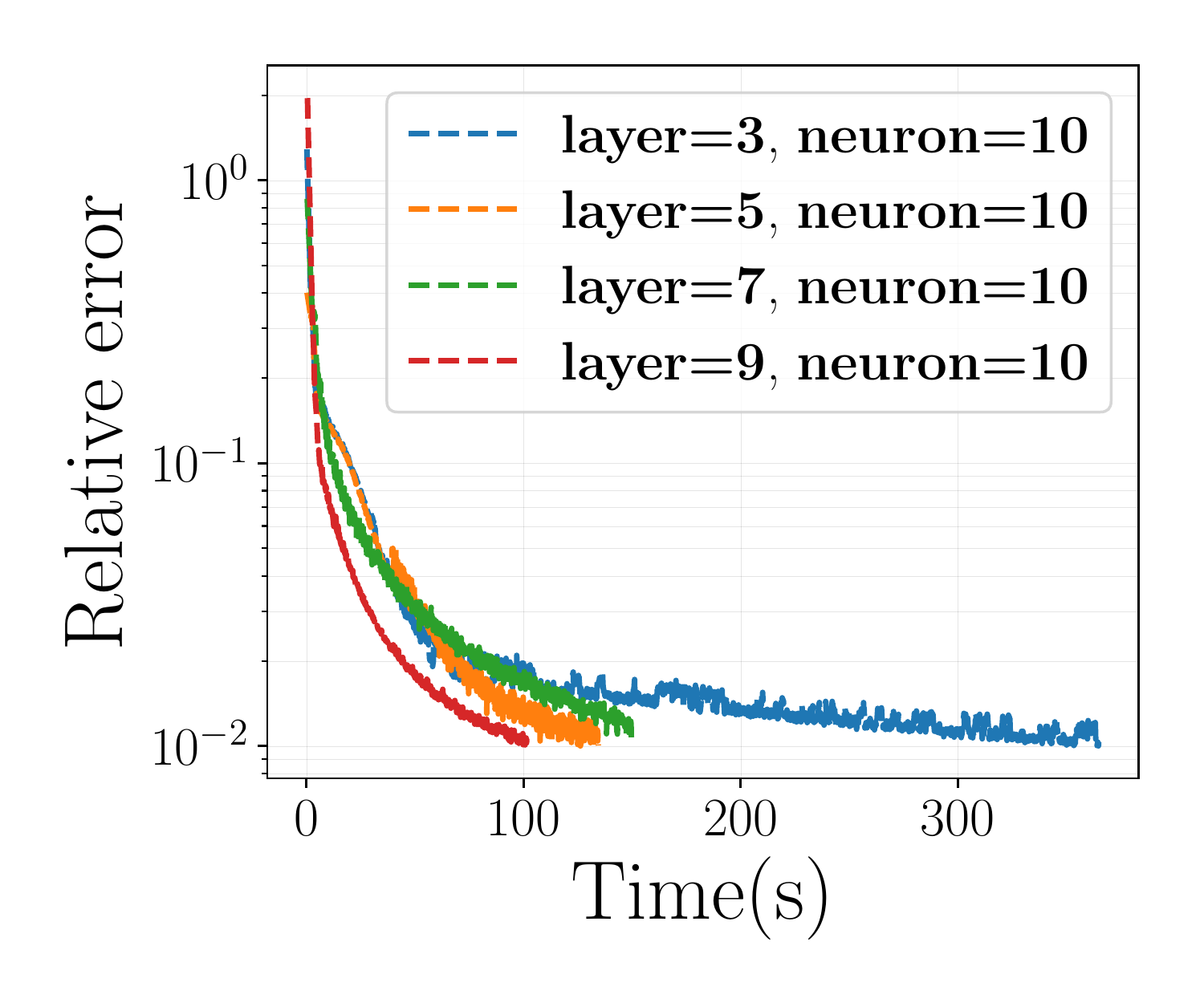} \caption{10 neurons} \label{subfig:struct_n10}
\end{subfigure}
\begin{subfigure}[b]{.196\textwidth}
\includegraphics[width=\textwidth]{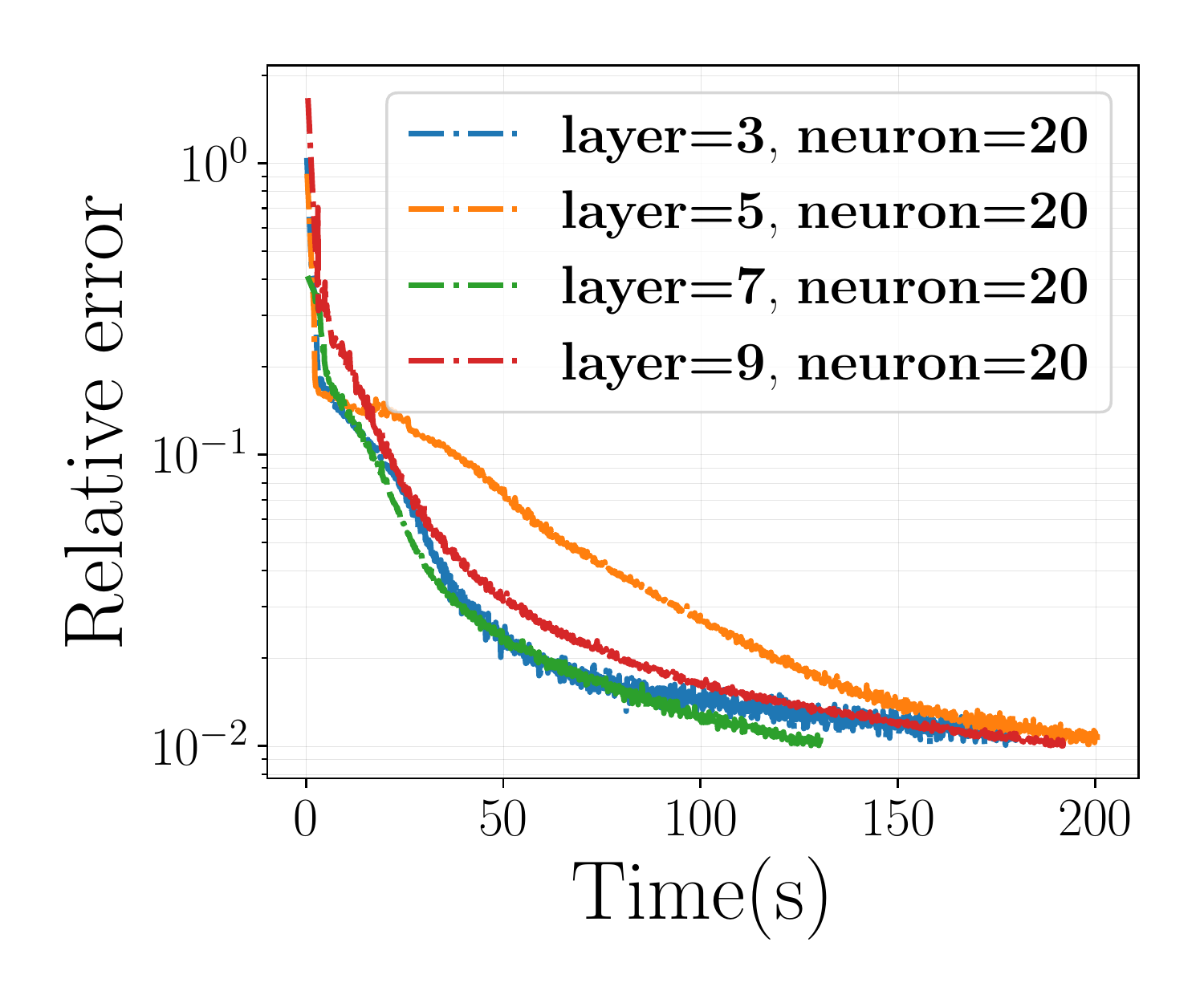} \caption{20 neurons} \label{subfig:struct_n20}
\end{subfigure}
\begin{subfigure}[b]{.196\textwidth}
\includegraphics[width=\textwidth]{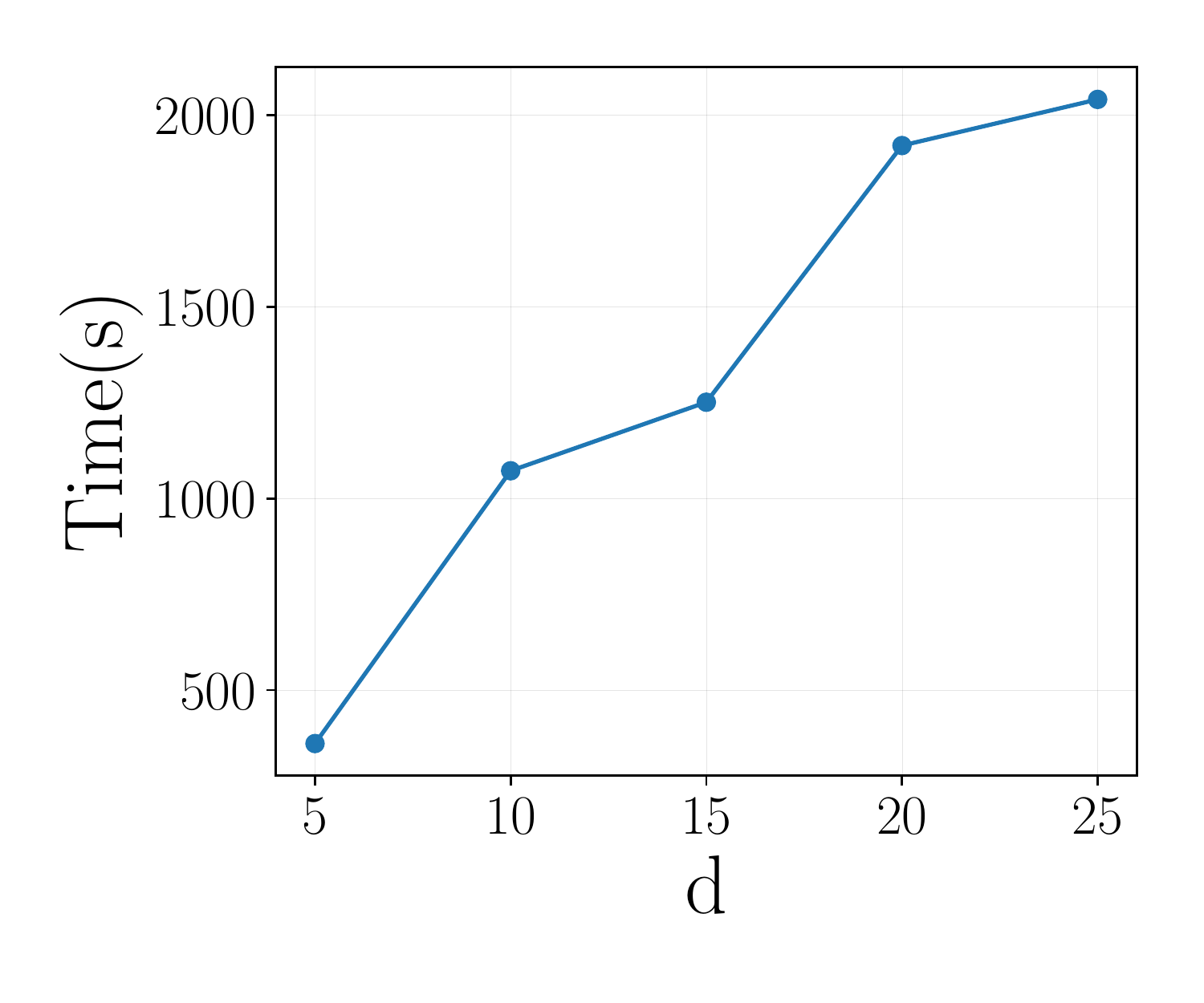} \caption{Time (s) vs $d$} \label{subfig:struct_t_vs_d}
\end{subfigure}
\caption{Effects of the numbers of layers and neurons, as well as the dimensionality, on the nonlinear elliptical PDE problem \eqref{eq:nonl_cube}. The progresses of relative error versus running time are shown with varying number of neurons per layer for a total of (a) 3 layers and (b) 9 layers, and with number of layers for the same number of (c) 10 and (d) 20 neurons per layer; (e) The computation time (in seconds) versus problem dimension $d=5,10,15,20,25$.}
\label{fig:layers}
\end{figure}

%%%%%%%%%%%%%%%%%%%%%%%%%%%%%%%%%%%%%%%%%%%%%%%%%%%%%%%%%%%%%%%%%%%%%%%%%%%%%%%%%%
\section{Concluding Remarks}
\label{sec:summary}

In this paper, we developed a novel approach, called \textit{weak adversarial network} or WAN, to solve general high-dimensional linear and nonlinear PDEs defined on arbitrary domains.
Inspired by the weak formulation of PDEs, we rewrite the problem of finding the weak solution of the PDE as a saddle-point problem, where the weak solution and the test function are parameterized as the primal and adversarial networks, respectively.
The objective function is completely determined by the PDE, the initial and boundary conditions, of the IBVP; and the parameters of these two networks are alternately updated during the training to reach optimum. %as in an unsupervised learning.
%
%The two networks operate pointwisely and hence only require small network architecture.
%
The training only requires evaluations of the networks on randomly sampled collocations points in the interior and boundary of the domain, and hence can be completed quickly on desktop-level machines with standard deep learning configuration.
We demonstrated the promising performance of WAN on a variety of PDEs with high dimension, nonlinearity, and nonconvex domain which are challenging issues in classical numerical PDE methods.
In all tests, WAN exhibits high efficiency and strong stability without suffering these issues.
%

%%%%%%%%%%%%%%%%%%%%%%%%%%%%%%%%%%%%%%%%%%%%%%%%%%%%%%%%%%%%%%%%%%%%%%%%%%%%%%%%%%%%%%%%%%
\section*{Acknowledgement}
YZ would like to acknowledge funding from the China Scholarship Council under grant No.201806320325.
The work of GB was supported in part by a NSFC Innovative Group Fund under grant No.11621101.
The work of XY was supported in part by the National Science Foundation under grants DMS-1620342, CMMI-1745382, DMS-1818886, and DMS-1925263.
The work of HZ was supported in part by National Science Foundation under grants DMS-1620345 and DMS-1830225, and the Office of Naval Research under grant N00014-18-1-2852.
\bibliographystyle{elsarticle-num}
\bibliography{wan_jcp_r1}

\end{document}